\documentclass[12pt]{article}
\usepackage{amssymb}
\usepackage{graphicx}

\usepackage{fancyvrb}

\usepackage[subfigure]{tocloft}
\usepackage{color}
\usepackage{latexsym}
\usepackage{amssymb,amsmath,amstext}
\usepackage{epsfig}
\usepackage{graphics}
\usepackage{subfigure}
\usepackage{euscript}
\usepackage{ifthen}
\usepackage{wrapfig}
\usepackage{amsthm}
\usepackage{multicol}
\usepackage{paralist}
\usepackage[colorlinks=true, urlcolor=blue, linkcolor=blue, citecolor=blue]{hyperref}
\usepackage[margin=1in]{geometry}

\usepackage{multicol}

\numberwithin{equation}{section}
\theoremstyle{plain}
\newtheorem{theorem}{Theorem}
\numberwithin{theorem}{section}
\newtheorem{proposition}[theorem]{Proposition}
\newtheorem{example}[theorem]{Example}
\newtheorem{lemma}[theorem]{Lemma}
\newtheorem{corollary}[theorem]{Corollary}

\newtheorem{remark}[theorem]{Remark}
\newtheorem{conjecture}[theorem]{Conjecture}

\newcommand{\C}{\mathbb{C}}
\newcommand{\Z}{\mathbb{Z}}
\newcommand{\PP}{\mathbb{P}}
\newcommand{\F}{\mathbb{F}}
\newcommand{\R}{\mathbb{R}}
\newcommand{\arxiv}[1]{\href{http://arxiv.org/abs/#1}{{\tt arXiv:#1}}}

\allowdisplaybreaks

\date{}
\begin{document}

\title{\bf The Universal Kummer Threefold}

\author{Qingchun Ren \and Steven V Sam \and Gus Schrader \and Bernd Sturmfels}

\maketitle

 \begin{abstract}
 \noindent The universal Kummer threefold is a $9$-dimensional variety that represents the total space of the 6-dimensional family of Kummer threefolds in $\PP^7$. We compute defining polynomials for three versions of this family, over the Satake hypersurface, over the G\"opel variety, and over the reflection representation of type $\mathrm{E}_7$.  We develop  classical themes such as theta functions and Coble's quartic hypersurface using current tools from combinatorics, geometry, and commutative algebra.
Symbolic and numerical computations for genus $3$ moduli spaces appear alongside toric and tropical methods.
   \end{abstract}

\tableofcontents

\section{Introduction}

Kummer varieties are quotients of abelian varieties
by their involution $x \mapsto -x$. Each $g$-dimensional
Kummer variety has a natural embedding into $\mathbb{P}^{2^g-1}$
by second order theta functions.
The moduli space is recorded in a second copy
of $\mathbb{P}^{2^g-1}$ by the corresponding theta constants. By taking the closure, this construction defines
the {\em universal Kummer variety} $\mathcal{K}_g$.
This is an irreducible projective variety  of  dimension $\binom{g+1}{2} + g$,
defined over the integers, and naturally embedded in
$\mathbb{P}^{2^g-1} \times \mathbb{P}^{2^g-1}$.
The image of the projection of $\mathcal{K}_g$ onto the $\PP^{2^g-1}$ of theta constants
is the Satake compactification of the
$\binom{g+1}{2}$-dimensional moduli space $\mathcal{A}_g(2,4)$.
The various Kummer varieties of dimension $g$ appear as fibers of this map.

Our object of interest is the {\em universal Kummer ideal} $\,\mathcal{I}_g$.
This is the bihomogeneous prime ideal 
in a polynomial ring in $2^{g+1}$ unknowns that defines
$\mathcal{K}_g$ as a subvariety of  $\mathbb{P}^{2^g-1} \times \mathbb{P}^{2^g-1}$.
What motivated this project was our desire to understand  the  ideal $\mathcal{I}_3$
of the universal Kummer threefold in $\mathbb{P}^7 \times \mathbb{P}^7$.
Grushevsky and Salvati Manni write in \cite[\S 6]{GS} that the ideal
$\mathcal{I}_3$ ``is known quite explicitly'', and we were wondering how
to communicate its generators to {\tt Macaulay 2}  \cite{m2}.
Our current state of knowledge about this question is
presented in Section~\ref{sec:kumeqn} of this paper. The example of $\mathcal{I}_2$ is
worked out in Example~\ref{ex:Kummsur} below.

In our attempts to understand $\mathcal{I}_3$, we also studied a variant of $\mathcal{K}_3$, 
over a base $\mathcal{G}$ that is a moduli space for plane quartics. We discovered a beautiful mathematical story that connects classical algebraic geometry topics, such as Coble's quartic hypersurface, G\"opel functions, and the type $\mathrm{E}_7$ reflection arrangement, with more modern topics, such as toric geometry, tropical methods, and numerical computation. Our study of this variation was originally inspired by Vinberg's theory of $\theta$-representations \cite{vinberg}, but can be defined without any reference to this theory. However, as we shall see in Section~\ref{sec:kumeqn}, this point of view gives the construction of a large number of equations that we cannot see how to construct otherwise.

Moduli of plane quartics have been studied extensively, notably in \cite{Cob} and \cite[\S IX]{DO}.
A key feature of Kummer varieties of the Jacobians of plane quartics is this:
in their embedding in $\PP^7$, there is a unique quartic hypersurface, called the {\it Coble quartic}, whose singular locus is the Kummer variety. We are interested in the moduli space of Coble quartics.
This moduli space can be parametrized by the G\"opel functions mentioned in \cite[\S IX]{DO}. These functions embed it as a subvariety $\mathcal{G}$
of $\PP^{134}$, so we call it the {\it G\"opel variety}. In fact, it sits in a linear $\PP^{14} \subset \PP^{134}$, and it can be alternatively parametrized by a Macdonald representation of the Weyl group of type $\mathrm{E}_7$. 
Hence the rich combinatorics of reflection arrangements is embodied in $\mathcal{G}$. 
Colombo, van Geemen and Looijenga   \cite{CGL} studied the G\"opel variety from this perspective, and realized it as a moduli space of marked del Pezzo surfaces of degree $2$.

It is natural to ask for the prime ideal of $\mathcal{G}$ embedded in $\PP^{14}$.
We solve this problem and calculate its graded Betti numbers. 
A delightful picture emerges in the
 embedding in $\PP^{134}$. Together with trinomial linear equations that define $\PP^{14} $
 in $ \PP^{134}$,    all equations are
    written as {\it binomials}. This shows that there is a larger toric variety $\mathcal{T}$ sitting in $\PP^{134}$ that contains $\mathcal{G}$ as a linear section.  In fact, there are $35$ cubics and $35$ quartics that
     cut out      $\mathcal{G}$. The cubics are genus 3 analogues of the {\em toric Segre cubic relations}
    of    Howard~{\it et al.} in~\cite[(1.2)]{HMSV}.

We obtain a range of results on $\mathcal{T}$ and the corresponding $35$-dimensional polytope $\mathcal{A}$. 
It has $135$ vertices and $63$ distinguished facets, and their incidence relations admit a compact description in terms of the finite symplectic space $(\mathbb{F}_2)^6$. Furthermore, it exhibits symmetry under the Weyl group of type $\mathrm{E}_7$. It should be very fruitful to further study the combinatorics of the polytope $\mathcal{A}$, such as its face lattice and Ehrhart polynomial. 
 Even more important, this also paves the way for studying the {\it tropicalization} of $\mathcal{G}$, which we hope will provide a useful model for genus 3 curves over fields with a non-trivial, non-archimedean valuation.

\smallskip

All relevant basics regarding 
 theta functions and
abelian varieties will be reviewed in the next two sections,
along with pointers to the literature and to software.
As a warm-up, we first present the  solution to our motivating
problem for the much easier case of  genus $g=2$.

\begin{example}[The Universal Kummer surface]
\label{ex:Kummsur} \rm
The five-dimensional variety $\mathcal{K}_2$ is a hypersurface 
of degree $(12,4)$ in $\mathbb{P}^3 \times \mathbb{P}^3$.
Its principal prime ideal $\mathcal{I}_2$ in the polynomial ring $\mathbb{Q}[u_{00},u_{01},u_{10},u_{11};
x_{00},x_{01},x_{10},x_{11}]$  is generated by $1/16$ times
the determinant~of 
\begin{equation}
\label{universalk2}
\begin{bmatrix}
x_{00}^4 {+} x_{01}^4 {+} x_{10}^4 {+} x_{11}^4 &
 x_{00}^2 x_{01}^2 {+} x_{10}^2 x_{11}^2 &
  x_{00}^2 x_{10}^2 {+} x_{01}^2 x_{11}^2 &
x_{00}^2 x_{11}^2 {+} x_{01}^2 x_{10}^2 & x_{00} x_{01} x_{10} x_{11} \\
4 u_{00}^3 & 2 u_{00} u_{01}^2 & 2 u_{00} u_{10}^2 & 2 u_{00} u_{11}^2 & u_{01} u_{10} u_{11} \\
4 u_{01}^3 & 2 u_{00}^2 u_{01} & 2 u_{01} u_{11}^2 & 2 u_{01} u_{10}^2 & u_{00} u_{10} u_{11} \\
4 u_{10}^3 & 2 u_{10} u_{11}^2 &  2 u_{00}^2 u_{10} &  2 u_{01}^2 u_{10} &  u_{00} u_{01} u_{11} \\
4 u_{11}^3 & 2 u_{10}^2 u_{11} & 2 u_{01}^2 u_{11} & 2 u_{00}^2 u_{11} & u_{00} u_{01} u_{10} 
\end{bmatrix}
\end{equation}
For fixed $u_{ij}$, this $5 \times 5$-determinant defines Kummer's quartic 
as a surface
in $\mathbb{P}^3$, with coordinates $(x_{00}{:}x_{01}{:}x_{10}{:}x_{11})$,
written as in \cite[Exercise 3, page 204]{BL} or \cite[page 354]{Mum}.

Note that the last four rows in (\ref{universalk2}) represent the Jacobian matrix of the first row.
The matrix derives from the fact that  $(u_{00}{:}u_{01}{:}u_{10}{:}u_{11})$ 
must be a singular point on the Kummer surface.
The surface and its $16$ nodes are invariant
under the group $(\mathbb{Z}/2\mathbb{Z})^4$
which acts by sign changes and permutations
  \cite[page 353]{Mum}.
The $15$ other nodes are easily found:
$$
(u_{00} : -u_{10}: u_{01}: -u_{11}), \,
(u_{00}: u_{10}: -u_{01}: -u_{11}),\,
 \ldots , \,
(u_{11}: -u_{01}: -u_{10}: u_{00}).
$$
A convenient representation of the $16_6$ configuration of nodes 
is the matrix product
\begin{equation}
\label{eq:UX}
 \begin{pmatrix}
\phantom{-} u_{00} & \phantom{-}  u_{10}  &  \phantom{-}  u_{01} &   \phantom{-}  u_{11} \, \\
                          \phantom{-} u_{11} &     -u_{01} &  \phantom{-}   u_{10} &     -u_{00} \, \\
                          \phantom{-}  u_{01} &   \phantom{-}   u_{11} &    -u_{00} &   -u_{10} \, \\
                           -u_{10} &   \phantom{-} u_{00} &   \phantom{-}  u_{11} &     -u_{01} \,
\end{pmatrix} \cdot
\begin{pmatrix}
                          \phantom{-}  x_{00} &  \phantom{-}   x_{11} &     -x_{01} &  \phantom{-}   x_{10} \, \\
                         \phantom{-}   x_{10} &  \phantom{-}   x_{01} &    \phantom{-}  x_{11} &     -x_{00} \, \\
                        \phantom{-}    x_{01} &   -x_{10} &   \phantom{-}  x_{00} & \phantom{-}     x_{11} \, \\
                        \phantom{-}    x_{11} &    -x_{00} &    -x_{10} &    -x_{01} \,
\end{pmatrix},
\end{equation}
whose $16$ entries are the linear forms whose coefficients are the $16$ nodes.
As explained in Hudson's book \cite[\S 16]{Hu}, the
combinatorial structure of the $16$ nodes can be read off from this $4 \times 4$-matrix,
and it leads to various alternate forms of the defining quartic polynomial 
in \cite[\S 19]{Hu}.
In \cite[\S 102]{Hu} it is shown that the product (\ref{eq:UX})
expresses  quadratic monomials in theta functions with characteristics
in terms of the second order theta functions.

To illustrate how our coordinates can express geometric properties, we note that
\begin{equation}
\label{eq:UU}
\begin{matrix}
(u_{00} u_{11}+u_{01} u_{10}) (u_{00} u_{01}+u_{10} u_{11}) (u_{00}u_{10}+u_{01} u_{11}) 
(u_{00}u_{10}-u_{01} u_{11})\\ (u_{00} u_{11}-u_{01} u_{10})
(u_{00}^2+u_{01}^2-u_{10}^2-u_{11}^2) (u_{00}^2-u_{01}^2+u_{10}^2-u_{11}^2)  \\
(u_{00}u_{01}-u_{10} u_{11}) (u_{00}^2-u_{01}^2-u_{10}^2+u_{11}^2)(u_{00}^2+u_{01}^2+u_{10}^2+u_{11}^2)
\end{matrix}
\end{equation}
vanishes if and only if the given abelian surface 
is a product of two elliptic curves.
If this happens, then the  $16_6$ configuration of all nodes
degenerates to a more special matroid,
and the quartic Kummer surface in $\mathbb{P}^3$ degenerates to a 
double quadric.
The  $10$ factors in (\ref{eq:UU}) are the non-zero entries left in
the matrix product (\ref{eq:UX}) after replacing each $x_{ij}$  by $u_{ij}$. \qed
\end{example}

This article is organized as follows.
In Section~\ref{sec:thetaintro} we review classical material
that can be mostly found in the books of
Coble \cite{Cob} and Dolgachev--Ortland \cite{DO}.
We define a Kummer threefold in $\PP^7$ as
the image of a transcendental map whose coordinates are
second-order theta functions. Each Kummer threefold is the singular
locus of  the associated Coble quartic hypersurface in $\PP^7$,
which is a natural genus three analogue of the
Kummer surface in $\PP^3$.

In Section~\ref{sec:satakehyper} we focus on the moduli space $\mathcal{A}_3(2,4)$ of polarized abelian threefolds with suitable level structure.
That space is a quotient of the Siegel upper halfspace.
It is embedded into $\PP^7$ by second-order theta constants \cite{GG}. The resulting {\em Satake hypersurface} $\mathcal{S}$ has degree $16$. Its defining polynomial has $471$ terms with integer coefficients, displayed in Proposition \ref{lem:GG}.
We examine how the combinatorics  of this polynomial expresses the geometry of
$\mathcal{A}_3(2,4)$, notably its hyperelliptic locus, its
Torelli boundary, and its Satake boundary. Later in Table~\ref{fig:glass},
this is refined to the stratification of $\mathcal{S}$ that was found by
Glass \cite[Theorem 3.1]{glass}.

The theta series in Sections~\ref{sec:thetaintro} and \ref{sec:satakehyper}
are defined over the field of complex numbers $\C$, and we use floating 
point approximations for computing them. 
The resulting algebraic objects, however, do not require the complex numbers.
This is, of course, well-known to the experts in abelian varieties \cite{BL}.
All the ideals we feature in this paper can be generated
by polynomials with integer coefficients, and their projective varieties
are thus defined over any field.

In Section~\ref{sec:param:gopel}  we express the G\"opel variety 
$\mathcal{G}$ as the Zariski closure of the image of an explicit rational
map $\PP^6 \dashrightarrow \PP^{14}$ of degree 24. Its $15$ coordinates  are polynomials of degree $7$
that span an irreducible representation of the Weyl group of type $\mathrm{E}_7$.
The root system and its reflection arrangement play a prominent role,
as do configurations of 7 points in $\PP^2$. After completion of our work, we learned that our
 parametrization of the G\"opel variety had already been 
studied in \cite{CGL}, under the name {\em Coble linear system}. 
The connection is made explicit in Theorem \ref{thm:indeterminacy}
where we give an affirmative answer to \cite[Question 4.19]{CGL}. 

In Section~\ref{sec:gopeleqns} we study the defining prime ideal of the G\"opel variety $\mathcal{G}$.
We show that it is minimally generated by $35$ cubics and $35$ quartics
in $15$ variables.
This ideal is Gorenstein, it has degree $175$, and we determine its Hilbert series and its minimal free resolution.

Section~\ref{sec:toricgopel} is concerned with a beautiful re-embedding of the 
G\"opel variety  $\mathcal{G} \subset \PP^{14}$
into a projective space of dimension $134$.
The $135$ coordinates are the G\"opel functions of \cite[\S 9.7]{DO}.
The Weyl group $W(\mathrm{E}_7)$ acts on these by signed permutations.
We construct a toric variety $\mathcal{T}$ of dimension $35$ in
$\PP^{134}$ whose intersection with $\PP^{14}$ is $\mathcal{G}$.
Algebraically, the ideal of $\mathcal{G}$  is now generated by binomials and linear trinomials.
We study the combinatorics of the toric ideal of $\mathcal{T}$ and its associated convex polytope
$\mathcal{A}$.
Its $63$ distinguished facets and its $135$ vertices are indexed by the lines
and the Lagrangians in the finite symplectic space $(\F_2)^6$.

In Section~\ref{sec:univcoble} we return to the embedding of $\mathcal{K}_3$ 
defined by theta constants and theta functions.
The {\em universal Coble quartic} is an irreducible subvariety
of codimension two in $\PP^7 \times \PP^7$. It is the
complete intersection of the Satake hypersurface $\mathcal{S}$
of degree $(16,0)$ and one other hypersurface $\mathcal{C}$ of
degree $(28,4)$. The latter is given by an explicit polynomial with $372060$ terms in $16$ variables. Together, the two polynomials generate a prime ideal.
In response to the  preprint version of this article, Grushevsky and Salvati Manni \cite{gsnew} found a shorter representation of the same polynomial, using a more conceptual geometric approach.
This  was further developed by
Dalla Piazza and Salvati Manni in \cite{dpsm};
see Remark \ref{rmk:piazza}.

The universal Kummer variety $\mathcal{K}_3$ lives in
 $ \mathcal{S} {\times} \PP^7 \subset \PP^7 {\times} \PP^7$.
 Other variants of this variety can be defined in
 $\PP^6 \times \PP^7$, via the parametrization of Section~\ref{sec:param:gopel},
 or in $\mathcal{G} {\times} \PP^7 \subset \PP^{14} {\times} \PP^7$,
 via the G\"opel embeddings of Sections~\ref{sec:gopeleqns} and \ref{sec:toricgopel}.
 In Section~\ref{sec:kumeqn} we study the bihomogeneous prime ideals
 for these three variants of the universal Kummer variety.
 We derive lists of minimal generators for two of these ideals.
  Conjectures \ref{conj:uno} and \ref{conj:due} state that these lists are complete.
 
 Section~\ref{sec:tropgeom} examines all of our constructions from the perspective of tropical geometry.
 That perspective was further developed by three of us in the
 subsequent article \cite{RSS}.
 
 \subsection*{Supplementary materials}

We have created supplementary files so that  the reader can reproduce 
many of the calculations that are claimed throughout the text. Most of them are in the format of {\tt Macaulay 2} \cite{m2}. These files can be found at 
\url{http://math.berkeley.edu/~svs/supp/univ_kummer/}. 
We have also included the supplementary files in version 3 of the arXiv submission of this paper, and they can be obtained by downloading the source.

A word on the symbolic computations we present: many of the calculations are significantly faster (almost all finishing within a few seconds) if the field of coefficients chosen is finite; one that we have used is $\Z/101$. However, we have made no attempt to verify which primes give bad reduction. All calculations can also be performed over the rational numbers. In this setting, many calculations take at most a few minutes, and the hardest one in Theorem~\ref{thm:gopel} took approximately 15 minutes on a modern mid-level performance computer.

\subsection*{Acknowledgements}

We thank Melody Chan, Igor Dolgachev, Jan Draisma, 
Bert van Geemen, Sam Grushevsky, Thomas Kahle, Daniel Plaumann, and Riccardo Salvati Manni for helpful communications.
We made extensive use of the software packages 
{\tt Sage} \cite{sage}, {\tt Macaulay~2} \cite{m2}, and {\tt GAP} \cite{gap}.
Qingchun Ren was supported by a Berkeley fellowship.
Gus Schrader was supported by a Fulbright fellowship.
Steven Sam was supported by an NDSEG fellowship and a Miller research fellowship.
Bernd Sturmfels was partially supported by NSF grant DMS-0968882.

\section{From theta functions to Coble quartics} \label{sec:thetaintro}

In this section, we review Riemann's theta function
and its relatives in the case of genus $3$,
and we use this to give a definition of Kummer threefolds as subvarieties of $\PP^7$. We also discuss background material on the action of the 2-torsion of the associated abelian threefold on $\PP^7$ and introduce the Coble quartic associated to a non-hyperelliptic Kummer threefold.

\smallskip

Let $\tau$ be a symmetric $3\times 3$ matrix with complex entries whose imaginary part is positive definite.  The set $\mathfrak{H}_3$ of  all such matrices is a six-dimensional complex manifold, called the {\em Siegel upper halfspace}. Each matrix $\tau\in\mathfrak{H}_3$ determines a lattice $\Lambda=\mathbb{Z}^3+\tau \mathbb{Z}^3$ of rank $6$ in $\mathbb{C}^3$, and a three-dimensional abelian variety $A_\tau=\mathbb{C}^3/\Lambda$. The {\it Riemann theta function} corresponding to a matrix $\tau\in\mathfrak{H}_3$ is the function $\theta \colon \mathbb{C}^3\rightarrow\mathbb{C}$ defined by the Fourier~series
\begin{equation}
\label{thetafct}
\theta(\tau; z) \,\,= \,\, \sum_{n \in\mathbb{Z}^3}\exp\left[\pi i n^t \tau  n +2\pi i  n^t z \right].
\end{equation}
This series converges for all $ z \in \C^3$ and $\tau \in \mathfrak{H}_3$,
and it satisfies the functional equation
\begin{equation}
\label{fcteqn}
\theta(\tau ;z + a + \tau b\,) \,\, = \,\,\, \theta(\tau \,; z) \cdot \exp \bigg[ 2 \pi i \bigl(-b^t z - \frac{1}{2} 
b^t \tau b \bigr)\biggr]
\qquad \hbox{for} \,\, a,b \in \mathbb{Z}^3.
\end{equation}
Deconinck {\it et al.}~\cite{DHBHS} gave a careful convergence analysis and
they implemented the numerical evaluation of $\theta(\tau;z)$ in {\tt maple}.
Their work has been extended by  Swierczewski and Deconinck  \cite{SD}
who implemented
the evaluation of abelian functions in {\tt Sage} \cite{sage}.  The Riemann theta function $\theta$ 
can now be called in {\tt Sage} as {\tt RiemannTheta(tau)(z)}, where 
{\tt tau} is a Riemann matrix and {\tt z} is a complex vector.
We used that {\tt Sage} code extensively.

Every pair of binary vectors $\epsilon,\epsilon'\in \{0,1\}^3$ defines a {\it theta function with characteristics}
\begin{equation}
\label{eq:thetachar}
\theta[\epsilon|\epsilon'](\tau ; z) \,\,= \,\, 
\sum_{n \in\mathbb{Z}^3}\exp\left[\pi i (n+\frac{\epsilon}{2})^t \tau   (n+ \frac{\epsilon}{2}) 
\,+\,2\pi i  (n+ \frac{\epsilon}{2})^t (z+\frac{\epsilon'}{2}) \right].
\end{equation}
From inspection of this Fourier series, one can see that
\begin{equation}
\theta[\epsilon|\epsilon'](\tau;-z) \,\,\,=\,\,\,
(-1)^{\epsilon^t\epsilon'} \cdot \theta[\epsilon|\epsilon'](\tau;z).
\end{equation}
Therefore, of the $2^{2\cdot 3}=64$ theta functions with characteristics, 
precisely $2^{3-1}(2^3+1)=36$ are even functions of the argument  $z \in \mathbb{C}^3$, 
and the other $2^{3-1}(2^3-1)=28$ are odd functions of $z$.  
We shall refer to these as {\it even (or odd) theta functions}.  

Finally, for any binary vector $\sigma\in \{0,1\}^3$, we consider the {\em second order theta function}
\begin{equation} \label{secondtheta}
\begin{split} 
\Theta_2[\sigma](\tau; z)
& \,\,\, = \,\,\, \theta \bigl(  2\tau ;  2z+\tau\sigma \bigr) \cdot
\exp \left[ \pi i\left(\frac{\sigma^t\tau\sigma}{2}+2\sigma^t z \right) \right] \\
&  \,\,\, = \,\,\, \sum_{n\in \Z^3}\exp\left[  
2\pi i \bigl(n+\frac{\sigma}{2} \bigr)^t \tau \bigr( n+\frac{\sigma}{2})\,+\,
4 \pi i \bigl(n+\frac{\sigma}{2}\bigr)^t z
\right].
\end{split}
\end{equation}
The second order theta functions are related to the theta functions with characteristics of an isogenous abelian threefold by the formula
\begin{equation}
\Theta_2[\sigma](\tau;z)\, \,\,\,=\,\,\,\theta[\sigma|0](2\tau;2z).
\end{equation}
Further relations between first and second order theta functions are the {\it addition theorem}
\begin{equation}
\label{eq:addition}
\theta[\epsilon|\epsilon'](\tau;z+w) \cdot \theta[\epsilon|\epsilon'](\tau;z-w)
\,\,\,=\,\,\,
\sum_{\sigma\in\mathbb{F}_2^3}(-1)^{\sigma\cdot\epsilon'}
\cdot \Theta_2[\sigma](\tau;w) \cdot \Theta_2[\sigma+\epsilon](\tau;z)
\end{equation}
and its inversion
\begin{equation}
\label{eq:additioninv}
8\cdot\Theta_2[\sigma](\tau;w)\cdot \Theta_2[\sigma+\epsilon](\tau;z)
\,\,\,\,=\,\,\, \sum_{\epsilon'\in\mathbb{F}_2^3}(-1)^{\sigma\cdot\epsilon'}
\cdot \theta[\epsilon|\epsilon'](\tau;z+w) \cdot \theta[\epsilon|\epsilon'](\tau;z-w).
\end{equation}

For a fixed matrix $\tau \in \mathfrak{H}_3$, the 
eight second order theta functions define the {\em Kummer~map}
\begin{equation}
\label{kummermap}
\kappa_\tau \colon 
 \C^3 \rightarrow \PP^7 ,\,
\, z \mapsto \bigl(
 \Theta_2[000](\tau;z)  :
 \Theta_2[001](\tau;z) : \cdots : 
  \Theta_2[111](\tau;z) \bigr).
\end{equation}
The identity (\ref{fcteqn}) implies that
$\kappa_\tau$ factors through a map $ A_\tau \rightarrow \PP^7$
from the abelian threefold $A_\tau = \C^3/\Lambda$.
This map, which we also denote by $\kappa_\tau $, has the following
geometric description. The equation $\theta(\tau;z) = 0$
defines the {\em theta divisor} $\Theta$ on $A_\tau$. The divisor $ 2 \Theta$ is ample but not very ample.
The eight functions (\ref{secondtheta}) form a basis for its space
of sections. These are known as the {\em Schr\"odinger coordinates},
and we denote them by $x_{000},x_{001}, \ldots, x_{111}$.
The morphism
$A_\tau \rightarrow \PP(\mathrm{H}^0(A_\tau,2 \Theta)) \simeq \PP^7$
is given in coordinates by  (\ref{kummermap}).
The image of the Kummer map $\kappa_\tau $ in $\PP^7$ is
isomorphic to the quotient $A_\tau / \{z = -z\}$. We call this 
the {\em Kummer threefold} of~$\tau$.

Let $A_\tau[2]$ denote the subgroup of two-torsion points in  the abelian threefold $A_\tau$.
This is a group of order $64$ which we identify with $\Lambda/2\Lambda\simeq
 (\mathbb{F}_2)^6$.  The abelian group 
$A_\tau[2] \simeq (\mathbb{F}_2)^6$
acts naturally on the set of second order theta functions via
  \begin{equation}
\label{eq:involution}
\Theta_2[\sigma](\tau;z) \,\,\mapsto \,\, \Theta_2[\sigma](\tau;z+\delta) 
\qquad 
\hbox{where $\delta\in A_\tau[2]$.}
\end{equation}
This defines an action on  $\PP(\mathrm{H}^0(A_\tau,2 \Theta))
 \simeq \PP^7$ by permuting the coordinates $x_{ijk}$ up to sign. The action lifts to a linear representation of the {\em Heisenberg group} $H$ (see \cite[Chapter 6]{BL}).
This mildly non-abelian group is a certain central extension
\[
1 \to \mathbb{F}_2 \to H \to A_\tau[2] \to 1.
\]
This can be made explicit in terms of the Schr\"odinger coordinates $x_{000}, x_{001}, \ldots, x_{111}$ we are using on $\PP^7$.
The Heisenberg group $H$ is generated by the following six operators: 
\begin{equation}
\label{eq:sixoperators}
\begin{matrix}
x_{i,j,k} \mapsto x_{i+1,j,k}, \quad \,& 
x_{i,j,k} \mapsto x_{i,j+1,k}, \quad \,& 
x_{i,j,k} \mapsto x_{i,j,k+1}, \quad \,\\
x_{i,j,k} \mapsto (-1)^ix_{i,j,k},& 
 x_{i,j,k} \mapsto (-1)^jx_{i,j,k},&
  x_{i,j,k} \mapsto (-1)^kx_{i,j,k}.
\end{matrix}
\end{equation}
These operators commute up to sign and give the projective representation of $A_\tau[2]$.

\medskip

For any given matrix $\tau \in \mathfrak{H}_3$,
the Kummer threefold has degree $24$ in $\PP^7$.  When $\tau$ is the period matrix of a smooth non-hyperelliptic curve of genus three, the prime ideal of the corresponding Kummer threefold is
minimally generated by $8$ cubics and $6$ quartics in the variables ${\bf x}$.
The Hilbert polynomial of such a Kummer threefold, 
 which is $4 n^3 + 4$, agrees with the Hilbert function for $n \geq 1$.
This  was first derived by Wirtinger in \cite[\S 21]{wirt}.

\begin{remark} \rm
In commutative algebra \cite{eisenbud}, it is customary to 
also look at higher syzygies. Numerical invariants are read off
from the {\em Betti table}.
For a general Kummer threefold, it is
\begin{Verbatim}[samepage=true]
            0  1  2  3  4 5
     total: 1 14 48 56 24 3
         0: 1  .  .  .  . .
         1: .  .  .  .  . .
         2: .  8  .  .  . .
         3: .  6 48 56 24 3
\end{Verbatim}
This Betti table shows that the Kummer threefold $\kappa_\tau(A_\tau)$ is not
arithmetically Cohen--Macaulay.
Except for  Kummer surfaces in $\PP^3$,
failure to be arithmetically Cohen--Macaulay is a general phenomenon for  abelian varieties of dimension $\geq 2$.
In fact, this holds for
 varieties whose structure sheaf has intermediate cohomology via \cite[Cor.~A1.12, Prop.~A1.16]{geom_syz}. \qed
\end{remark}

To compute the above Betti table (in {\tt Macaulay~2} \cite{m2}),  we used the geometric representation of the Kummer threefold as the singular locus of a certain quartic hypersurface $C_\tau$ in $\PP^7$.  That hypersurface was characterized  by Arthur Coble in his seminal book \cite{Cob}.
We follow Dolgachev and Ortland \cite{DO}
in our discussion of the {\it Coble quartic}~$C_\tau$.

\begin{proposition}[{\cite[Proposition 2.2]{beauville}, \cite[\S 33]{Cob}, \cite[\S IX.5, Proposition 7]{DO}}] \label{prop:CC} 
Let $A_\tau$ be the Jacobian of a smooth non-hyperelliptic curve of genus three.
There exists a unique quartic hypersurface $C_\tau$ in $\mathbb{P}^7$ 
whose singular locus equals the Kummer threefold $\kappa_\tau(A_\tau)$. 
The eight partial derivatives of the defining polynomial of $C_\tau$ span the space of cubics containing $\kappa_\tau(A_\tau)$ and they generate the prime ideal of $\kappa_\tau(A_\tau)$ up to saturation.
\end{proposition}

The quartic $C_\tau$ is invariant under the action of the Heisenberg group $H$. 
The space of $H$-invariant quartics is $15$-dimensional, and
the defining polynomial $F_\tau$ of the Coble quartic hypersurface can be written as a 
linear combination of a basis for this space of $H$-invariants.
Using the Schr\"odinger coordinates ${\bf x}$ on $\PP^7$, we write 
\begin{equation}
\label{eq:coblequartic}
 \begin{matrix}
F_\tau \quad = \quad r  \cdot (x_{000}^4+x_{001}^4+x_{010}^4+x_{011}^4+x_{100}^4+x_{101}^4+x_{110}^4+x_{111}^4)
\quad \\
 + \, s_{001} \cdot (x_{000}^2 x_{001}^2+x_{010}^2 x_{011}^2+x_{100}^2 x_{101}^2+x_{110}^2 x_{111}^2) \\
 + \, s_{010} \cdot (x_{000}^2 x_{010}^2+x_{001}^2 x_{011}^2+x_{100}^2 x_{110}^2+x_{101}^2 x_{111}^2) \\
 + \, s_{011} \cdot (x_{000}^2 x_{011}^2+x_{001}^2 x_{010}^2+x_{100}^2 x_{111}^2+x_{101}^2 x_{110}^2) \\
 + \, s_{100} \cdot (x_{000}^2 x_{100}^2+x_{001}^2 x_{101}^2+x_{010}^2 x_{110}^2+x_{011}^2 x_{111}^2) \\
 + \, s_{101} \cdot (x_{000}^2 x_{101}^2+x_{001}^2 x_{100}^2+x_{010}^2 x_{111}^2+x_{011}^2 x_{110}^2) \\
 + \, s_{110} \cdot (x_{000}^2 x_{110}^2+x_{001}^2 x_{111}^2+x_{010}^2 x_{100}^2+x_{011}^2 x_{101}^2) \\
 + \, s_{111} \cdot (x_{000}^2 x_{111}^2+x_{001}^2 x_{110}^2+x_{010}^2 x_{101}^2+x_{011}^2 x_{100}^2) \\
 + \, t_{001} \cdot (x_{000} x_{010} x_{100} x_{110} + x_{001} x_{011} x_{101} x_{111}) \\
 + \, t_{010} \cdot (x_{000} x_{001} x_{100} x_{101} +x_{010} x_{011} x_{110} x_{111}) \\
 + \, t_{011} \cdot (x_{000} x_{011} x_{100}  x_{111}+x_{001} x_{010} x_{101}  x_{110}) \\
 + \, t_{100} \cdot (x_{000} x_{001} x_{010} x_{011}+x_{100} x_{101}x_{110} x_{111}) \\
 + \, t_{101} \cdot (x_{000} x_{010} x_{101} x_{111}+x_{001} x_{011} x_{100} x_{110}) \\
 + \, t_{110} \cdot (x_{000} x_{001} x_{110} x_{111}+x_{010} x_{011} x_{100} x_{101}) \\
 + \, t_{111} \cdot (x_{000} x_{011} x_{101} x_{110}+x_{001} x_{010} x_{100} x_{111}). \\
 \end{matrix}
\end{equation}
This representation
appears in \cite[\S IX.5, Proposition 8]{DO} and after Theorem 3.2 in \cite[Section 3]{Beau}.
The $15$ coefficients $r,s_\bullet, t_\bullet$ are parameters.
This notation is used throughout this paper.

\begin{remark} \rm
 The  monomials in ${\bf x}$ that appear in the equation (\ref{eq:coblequartic}) of the Coble quartic can be understood
combinatorially via the affine geometry of the 
eight-point vector space $(\mathbb{F}_2)^3$.
Namely, the four terms multiplied by
$s_{ijk}$ are the four affine lines parallel to the vector $(i,j,k)$, and
the two terms multiplied by $t_{ijk}$
are the two affine planes perpendicular
to the vector $(i,j,k)$, with respect to the 
usual dot product in $(\mathbb{F}_2)^3$. \qed
\end{remark}

In this paper, each non-hyperelliptic Kummer threefold will be represented as the variety in $\PP^7$ cut out by the eight partial derivatives $\partial F_\tau/\partial x_{ijk}$. The adjective ``universal'' in our title means that we
are working over the six-dimensional base of all Coble polynomials $F_\tau$. This family is obtained by letting $\tau$ run over the Siegel upper halfspace $\mathfrak{H}_3$.
In particular, $r, s_\bullet, t_\bullet$ depend analytically on $\tau$.
We shall review the relevant moduli space in  Section \ref{sec:satakehyper}

\smallskip

The Coble quartic is closely related to the Kummer surface
in Example \ref{ex:Kummsur}.  Indeed, the Kummer quartic is the expansion of
the determinant  (\ref{universalk2}) along the first row:
\begin{equation}
\label{eq:kumquar2}
\begin{matrix}
r \cdot (x_{00}^4 {+} x_{01}^4 {+} x_{10}^4 {+} x_{11}^4) \,+\, t \cdot ( x_{00} x_{01} x_{10} x_{11} ) \,\,+ \\ 
 s_{01} \cdot ( x_{00}^2 x_{01}^2 {+} x_{10}^2 x_{11}^2) + 
 s_{10}  \cdot ( x_{00}^2 x_{10}^2 {+} x_{01}^2 x_{11}^2) + 
 s_{11} \cdot ( x_{00}^2 x_{11}^2 {+} x_{01}^2 x_{10}^2 ).  \\
\end{matrix}
\end{equation}
The monomials in $x_{00},x_{01},x_{10},x_{11}$ seen here
 correspond to the  affine subspaces of $(\mathbb{F}_2)^2$.
The coefficients are polynomials of degree $12$ in the  theta constants $u_{ij}$.
They are obtained as the $4 \times 4$-minors of the
last four rows in (\ref{universalk2}). These minors satisfy the cubic equation
\begin{equation}
\label{eq:cubicrel}
16 r^3+rt^2+4 ( s_{01} s_{10} s_{11}- r s_{01}^2
- r s_{10}^2- r s_{11}^2) \,\, = \,\, 0. \end{equation}
This cubic defines a hypersurface in $\PP^4$ which is known 
as {\em Segre's primal cubic}.
In Section~\ref{sec:gopeleqns} we shall derive the analogous relations for the fifteen coefficients
of (\ref{eq:coblequartic}).

\begin{remark} \rm
Both the Kummer surface in $\PP^3$ and the Coble quartic in $\PP^7$ are self-dual hypersurfaces \cite{Pau}. The role of the $16$ singular points on the Kummer surface
is now played by the $64$ singular points on the Kummer threefold.
These are the images of the $2$-torsion points of $A_\tau$ under the Kummer map
$\kappa_\tau$. In analogy to the $16$ entries of the matrix (\ref{eq:UX}),
we consider the $64$ linear forms 
in (\ref{thetasquares}) below. Their coefficients
give the $64$ special singular points of $\kappa_\tau(A_\tau)$ in $\PP^7$.
These  points lie on $64$ special hyperplanes.
Self-duality gives a $64_{28}$ configuration consisting of
points and hyperplanes in $\mathbb{P}^7$.  If $A_\tau$ is the Jacobian of a genus three hyperelliptic curve,
then the Coble quartic $F_\tau$ becomes the square of a quadric,
as we shall see in (\ref{eq:squareofquadric}), and the Kummer threefold fails to be projectively normal \cite[\S 2.9.3]{khaled}. \qed
\end{remark}
 
 \section{The Satake hypersurface} \label{sec:satakehyper}

In order to understand the family of all Kummer threefolds,
we now vary the matrix $\tau$ throughout the Siegel upper halfspace
$\mathfrak{H}_3$. The {\it modular group} $\text{Sp}_{6}(\mathbb{Z})$ consists of block matrices $\gamma=\begin{pmatrix}
a&b\\
c&d
\end{pmatrix}$ where $a,b,c,d$ are $3\times 3$ matrices with integer entries such that $\gamma J\gamma^t=J$, where $J=\begin{pmatrix}
0& \! -{\rm Id}_3\\
{\rm Id}_3& \! 0
\end{pmatrix} $.  
Following \cite{BL, DO, Mum}, 
the modular group $\text{Sp}_{6}(\mathbb{Z})$ acts on $\mathfrak{H}_3$ by
$$
\gamma\circ\tau\,\,=\,\,(\tau c+d)^{-1}(\tau a+b).
$$
The quotient is the moduli space of principally polarized abelian threefolds:
\[
\mathcal{A}_3\,\,=\,\,\mathfrak{H}_3/\text{Sp}_{6}(\mathbb{Z}).
\]
 
We will also consider certain {\it level covers} of $\mathcal{A}_3$, which can be constructed by taking quotients of $\mathfrak{H}_3$ by appropriate normal congruence subgroups of $\text{Sp}_{6}(\mathbb{Z})$.  The subgroup
\begin{equation}
\label{eqn:gamma2}
\Gamma_3(2)\,\,=\,\,
\left\{\,\gamma
\in \text{Sp}_6(\mathbb{Z}) : \gamma\equiv \begin{pmatrix} {\rm Id}_3 &0\\0& {\rm Id}_3 \end{pmatrix} \pmod 2\right\}
\end{equation}
has index $|\text{Sp}_6(\mathbb{F}_2)| = 
1451520$; see (\ref{eq:1451520}). We also define 
$$
\Gamma_3(2,4)\,\,=\,\,\left\{\gamma=\begin{pmatrix}
a&b\\
c&d
\end{pmatrix} \in \Gamma_3(2)\ : \ \text{diag}(a^Tb)\equiv\text{diag}(c^Td)\equiv0 \mod 4\right\}.
$$
This group has index $64$ in $\Gamma_3(2)$, and the
quotient group $\Gamma_3(2)/\Gamma_3(2,4)$ is isomorphic to $(\mathbb{F}_2)^6$.
These subgroups determine moduli spaces
\[
\mathcal{A}_3(2)  = \mathfrak{H}_3/\Gamma_3(2) 
\quad {\rm and} \quad  \mathcal{A}_3(2,4) = \mathfrak{H}_3/\Gamma_3(2,4).
\]
We record that the induced quotient map of moduli spaces is a 64-to-1 cover:
\begin{equation}
\label{eq:64to1}
\mathcal{A}_3(2,4)  \xrightarrow{64:1} \mathcal{A}_3(2).
\end{equation}

The Torelli map gives an embedding  of the moduli space $\mathcal{M}_3$
 of smooth genus three~curves into~$\mathcal{A}_3$. We can thus regard
 $\mathcal{M}_3$ as a subset of $\mathcal{A}_3$.  
(For the experts: this would not be correct if we considered these spaces
as stacks, but we will not do this here.)
The inverse images of $\mathcal{M}_3$
in  $\mathcal{A}_3(2)$ and in $\mathcal{A}_3(2,4)$
are denoted by  $\mathcal{M}_3(2)$ and  $\mathcal{M}_3(2,4)$ respectively.

 The moduli spaces above are six-dimensional quasi-projective varieties, and
we are interested in the homogeneous prime ideals of certain embeddings.
In this section we focus on $\mathcal{A}_3(2,4)$.
Our point of departure is the {\em theta constant map}  $\vartheta\colon \mathfrak{H}_3\rightarrow\mathbb{P}^7$,
which is defined~by
\begin{equation}
\label{thetaconstants}
\vartheta \colon \tau \,\, \mapsto  \,\,
\bigl(
 \Theta_2[000](\tau;0)  :
 \Theta_2[001](\tau;0) : \cdots : 
  \Theta_2[111](\tau;0) \bigr).
\end{equation}
This map is not injective, since the second order theta constants
\begin{equation}
\label{eq:2ndorder}
\Theta_2[\sigma](\tau;0)
  \,\,\, = \,\,\, \sum_{n\in \Z^3}\exp\left[  
2\pi i \bigl(n+\frac{\sigma}{2} \bigr)^t \tau \bigr( n+\frac{\sigma}{2})\,
\right]
\end{equation}
 are modular forms of weight $1/2$ with respect to the   subgroup $\Gamma_3(2,4)$ 
 of  $\text{Sp}_6(\mathbb{Z})$.  Hence the theta constant map $\vartheta$ can be regarded as
 a morphism from the level cover $\mathcal{A}_3(2,4)=\mathfrak{H}_3/\Gamma_3(2,4)$ 
 into the projective space $\PP^7$.
This map is an embedding by \cite[Theorem 4.1]{glass}.

\smallskip

The closure $\mathcal{S} = \overline{\vartheta(\mathcal{A}_3(2,4))}$ of the image of the theta map in $\PP^7$ is a six-dimensional hypersurface. It is isomorphic to the Satake  compactification $\overline{\mathcal{A}_3(2,4)}$
 of the moduli space ${\mathcal{A}_3(2,4)}$.  This is a degree 16 hypersurface which we call the {\em Satake hypersurface}.
Its defining polynomial, which we also denote by $\mathcal{S}$, was found by Runge \cite[\S 6]{runge} and we now describe it. Write $G$ for the
 permutation group of order $1344$ that is generated by the four transpositions
$$ (u_{001} u_{010})(u_{101} u_{110}) ,\,
  (u_{010} u_{100})(u_{011} u_{101}),\,
 (u_{000} u_{001})( u_{010} u_{011}) , \,{\rm and} \,\,
 (u_{100} u_{101}) (u_{110}  u_{111}). $$
As an abstract group,  $\,G=\text{SL}_3(\mathbb{F}_2) \rtimes\left(\mathbb{F}_2\right)^3$.
In our context, it is precisely the subgroup of
$  \text{Sp}_6(\mathbb{F}_2)$ that acts on the Schr\"odinger variables ${\bf u}$ by coordinate permutations. Let $[\, {\bf u}^* ]_m$ denote the sum over the
$G$-orbit of a monomial ${\bf u}^*$. The index $m$ is the size of this orbit.

\begin{proposition}[Runge] \label{lem:GG}
The Satake hypersurface $\mathcal{S} = \overline{\vartheta(\mathcal{A}_3(2,4))}$
is an irreducible hypersurface of degree $16$ in $\PP^7$. Its defining polynomial is the following sum  of $471$ monomials:
$$ \begin{matrix}
  [u_{000}^4 u_{001}^4 u_{010}^4 u_{100}^4]_{56} \,
 - \, 2 [u_{000}^9 u_{001} u_{010} u_{100} u_{011} u_{101} u_{110} u_{111} ]_8   
 + \, 2 [u_{000}^4 u_{001}^4 u_{010}^2 u_{100}^2 u_{011}^2 u_{101}^2 ]_{84} \\
\! + [u_{000}^8 u_{001}^2 u_{010}^2 u_{100}^2 u_{111}^2]_{56} 
 -  [u_{000}^6 u_{001}^4 u_{010}^2 u_{100}^2 u_{110}^2]_{224} 
 + 4 [u_{000}^5 u_{001}^5 u_{010} u_{100} u_{011} u_{101} u_{110} u_{111} ]_{28} \\
 -\, 16  \,[u_{000}^3 u_{001}^3 u_{010}^3 u_{100} u_{011}^3 u_{101} u_{110} u_{111} ]_{14}
\,\,+\,\, 72 \,u_{000}^2 u_{001}^2 u_{010}^2 u_{100}^2 u_{011}^2 u_{101}^2 u_{110}^2 u_{111}^2.
\end{matrix}
$$
\end{proposition}

A more compact representation of the Satake hypersurface $\mathcal{S}$, in
terms of second order theta constants, appears in
\cite[Example~1.4]{GG}. Our $471$ monomials can be derived from these.

We next discuss a beautiful direct relationship between the polynomials
in Example \ref{ex:Kummsur} and Proposition \ref{lem:GG}.
We learned this from Sam Grushevsky and Riccardo Salvati Manni.
Both polynomials have total degree $16$ in eight unknowns, and
the former appears as an initial form in the latter.
Let $q$ denote a deformation parameter and~set
\begin{equation}
\label{eq:fourier}
\!\!\!\! \begin{matrix}
u_{000}=u_{00}+O(q^4), &
u_{001}=u_{01}+O(q^4), &
u_{010}=u_{10}+O(q^4), &
u_{011}=u_{11}+O(q^4), \\
u_{100}=2q x_{00}{+}O(q^9), &
\!\! u_{101}=2q x_{01}{+}O(q^9), &
\!\! u_{110}=2q x_{10}{+}O(q^9), &
\!\! u_{111}=2q x_{11}{+}O(q^9).
\end{matrix}
\end{equation}
Under this substitution,
the Satake polynomial in Proposition \ref{lem:GG} takes the form
$$ \mathcal{S} \,\,\, = \,\,\,{\rm det}[ \bullet ] \cdot q^4 \,+\, O(q^8) , $$
where $[\bullet]$ is the $5 {\times} 5$-matrix of 
Example \ref{ex:Kummsur} that defines the
universal Kummer surface $\mathcal{K}_2$.

This identity can be derived from the
{\em Fourier--Jacobi expansion}  of second-order theta constants.
Namely, in that expansion we write the $3 \times 3$-matrix $\tau \in \mathfrak{H}_3$ in the form
$$
\qquad \tau= \begin{pmatrix}
s & 2z^t\\
2z&\tau'
\end{pmatrix},
\quad \hbox{
where $\tau' \in \mathfrak{H}_2$, $s \in \C$, and $z \in \C^2$.}
$$
If we set $q=e^{\pi is/2}$, then each of the eight genus three second order theta constants
$$
u_{\sigma}\,=\,\Theta_2(\tau;0)[\sigma]
$$ 
has a Taylor series expansion in $q$. The leading coefficient in these Taylor series is either a
theta constant or a theta function of genus two. These expansions are given in~(\ref{eq:fourier}).

\smallskip

The Satake hypersurface $\mathcal{S} = \overline{\vartheta(\mathcal{A}_3(2,4))}$ 
in $\PP^7$ contains several loci of geometric~interest:
\begin{compactitem}
\item The  {\em hyperelliptic locus} has codimension one in $\mathcal{S}$,
and hence codimension two in $\PP^7$.
\item The {\em Torelli boundary}
$\mathcal{S} \setminus \vartheta(\mathcal{M}_3(2,4))$
has codimension two in $\mathcal{S}$.
\item 
The  {\em Satake boundary} $\mathcal{S}
\setminus \vartheta(\mathcal{A}_3(2,4))$
has codimension three in $\mathcal{S}$.
\end{compactitem}
For each of these loci, we shall describe its irreducible components and  defining polynomials.
We begin with the Satake boundary. 
 By \cite[Lemma 3.5]{vangeemen}, 
 it consists of $126 $ three-dimensional subspaces $\PP^3$ in $\PP^7$:
each of the $63$ non-zero half-periods 
$\epsilon \in  \Lambda/2 \Lambda$ induces a linear involution on $\mathbb{P}^7$
via the action on second order theta constants
given in (\ref{eq:involution}), and
the fixed point set  of this involution on $\mathbb{P}^7$ is the union of two  $\PP^3$s.
For instance, for $\epsilon = (1/2,0,0)$, the involution (\ref{eq:involution}) 
fixes the coordinates 
$u_{000}, u_{001}, u_{010}, u_{011}$, it switches the sign on
$u_{100}, u_{101}, u_{110}, u_{111}$, and the two  $\PP^3$s
are obtained by setting either of these two groups of four variables to zero.

The hyperelliptic locus in $\mathcal{S}$ is the closure of the set of all points $\vartheta(\tau)$
where $A_\tau $ is the Jacobian of a smooth hyperelliptic curve of genus 3.
It is known (see e.g.~\cite{glass}) that a genus 3 curve is hyperelliptic if and only if one of its 
$36$ first order theta constants  $\theta[\epsilon|\epsilon'](\tau; 0)$ vanishes.
We write these $36$ divisors in our eight coordinates $u_{000}, \ldots, u_{111}$ using the formula
\begin{equation}
\label{thetasquares2}
\theta[\epsilon|\epsilon'](\tau;0)^2\,\,\,=\,\,\,
\sum_{\sigma\in\mathbb{F}_2^3}(-1)^{\sigma\cdot\epsilon'}
\cdot u_\sigma \cdot u_{\sigma+\epsilon},
\end{equation}
which is obtained by setting $w=z=0$ in 
the addition theorem (\ref{eq:addition}).
Hence the hyperelliptic locus in $\mathcal{S}$ can be defined set-theoretically by the equation of degree $72$ obtained by taking the product of the quadrics (\ref{thetasquares2}) where $(\epsilon, \epsilon')$ runs over the $36$ even theta characteristics. 

We will see in Section~\ref{sec:univcoble} that the product $\prod\theta[\epsilon|\epsilon'](\tau;0)$ of the
$ 36$ first order theta constants is in fact a polynomial in  ${\bf u}$, so the scheme defined by the degree
$ 72$ polynomial is not reduced.  Modulo each of the quadrics (\ref{thetasquares2}), the Satake polynomial becomes the square of an octic. In the supplementary materials, we give this octic for one of the $36$ components, as well as a verification that the subscheme of this component 
defined by the octic is reduced. So each component of the hyperelliptic locus is a complete intersection of degree $16$.   Hence as a subvariety of $\PP^7$, the hyperelliptic locus has codimension two, degree $576=16\cdot36$, and $36$ irreducible components. Glass analyzed in \cite[Theorem 3.1]{glass}
how many of these
$36$ theta constants can simultaneously vanish. We shall refine his results  in 
our Table~\ref{fig:glass}.

\smallskip

The Torelli boundary is the closure of the image under $\vartheta$ of  the set of all polarized abelian threefolds that decompose as the product, as a polarized variety, of an elliptic curve and an abelian surface. 
This means that $\tau$ can be transformed into  block form under $\mathrm{Sp}_6(\mathbb{Z})$.

 \begin{proposition} 
 \label{prop:torellibdr} The Torelli boundary coincides with the
 singular locus of the Satake hypersurface
 $\mathcal{S}$.  It is the union of
 $336$ irreducible four-dimensional subvarieties of $\PP^7$,
 each of which is defined by the $2 \times 2$-minors
 of a $2 \times 4$-matrix of linear forms in $u_{000}, u_{001},\ldots, u_{111}$.
 \end{proposition}
 
 \begin{proof}
 By \cite[Lemma 3.2]{GG}, the Satake hypersurface is non-singular 
 at each point that represents an indecomposable polarized abelian threefold.
It follows from \cite[Theorem 3.1]{glass}
that a polarized abelian threefold decomposes as a product if and only if at least two of its first order theta constants vanish.  Moreover, the vanishing of two first theta constants implies the vanishing of at least six of the quadrics in (\ref{thetasquares2}).
The relevant $6$-tuples $I$ of even theta characteristics 
$m = [\epsilon|\epsilon']$
have the property that, for any three $m_1,m_2,m_3$ in $ I$, the sum $m_1+m_2+m_3$ is an odd characteristic. There are $336$ such $6$-tuples $I$. They correspond to {\em azygetic triads of Steiner complexes}, 
and hence to non-isotropic planes in the symplectic vector space $(\F_2)^6$, 
and hence also to root subsystems $\mathrm{A}_2$ in $\mathrm{E}_7$; see \cite[Prop.~1(2)]{manivel}.

 Now, if $\vartheta(\tau)$ is in the Torelli boundary then the $3 \times 3$-matrix $\tau$ is in the $\text{Sp}_6(\mathbb{Z})$-orbit of a matrix $\tau_0$ that decomposes into two blocks given by matrices in $\mathfrak{H}_1$
 and $\mathfrak{H}_2$. Since theta constants behave multiplicatively
 under this decomposition, the image of the locus of $\tau_0$ admitting such a decomposition is a Segre variety
 $\PP^1 \times  \PP^3$ in $\PP^7$.
   That Segre variety is defined by the $2 \times 2$-minors of
  the   $2 \times 4$-matrix 
  \begin{equation}
  \label{eq:2x4matrix}
  \begin{pmatrix} u_{000} & u_{001} & u_{010} & u_{011} \\
                              u_{100} & u_{101} & u_{110} & u_{111}
   \end{pmatrix}.
   \end{equation}
It is easy to check in {\tt Macaulay~2} that the ideal of $2 {\times} 2$ minors of \eqref{eq:2x4matrix} contains the partial derivatives of the Satake polynomial $\mathcal{S}$. In particular, this component consists of singular points.
Since the other $335$ components $\PP^1 \times \PP^3$ are obtained by applying the action of the modular group, we see that all components are in the singular locus. Putting everything together, we 
conclude that the singular locus of $\mathcal{S}$ coincides with the Torelli boundary.
\end{proof}

The $336$ irreducible components of the Torelli boundary in $\mathcal{S}$
are a direct generalization of the $10$ irreducible factors in (\ref{eq:UU}).
In both cases, the components are defined by certain quadrics in
the ${\bf u}$-coordinates that can be written as $2 \times 2$-determinants
of linear forms. In Section~\ref{sec:univcoble} we shall return to 
the equations of the  Torelli boundary  when we study the
{\em universal Coble quartic} in $\PP^7 \times \PP^7$.
 Let us now define our universal objects in precise terms.

\smallskip

We  combine the Kummer map (\ref{kummermap})
and the theta constant map (\ref{thetaconstants}) by setting
\begin{equation}
\label{univkummer}
\kappa \colon \mathfrak{H}_3 \times \C^3 \rightarrow \PP^7 \times \PP^7 ,\qquad
(\tau,z) \mapsto (\vartheta(\tau), \kappa_\tau(z)) .
\end{equation}
This is the {\em universal Kummer map} in genus $g=3$. The closure of its image
in $\PP^7 \times \PP^7$  is the {\em universal Kummer threefold}.
This irreducible variety of dimension nine is denoted
$$ \mathcal{K}_3 \,\, = \,\,
 \mathcal{K}_3(2,4) \,\,:=\,\, \overline{\kappa(\mathfrak{H}_3 \times \C^3)}. $$
 We fix the coordinates ${\bf u} = (u_{000}: u_{001}: \cdots : u_{111})$ 
on the first copy of $\PP^7$, and we fix the coordinates
${\bf x} = (x_{000}: x_{001}: \cdots : x_{111})$ on the second $\PP^7$.
Thus the $u_{ijk}$ represent the theta constants which can be regarded as coordinates for 
the moduli space $\mathcal{A}_3(2,4)$ of polarized
abelian threefolds with level structure, while the 
$x_{ijk}$ represent the second order theta functions which are coordinates
 for the individual abelian threefolds $A_\tau$ themselves.

The following two formulas relate these coordinates on our
 $\PP^7 \times \PP^7$ to the theta functions with characteristics.
First, by specializing $w=z$  in (\ref{eq:addition}) we obtain the formula
\begin{equation}
\label{doubledarg}
\theta[\epsilon|\epsilon'](\tau;2z) \cdot \theta[\epsilon|\epsilon'](\tau;0)
\,\,\,=\,\,\,
\sum_{\sigma\in\mathbb{F}_2^3}(-1)^{\sigma\cdot\epsilon'} \cdot
x_\sigma \cdot x_{\sigma+\epsilon}.
\end{equation}
Second, the specialization $w=0$ in (\ref{eq:addition}) yields 
\begin{equation}
\label{thetasquares}
\theta[\epsilon|\epsilon'](\tau;z)^2\,\,\,=\,\,\,
\sum_{\sigma\in\mathbb{F}_2^3}(-1)^{\sigma\cdot \epsilon'}
\cdot u_\sigma \cdot x_{\sigma+\epsilon}.
\end{equation}
The identity (\ref{thetasquares2}) arises from either of these by setting $z = 0$.
Our formulas admit $\mathbb{Z}[1/8]$-linear inversions, derived from (\ref{eq:additioninv}),
that express the various quadratic monomials on $\PP^7 \times \PP^7$ in terms
of thetas with characteristics.
 Equation (\ref{doubledarg})
expresses first order theta functions with doubled argument 
$2z$ as quadratic polynomials in the second order theta functions.

It is sometimes advantageous to embed the
Kummer threefold not in $\PP^7$, but in the larger space $\PP^{35}$
whose coordinates are indexed by even pairs $(\epsilon,\epsilon')$.
Likewise, it makes sense to re-embed
the universal Kummer variety $\mathcal{K}_3(2,4)$ from $  \PP^7 \times \PP^7$
into $\mathbb{P}^{35}\times\mathbb{P}^{35}$.
Using the addition theorem for theta functions \eqref{eq:addition}, this can be accomplished in two 
 different ways.   First, we can use the formulas (\ref{thetasquares}) 
 and (\ref{thetasquares2})
  to map $\mathcal{K}_3(2,4)$ into $\mathbb{P}^{35}\times\mathbb{P}^{35}$ with coordinates
\begin{equation}
\label{eq:P35A}
\left(\,\theta[\epsilon|\epsilon'](\tau;0)^2, \,
           \theta[\epsilon|\epsilon'](\tau;z)^2 \, \right).
\end{equation}
Second, we can use the formulas 
(\ref{doubledarg}) and (\ref{thetasquares2})
 to map $\mathcal{K}_3(2,4)$ into $\mathbb{P}^{35}\times\mathbb{P}^{35}$ with coordinates
\begin{equation}
\label{eq:P35B}
\left(\,\theta[\epsilon|\epsilon'](\tau;0)^2,\, \theta[\epsilon|\epsilon'](\tau;2z)\,\right).
\end{equation}
In both cases, the re-embedding is given by a certain Veronese map
$\PP^7 \times \PP^7 \rightarrow \PP^{35} \times \PP^{35}$.

\begin{remark} \rm
Our motivating problem was to determine
the bihomogeneous prime~ideal 
\[
\mathcal{I}_3 \,\, \subset \,\, \mathbb{Q} [{\bf u},{\bf x}]
\]
of the universal Kummer threefold $\mathcal{K}_3(2,4)$.
One of the minimal generators of $\mathcal{I}_3$ is the Satake polynomial 
of degree $(16,0)$. To find others,  one might try the following approach.
For any two non-negative integers $r$ and $s$, consider the
space $\mathbb{Q}[{\bf u},{\bf x}]_{r,s}$ of polynomials that are bihomogeneous
of degree $(r,s)$. This space has dimension
$ \binom{7+r}{7} \binom{7+s}{7}$. We seek to identify the subspace
$(\mathcal{I}_3)_{(r,s)}$ of polynomials that lie in our ideal $\mathcal{I}_3$.
That subspace can be computed using (numerical) linear algebra.
The basic idea is simple: using Swierczewski's 
code for  the Riemann theta function $\theta$,
we implemented pieces of {\tt Sage} code for
the second order theta functions (\ref{secondtheta}),
for the Kummer map (\ref{kummermap}),
for the theta constant map (\ref{thetaconstants}),
and for the universal Kummer map (\ref{univkummer}).
For each point $(\tau,z) \in \mathfrak{H}_3 \times \C^3$,
we can thus compute one linear constraint on polynomials in
$(\mathcal{I}_3)_{(r,s)}$, as these vanish at $\kappa(\tau,z)$.
By plugging in enough points, we get linear equations in $ \binom{7+r}{7} \binom{7+s}{7}$ unknowns
whose solution space equals $(\mathcal{I}_3)_{(r,s)}$.

In practice, however, this approach does not work at all.
The primary reason is that the size of our linear systems
is too large, even with the use of state-of-the-art
software for numerical linear algebra.
The key to any success would be the identification of 
subspaces in which equations for $\mathcal{I}_3$ can lie.
Methods from representation theory are essential here.

For instance, suppose we are told that
the Satake hypersurface has degree $16$, and 
we are asked to find its polynomial $\mathcal{S}$.
Solving the naive system of linear equations in $ \binom{7+16}{7}  = 245157$
unknowns is an impossible task.
However, all but $471$ of the unknown coefficients are zero. Solving the restricted system for the $471$ coefficients is easy.
We do this by computing a singular value decomposition of a matrix
whose rows are the $471$ monomials evaluated at
$\vartheta(\tau_i)$ for sufficiently many $3 \times 3$-matrices $\tau_i$.
The same works with the $1168$ monomials of degree $(16,4)$
in Lemma~\ref{lem:july27}.
From these we can compute a polynomial in $(\mathcal{I}_3)_{(16,4)}$. \qed
\end{remark}

 \section{Parametrization of the G\"opel variety} \label{sec:param:gopel}
 
 Consider the family of homogeneous quartic polynomials  in eight unknowns $x_{ijk}$
 given by the formula in (\ref{eq:coblequartic}).
 This family constitutes a projective space $\PP^{14}$ whose coordinates~are
 $$ (r: s_{001}: s_{010} : \cdots : s_{111}:   t_{001}: t_{010} : \cdots : t_{111}) .$$
 The {\em G\"opel variety} is the six-dimensional irreducible subvariety
 $\mathcal{G}$ of $\PP^{14}$  that is obtained as the closure of the set of all Coble quartics $F_\tau$ where $\tau$ runs over non-hyperelliptic $\tau \in \mathfrak{H}_3$.
  
 Coble \cite[\S 49]{Cob} describes $63$ cubic polynomials that cut out
$\mathcal{G}$  set-theoretically; see also \cite[\S IX.7, Proposition 9]{DO}.
In Section~\ref{sec:gopeleqns} we review Coble's construction, and we determine the prime ideal and its graded Betti table.
The G\"opel variety $\mathcal{G}$ is a compactification of the non-hyperelliptic locus inside the moduli space $\mathcal{M}_3(2)$. See \cite{Kon} for an analytic approach.

We now present an utterly explicit polynomial parametrization $\gamma$ of the  G\"opel variety
$\mathcal{G}$. Let $\PP^6$  be the projective space with coordinates
$(c_1:c_2:c_3:c_4:c_5:c_6:c_7)$. Our map
\begin{align} \label{eqn:param:weyl}
\gamma \colon \PP^6 \dashrightarrow \PP^{14}
\end{align}
is defined by the following $15$ homogeneous polynomials of degree $7$ in $c_1,c_2,\ldots,c_7$: 
\begin{align*}
r \,\, & =  \quad
4c_1c_2c_3c_4c_5c_6c_7 \\
s_{001} \, & = \phantom{-} c_1c_2c_7 (c_3^4-2c_3^2c_4^2+c_4^4-2c_3^2c_5^2-2c_4^2c_5^2+c_5^4-2c_3^2c_6^2-2c_4^2c_6^2-2c_5^2c_6^2+c_6^4) \\
s_{010} \, & =  -c_1c_3c_5 (c_2^4-2c_2^2c_4^2+c_4^4-2c_2^2c_6^2-2c_4^2c_6^2+c_6^4-2c_2^2c_7^2-2c_4^2c_7^2-2c_6^2c_7^2+c_7^4) \\
s_{011} \, & =  \phantom{-}c_1c_4c_6 (c_2^4-2c_2^2c_3^2+c_3^4-2c_2^2c_5^2-2c_3^2c_5^2+c_5^4-2c_2^2c_7^2-2c_3^2c_7^2-2c_5^2c_7^2+c_7^4) \\
s_{100} \, & = -c_2c_3c_6 (c_1^4-2c_1^2c_4^2+c_4^4-2c_1^2c_5^2-2c_4^2c_5^2+c_5^4-2c_1^2c_7^2-2c_4^2c_7^2-2c_5^2c_7^2+c_7^4)  \\
s_{101} \, & = 
\phantom{-} c_2c_4c_5 (c_1^4-2c_1^2c_3^2+c_3^4-2c_1^2c_6^2-2c_3^2c_6^2+c_6^4-2c_1^2c_7^2-2c_3^2c_7^2-2c_6^2c_7^2+c_7^4) \\
s_{110} \, & =
-c_3c_4c_7 (c_1^4-2c_1^2c_2^2+c_2^4-2c_1^2c_5^2-2c_2^2c_5^2+c_5^4-2c_1^2c_6^2-2c_2^2c_6^2-2c_5^2c_6^2+c_6^4) \\
s_{111} \, & =
\phantom{-} c_5c_6c_7 (c_1^4-2c_1^2c_2^2+c_2^4-2c_1^2c_3^2-2c_2^2c_3^2+c_3^4-2c_1^2c_4^2-2c_2^2c_4^2-2c_3^2c_4^2+c_4^4)  \\
t_{100} \, & = \,\,\,c_1(-c_2^4c_3^2+c_2^2c_3^4+c_2^4c_4^2-c_3^4c_4^2-c_2^2c_4^4+c_3^2c_4^4-c_2^4c_5^2-2c_2^2c_3^2c_5^2 +2c_3^2c_4^2c_5^2+c_4^4c_5^2 \\* 
& \quad +c_2^2c_5^4-c_4^2c_5^4 +c_2^4c_6^2-c_3^4c_6^2+2c_2^2c_4^2c_6^2-2c_3^2c_4^2c_6^2+2c_3^2c_5^2c_6^2-2c_4^2c_5^2c_6^2-c_5^4c_6^2\\* 
& \quad -c_2^2c_6^4+c_3^2c_6^4+c_5^2c_6^4+2c_2^2c_3^2c_7^2+c_3^4c_7^2-2c_2^2c_4^2c_7^2-c_4^4c_7^2+2c_2^2c_5^2c_7^2\\* 
& \quad -2c_3^2c_5^2c_7^2+c_5^4c_7^2-2c_2^2c_6^2c_7^2+2c_4^2c_6^2c_7^2-c_6^4c_7^2-c_3^2c_7^4+c_4^2c_7^4-c_5^2c_7^4+c_6^2c_7^4)  \\
t_{010} \, & = \,\,\, c_2(-c_1^4c_3^2+c_1^2c_3^4+c_1^4c_4^2-c_3^4c_4^2-c_1^2c_4^4+c_3^2c_4^4+c_1^4c_5^2-c_3^4c_5^2+2c_1^2c_4^2c_5^2-2c_3^2c_4^2c_5^2 \\* 
& \quad - c_1^2c_5^4+c_3^2c_5^4-c_1^4c_6^2-2c_1^2c_3^2c_6^2 +2c_3^2c_4^2c_6^2+c_4^4c_6^2+2c_3^2c_5^2c_6^2-2c_4^2c_5^2c_6^2\\* 
& \quad +c_5^4c_6^2+c_1^2c_6^4-c_4^2c_6^4-c_5^2c_6^4+2c_1^2c_3^2c_7^2+c_3^4c_7^2-2c_1^2c_4^2c_7^2-c_4^4c_7^2-2c_1^2c_5^2c_7^2\\* 
& \quad +2c_4^2c_5^2c_7^2-c_5^4c_7^2+2c_1^2c_6^2c_7^2-2c_3^2c_6^2c_7^2+c_6^4c_7^2-c_3^2c_7^4+c_4^2c_7^4+c_5^2c_7^4-c_6^2c_7^4)  \\
t_{001} \, & = \,\,\, c_3(-c_1^4c_2^2+c_1^2c_2^4+c_1^4c_4^2-c_2^4c_4^2-c_1^2c_4^4+c_2^2c_4^4+2c_1^2c_2^2c_5^2+c_2^4c_5^2 -2c_1^2c_4^2c_5^2\\* 
& \quad -c_4^4c_5^2-c_2^2c_5^4+c_4^2c_5^4-c_1^4c_6^2-2c_1^2c_2^2c_6^2+2c_2^2c_4^2c_6^2+c_4^4c_6^2+2c_1^2c_5^2c_6^2-2c_2^2c_5^2c_6^2\\* 
& \quad -c_5^4c_6^2+c_1^2c_6^4-c_4^2c_6^4+c_5^2c_6^4+c_1^4c_7^2 -c_2^4c_7^2+2c_1^2c_4^2c_7^2-2c_2^2c_4^2c_7^2-2c_1^2c_5^2c_7^2\\* 
& \quad +2c_4^2c_5^2c_7^2+c_5^4c_7^2+2c_2^2c_6^2c_7^2-2c_4^2c_6^2c_7^2-c_6^4c_7^2 -c_1^2c_7^4+c_2^2c_7^4-c_5^2c_7^4+c_6^2c_7^4) \\
t_{111} \, & = \,\,\, c_4(-c_1^4c_2^2+c_1^2c_2^4+c_1^4c_3^2-c_2^4c_3^2-c_1^2c_3^4+c_2^2c_3^4-c_1^4c_5^2-2c_1^2c_2^2c_5^2+2c_2^2c_3^2c_5^2+c_3^4c_5^2\\* 
& \quad +c_1^2c_5^4 -c_3^2c_5^4+2c_1^2c_2^2c_6^2+c_2^4c_6^2 -2c_1^2c_3^2c_6^2-c_3^4c_6^2+2c_1^2c_5^2c_6^2-2c_2^2c_5^2c_6^2+c_5^4c_6^2\\* 
& \quad -c_2^2c_6^4 +c_3^2c_6^4-c_5^2c_6^4+c_1^4c_7^2 -c_2^4c_7^2+2c_1^2c_3^2c_7^2-2c_2^2c_3^2c_7^2+2c_2^2c_5^2c_7^2-2c_3^2c_5^2c_7^2 \\* 
& \quad -c_5^4c_7^2 -2c_1^2c_6^2c_7^2 +2c_3^2c_6^2c_7^2+c_6^4c_7^2-c_1^2c_7^4+c_2^2c_7^4+c_5^2c_7^4-c_6^2c_7^4)\\
t_{011} \, & = \,\,\, c_5(-c_1^4c_2^2+c_1^2c_2^4+2c_1^2c_2^2c_3^2+c_2^4c_3^2-c_2^2c_3^4-c_1^4c_4^2-2c_1^2c_2^2c_4^2+2c_1^2c_3^2c_4^2-2c_2^2c_3^2c_4^2\\* 
& \quad -c_3^4c_4^2+c_1^2c_4^4+c_3^2c_4^4+c_1^4c_6^2 -c_2^4c_6^2-2c_1^2c_3^2c_6^2+c_3^4c_6^2+2c_2^2c_4^2c_6^2-c_4^4c_6^2\\* 
& \quad -c_1^2c_6^4+c_2^2c_6^4-c_3^2c_6^4+c_4^2c_6^4+c_1^4c_7^2-c_2^4c_7^2-2c_1^2c_3^2c_7^2+c_3^4c_7^2+2c_2^2c_4^2c_7^2 \\*
& \quad -c_4^4c_7^2+2c_1^2c_6^2c_7^2-2c_2^2c_6^2c_7^2+2c_3^2c_6^2c_7^2-2c_4^2c_6^2c_7^2-c_1^2c_7^4+c_2^2c_7^4-c_3^2c_7^4+c_4^2c_7^4) \\
t_{101} \, & = \,\,\, c_6(-c_1^4c_2^2+c_1^2c_2^4  - c_1^4c_3^2-2c_1^2c_2^2c_3^2+c_1^2c_3^4+2c_1^2c_2^2c_4^2+c_2^4c_4^2+2c_1^2c_3^2c_4^2 \\*
& \quad -2c_2^2c_3^2c_4^2+c_3^4c_4^2-c_2^2c_4^4-c_3^2c_4^4+c_1^4c_5^2-c_2^4c_5^2+2c_2^2c_3^2c_5^2-c_3^4c_5^2-2c_1^2c_4^2c_5^2+c_4^4c_5^2 \\* 
& \quad -c_1^2c_5^4+c_2^2c_5^4+c_3^2c_5^4-c_4^2c_5^4+c_1^4c_7^2-c_2^4c_7^2+2c_2^2c_3^2c_7^2-c_3^4c_7^2-2c_1^2c_4^2c_7^2+c_4^4c_7^2 \\* 
& \quad +2c_1^2c_5^2c_7^2 -2c_2^2c_5^2c_7^2-2c_3^2c_5^2c_7^2+2c_4^2c_5^2c_7^2-c_1^2c_7^4+c_2^2c_7^4+c_3^2c_7^4-c_4^2c_7^4) \\
t_{110} \, & = \,\,\,c_7(-c_1^4c_3^2+2c_1^2c_2^2c_3^2-c_2^4c_3^2+c_1^2c_3^4+c_2^2c_3^4-c_1^4c_4^2+2c_1^2c_2^2c_4^2-c_2^4c_4^2-2c_1^2c_3^2c_4^2\\* 
& \quad -2c_2^2c_3^2c_4^2+c_1^2c_4^4+c_2^2c_4^4+c_1^4c_5^2-2c_1^2c_2^2c_5^2+c_2^4c_5^2-c_3^4c_5^2+2c_3^2c_4^2c_5^2-c_4^4c_5^2-c_1^2c_5^4 \\* 
& \quad -c_2^2c_5^4+c_3^2c_5^4+c_4^2c_5^4+c_1^4c_6^2-2c_1^2c_2^2c_6^2+c_2^4c_6^2-c_3^4c_6^2+2c_3^2c_4^2c_6^2-c_4^4c_6^2+2c_1^2c_5^2c_6^2 \\* 
& \quad +2c_2^2c_5^2c_6^2-2c_3^2c_5^2c_6^2-2c_4^2c_5^2c_6^2-c_1^2c_6^4-c_2^2c_6^4+c_3^2c_6^4+c_4^2c_6^4)
\end{align*}

\begin{theorem}
\label{thm:24to1}
The G\"opel variety $\mathcal{G}$ is the closure of the image of the 
rational map $\gamma \colon \PP^6 \dashrightarrow \PP^{14}$
given by the polynomials above.
This map defines a $24$-to-$1$ cover of $\mathcal{G}$.
\end{theorem}

The rest of this section is devoted to a conceptual derivation
and geometric explanation of the map $\gamma$.
This will then furnish the proof of Theorem \ref{thm:24to1}.
In Section~\ref{sec:toricgopel} we shall see that $\gamma$ is equivalent to 
the embedding via G\"opel functions  described in \cite[\S IX.8, Theorem 5]{DO}.
Our construction here is based on the representation-theoretic approach developed in \cite{gsw}.

Consider the Heisenberg group $H$ that is generated by the six operators in (\ref{eq:sixoperators}).
Let $H'$ be the subgroup of $\mathrm{GL}_8(\mathbb{C})$ generated by $H$ and the scalar matrices $\mathbb{C}^*$. Let $N(H')$ be its normalizer. Then $N(H')/H' \cong \mathrm{Sp}_6(\mathbb{F}_2)$ by \cite[Exercise 6.14]{BL}. Since every element of $H'$ acts on the 15 parenthesized polynomials in \eqref{eq:coblequartic} by scalar multiplication, the space spanned by these is preserved by the action of $N(H')$. This action factors through $H$, so we get a 15-dimensional representation of $N(H') / H$. This contains $\mathrm{Sp}_6(\mathbb{F}_2)$ as a subgroup, and this 15-dimensional representation of $\mathrm{Sp}_6(\mathbb{F}_2)$ is irreducible. There is a unique such irreducible representation, so we call it $U^{15}$. 
For our computations we shall use the larger
group $W(\mathrm{E}_7) = \mathrm{Sp}_6(\mathbb{F}_2) \times \{\pm 1\}$.
  This is the Weyl group of the root system of type $\mathrm{E}_7$. We use the same symbol
  $U^{15}$ to denote the corresponding irreducible representation of $W(\mathrm{E}_7)$.

Let $\mathfrak{h} \cong \mathbb{C}^7$ be the reflection representation of $W(\mathrm{E}_7)$. We will work in two different coordinate systems of $\mathfrak{h}$. The first one we call $c_1, \dots, c_7$ (with the standard quadratic form $\sum_i c_i^2$), which has the property that each $c_i=0$ is a reflection hyperplane. Explicitly, the following two matrices are a pair of generators for $W(\mathrm{E}_7)$ in terms of this basis:
\small \begin{equation}\label{equation:two_generators}
\mu \,=\, \begin{pmatrix}0& 0& 0& 0& 1& 0& 0\\ 0& 1& 0& 0& 0& 0& 0\\ 0& 0& 1& 0& 0& 0& 0\\ 0& 0& 0& 0& 0& 0& {-1}\\ 1& 0& 0& 0& 0& 0& 0\\ 0& 0& 0& 0& 0& 1& 0\\ 0& 0& 0& {-1}& 0& 0& 0 \end{pmatrix}
\qquad \nu \,=\, \begin{pmatrix}0& {-1}& 0& 0& 0& 0& 0\\ 0& 0& {-1}& 0& 0& 0& 0\\ 1/2& 0& 0& 1/2& {-1/2}& 0& 1/2\\ 1/2& 0& 0& 1/2& 1/2& 0& {-1/2}\\ 1/2& 0& 0& {-1/2}& 1/2& 0& 1/2\\ 1/2& 0& 0& {-1/2}& {-1/2}& 0& {-1/2}\\ 0& 0& 0& 0& 0& {-1}& 0\end{pmatrix}
\end{equation}
\normalsize
One can verify with {\tt GAP} \cite{gap} that they generate a group of the correct size, and it contains seven
reflections that satisfy the Coxeter relations of type $\mathrm{E}_7$. See our supplementary files.

All representations of the Weyl group $W(\mathrm{E}_7)$ are isomorphic to their duals, as is true for any real reflection group. In what follows we shall
distinguish between representations and their duals to avoid confusion.
The coordinate ring of $\mathfrak{h}$ is ${\rm Sym}(\mathfrak{h}^*) =
\C[c_1,c_2,c_3,c_4,c_5,c_6,c_7]$.
 There is a unique copy of the 15-dimensional irreducible representation $U^{15}$ in 
${\rm Sym}^7(\mathfrak{h}^*)$. We call this subspace $U^{15}_c$. It can be realized explicitly by taking the $W(\mathrm{E}_7)$-submodule spanned by the monomial $c_1 c_2 c_3 c_4 c_5 c_6 c_7$. This is an example of a {\em Macdonald representation}~\cite{macdonald}. 

Since $U^{15} \cong (U^{15})^*$, there exists a unique basis in $U^{15}_c$ which is dual to the basis of the 15 Heisenberg-invariant quartics of \eqref{eq:coblequartic} in $U^{15}$. 
We computed that dual basis of $U^{15}_c$. It was found to consist of the above $15$ polynomials $r,s_{001}, \ldots, t_{111}$ of degree $7$ in $c_1,c_2,\ldots,c_7$.
These elements of  ${\rm Sym}^7(\mathfrak{h}^*)$ are the coordinates of our 
rational map $\gamma \colon \PP^6 \dashrightarrow \PP^{14}$.

By multiplying the  dual bases and summing, we get a $W(\mathrm{E}_7)$-invariant 
in $U^{15} \otimes U^{15}_c$. This is the defining polynomial
$F_\tau$ of the Coble quartic, with coefficients as above.
This concludes our derivation of the map $\gamma$.
Here is an alternative way of getting our formulas.
    
\begin{remark} \label{rmk:gsw} \rm
Let $A$ be an $8$-dimensional complex vector space with basis 
\[
a_1=x_{000},\ a_2=x_{100},\ a_3=x_{010},\ a_4=x_{110},\ a_5=x_{001},\ a_6=x_{101},\ a_7=x_{011},\ a_8=x_{111}.
\] 
The Heisenberg group $H$ acts on $\bigwedge^4 A$ and the space of invariants has the following basis:
\begin{small}
\begin{align*}
h_1 \,\,=\,\, a_1 \wedge a_2 \wedge a_3 \wedge a_4 \,+\, a_5 \wedge a_6 \wedge a_7 \wedge a_8,\\
h_2 \,\,= \,\,a_1 \wedge a_2 \wedge a_5 \wedge a_6 \,+\, a_3 \wedge a_4 \wedge a_7 \wedge a_8,\\
h_3 \,\,= \,\,a_1 \wedge a_3 \wedge a_5 \wedge a_7 \,+\, a_2 \wedge a_4 \wedge a_6 \wedge a_8,\\
h_4 \,\,= \,\,a_1 \wedge a_4 \wedge a_6 \wedge a_7 \,+\, a_2 \wedge a_3 \wedge a_5 \wedge a_8,\\
h_5 \,\,= \,\,a_1 \wedge a_3 \wedge a_6 \wedge a_8 \,+\, a_2 \wedge a_4 \wedge a_5 \wedge a_7,\\
h_6 \,\,=\,\, a_1 \wedge a_4 \wedge a_5 \wedge a_8 \,+\, a_2 \wedge a_3 \wedge a_6 \wedge a_7,\\
h_7 \,\,=\,\, a_1 \wedge a_2 \wedge a_7 \wedge a_8 \,+\, a_3 \wedge a_4 \wedge a_5 \wedge a_6.
\end{align*}
\end{small}
The techniques in \cite[\S 6.2, \S 6.5]{gsw} show how to start with any vector in $\bigwedge^4 A$ and produce the equations for a Kummer variety in the projective space of lines in $A^*$ along with the equation for its Coble quartic. If we apply these techniques to $c_1 h_1 + \cdots + c_7 h_7$, then we get the above formula for $F_\tau$. The derivation of Conjecture~\ref{conj:uno} will discuss this in detail. \qed
\end{remark}

We now come to the proof that our parametrization of the
G\"opel variety is correct.

\begin{proof}[Proof of Theorem \ref{thm:24to1}]
The Macdonald representation $U^{15}_c$ in ${\rm Sym}^7(\mathfrak{h}^*)$
is unique. It defines a  unique $W(\mathrm{E}_7)$-equivariant  rational map
$$
\gamma \colon \PP^6 \dashrightarrow \PP^{14}.
$$
Hence,  up to choosing coordinates on the two projective spaces, the map $\gamma$ agrees with the one described in \cite[\S IX.7, Remark 7]{DO}. The closure of its image must be the  G\"opel variety~$\mathcal{G}$.

To see this more explicitly, we now come to our second coordinate system,
denoted $d = (d_1,d_2,d_3,d_4,d_5,d_6,d_7)$,
for the reflection representation $\mathfrak{h}$ of $W(\mathrm{E}_7)$.
It is defined by 
\begin{equation}
\label{eq:fromctod}
\begin{matrix}
c_1 &=& d_2+d_4+d_5, & & 
c_2 &=& d_1+d_4+d_7,\\
c_3 &=& d_2+d_3+d_7, & & 
c_4 &=& d_1+d_3+d_5,\\
c_5 &=& d_1+d_2+d_6, & & 
c_6 &=& d_5+d_6+d_7,\\
c_7 &=& d_3+d_4+d_6.
\end{matrix}
\end{equation}
The labels on the right hand side define a Fano configuration like (\ref{eq:fano}).
In these coordinates, the $63$ reflection hyperplanes of $W(\mathrm{E}_7)$ are given
by the $21$ linear forms $d_i - d_j$, 
the $35$ linear forms $d_i+d_j+d_k$, and the
$7$ linear forms $d_i + d_j +d_k + d_l + d_m + d_n$.
Here the indices are distinct. 
As a warning to the reader, we note
that the $d$ variables do not represent orthogonal coordinates
for $\mathrm{E}_7$ since the quadratic form $\sum_i c_i^2$ is not mapped to $\sum_i d_i^2$.

The Weyl group $W(\mathrm{E}_7)$ acts by Cremona transformations
on the space of configurations of seven points in $\PP^2$. 
Following Dolgachev and Ortland \cite[\S IX.7, Remark 7]{DO},
there exists a canonical $W(\mathrm{E}_7)$-equivariant birational map
from  $\PP(\mathfrak{h}) = \PP^6$ to that space. In coordinates, 
this canonical map takes $(d_1:d_2: \cdots: d_7)$ to the column
configuration of the matrix
\begin{equation}
\label{eq:threebyseven}
D \,\,\,= \,\,\,
\begin{pmatrix} 1 & 1 & 1 & 1 & 1 & 1 & 1 \\ d_1 & d_2 & d_3 & d_4 & d_5 & d_6 & d_7 \\ d_1^3 & d_2^3 & d_3^3 & d_4^3 & d_5^3 & d_6^3 & d_7^3 \end{pmatrix}.
\end{equation}
The rational map $\gamma$ is the composition of that birational morphism
and the G\"opel map of \cite[\S IX.7, Thm.~5]{DO},  written 
explicitly  in Section~\ref{sec:toricgopel}, that takes $7$-tuples in $\PP^2$ to points in~$\mathcal{G}$.

Our claim about $\gamma$ being a 24-to-1 cover follows from the
fact that a general net of cubics in $\PP^2$ contains $24$ cuspidal cubics \cite[\S IV]{Bra}.
Indeed, the ring of invariants for plane cubics is generated by two polynomials of degrees $4$ and $6$, and the condition for a plane cubic to have a cusp is that both of these invariants vanish \cite[\S 10.3]{dolgachev:invt}.
This implies that, given seven general points in $\PP^2$, there exist precisely
$24$ representations (\ref{eq:threebyseven}) of that configuration up to projective transformations. The net of cubics through these seven points defines
a morphism from the corresponding del Pezzo surface of degree $2$ onto
$\PP^2$. The branch locus is a quartic curve, and the $24$ cuspidal cubics
in that net correspond to the $24$ flexes of that quartic.
\end{proof}

\begin{remark} \rm
The coordinate system $d$  and the connection to cuspidal cubics are due to Bramble \cite{Bra}.
Bramble also gives an explicit formula for the plane quartic in terms of
$W(\mathrm{E}_7)$-invariant polynomials in $d_1,\ldots,d_7$ \cite[\S VI]{Bra} (there the coordinates are called $t_1, \dots, t_7$).
We first learned about the matrix $D$ from \cite[\S 6]{HKT}. \qed
\end{remark}

\begin{remark} \rm
Here is another way to see that the degree of $\gamma$ is $24$. Form the $2 \times 15$-matrix
$$ 
\begin{pmatrix}
\,\,\gamma(c) \,\,\\
\,\,\gamma(C)\,\,
\end{pmatrix}
\quad = \quad
\begin{pmatrix}
\,r \,& s_{001} & s_{010} & \,\cdots \,& s_{110} & s_{111} &
\,      t_{001} & t_{010} & \,\cdots\, &   t_{111}  \\
\,R \,& S_{001} & S_{010} &\, \cdots \,& S_{110} & S_{111} &
\,      T_{001} & T_{010} & \,\cdots\, & T_{111}  
\end{pmatrix}.
$$
The first row consists of our degree $7$ polynomials in $c_1,\ldots,c_7$,
and the second row is the image under $\gamma$ of a random point $C \in \PP^6$. The corresponding fiber of $\gamma$ is defined by the 
ideal of $2 \times 2$-minors. One needs to verify that its scheme is reduced and consists of $24$ points in 
$\PP^{6}$. In practice, this is done as follows.
With probability one, we have  $R \ne 0$, and we can consider the affine scheme defined by $r=1, s_{001} = S_{001}/R, \dots, t_{111} = T_{111}/R$. Using {\tt Macaulay~2} we verify that this is a $0$-dimensional scheme of degree $168$. This gives the number of solutions of $(c_1, \dots, c_7) \in \mathbb{A}^7$ that belong to the fiber of this point. But if $\lambda c_1, \dots, \lambda c_7$ were also a solution, we would need $\lambda^7 = 1$ because all of the functions used to define $\gamma$ are of degree $7$. Hence the degree of the fiber of $\gamma$ is $168/7 = 24$, as desired. \qed
\end{remark}

\begin{remark} \label{rmk:E7hyperplanes} \rm
As indicated in the proof of Theorem \ref{thm:24to1}, the map
$\gamma$ is closely related to a map from the moduli space of $7$ points in $\PP^2$ 
to the moduli space of plane quartics. Here,
the reflection hyperplanes correspond to the $63$ discriminantal conditions which give singular plane quartics. It follows that the non-hyperelliptic locus of the moduli space $\mathcal{M}_3(2)$ is the image 
% in the G\"opel variety $\mathcal{G} = \overline{\mathcal{A}_3(2)}$ 
 in the G\"opel variety $\mathcal{G} $
 % NEW022413
of the complement of the $63$ reflection hyperplanes in $\PP^6$. The fiber of size $24$ corresponds to a choice of flex point on the plane quartic.

In Section~\ref{sec:toricgopel} we shall define an embedding of the G\"opel variety $\mathcal{G}$ into a certain $35$-dimensional toric variety $\mathcal{T}$. It lives in $\PP^{134}$
and $W(\mathrm{E}_7)$ acts on this embedding by permuting coordinates.
That toric variety $\mathcal{T}$ and its polytope elucidate the combinatorial structure of 
% the compactification $\mathcal{G} = \overline{\mathcal{A}_3(2)}$.
our varieties.
% NEW022413
The image of the $63$ reflection hyperplanes gives 
the $63$ boundary divisors as the complement of 
the non-hyperelliptic locus of $\mathcal{M}_3(2)$ in 
$\mathcal{G}$. This determines the 
$63$ distinguished facets of the polytope of $\mathcal{T}$
mentioned in Theorem \ref{thm:toricvariety}.\qed
\end{remark}

After completing the present work, we learned that our
parametrization of the G\"opel variety, as well as some
of the material in Section \ref{sec:toricgopel}, had already been 
found by Colombo, van Geemen and Looijenga in \cite{CGL}.
In particular, the indeterminacy locus of our map
$\gamma \colon \PP^6 \dashrightarrow \mathcal{G}$
is determined in \cite[Proposition 4.18]{CGL}.
See also \cite[\S 1, Quartic curves]{CGL} for an interpretation of this result in terms of the geometric invariant theory of plane quartics. We now state their result, and we strengthen it by including the answer to \cite[Question 4.19]{CGL}.
This underlines the benefit of combining 
theoretical studies of moduli spaces with hands-on
computations.

Recall that a {\em flat} of a hyperplane arrangement is the intersection of some subset of the hyperplanes. The dimension of a flat will refer to the dimension of its projectivization.

\begin{theorem} {\rm (cf.~\cite{CGL})}
\label{thm:indeterminacy}
The indeterminacy locus of $\gamma$ is the reduced union of $315$ two-dimensional flats 
and $336$ one-dimensional flats of the reflection arrangement of $\mathrm{E}_7$ in $\mathbb{P}^6$.
 They correspond to the root subsystems of type $\mathrm{D}_4$ and $\mathrm{A}_5$
    listed in row 9 and row 15 of Table~\ref{fig:E7flats} in Section~\ref{sec:tropgeom}.
 The ideal  $\langle r, s_\bullet, t_\bullet \rangle$ is the intersection of these $651$ linear ideals.
   \end{theorem}

\begin{proof}
The set-theoretic version of this theorem was proved in 
\cite[Proposition 4.18]{CGL}. We established the ideal-theoretic statement 
with a {\tt Macaulay 2} computation that is posted in our supplementary files.
The following remark offers a more detailed explanation.
\end{proof}

\begin{remark} \rm
The indeterminacy locus of $\gamma$ is the zero set in $\PP^6$ of
the ideal $\langle r, s_\bullet, t_\bullet \rangle = \langle r,s_{001}, \ldots,
s_{111}, t_{001}, \ldots,  t_{111}\rangle$ in the polynomial ring 
$\mathbb{Q}[c_1,c_2,\ldots,c_7] = \mathbb{Q}[d_1,d_2,\ldots,d_7]$.
Theorem \ref{thm:indeterminacy} states that this ideal is radical,
and that it is the intersection of $651 = 315 + 336$ ideals 
which are generated by linear forms. One of these components is the 
height $4$ prime
\begin{equation}
\label{eq:12hyperplanes}
\begin{matrix}
 \bigl\langle
 d_1{-}d_2, \,d_1{-}d_3,\,\,
 d_1{+}d_4{+}d_5\,, d_2{+}d_4{+}d_5\,, d_3{+}d_4{+}d_5\,
 , d_1{+}d_6{+}d_7, \,d_2{+}d_6{+}d_7, \,d_3{+}d_6{+}d_7, \\ \quad d_2{-}d_3,\,\,
 d_2{+}d_3{+}d_4{+}d_5{+}d_6{+}d_7, \,d_1{+}d_3{+}d_4{+}d_5{+}d_6{+}d_7, \,
 d_1{+}d_2{+}d_4{+}d_5{+}d_6{+}d_7
 \bigr\rangle.
 \end{matrix}
 \end{equation}
The $12$ listed generators form a root subsystem of type $\mathrm{D}_4$ in the root system $\mathrm{E}_7$. There are $315$ such subsystems in $\mathrm{E}_7$.
By \cite[\S 4]{manivel}, they are in  bijection with the
syzygetic triples of Steiner complexes, and also with the isotropic planes in 
$(\F_2)^6$. The latter will be discussed in Section~\ref{sec:toricgopel},
and a census of all root subsystems will be given in Table \ref{fig:E7flats}.
Each root subsystem $\mathrm{D}_4$ consists of $12 $ of the $63$ 
reflection hyperplanes for $\mathrm{E}_7$, and these
intersect in a $\PP^2$ inside $\PP^6 = \PP(\mathfrak{h})$.
The rank $4$ matroid  given by (\ref{eq:12hyperplanes})
 is  the {\em Reye configuration} in \cite[Figure 2]{manivel}.
 
The other $336$  components of $\langle r, s_\bullet, t_\bullet \rangle$
correspond to an orbit of root subsystems of type $\mathrm{A}_5$. Each such prime ideal is linear of height $5$,
and it contains $15$ of the $63$ linear forms.

To verify the claim,  we intersect all $651$ linear ideals in {\tt Macaulay 2}. We note that a naive intersection may run for several hours and still not terminate, but a careful choice of ordering will finish in a matter of seconds. First, we use the $c$-coordinates rather than the $d$-coordinates.
Each of our ideals contains at least one of the variables $c_i$. 
We separate them into seven groups of $93$ based
on which $c_i$ they contain (there is some redundancy). Then we intersect each of these groups.
Finally, we intersect the seven resulting ideals, and
the result is $\langle r, s_\bullet, t_\bullet \rangle$. We also used this procedure to get the ideals of the $\mathrm{D}_4$ flats and the $\mathrm{A}_5$ flats. We list the graded Betti tables for the indeterminacy locus of $\gamma$, the reduced union of the 315 flats of type $\mathrm{D}_4$, and the reduced union of the 336 flats of type $\mathrm{A}_5$, respectively:
\begin{multicols}{3}
\footnotesize
\begin{Verbatim}[samepage=true]
       0  1  2   3   4  5
total: 1 15 84 168 126 28
    0: 1  .  .   .   .  .
       ...
    6: . 15  .   .   .  .
    7: .  .  .   .   .  .
    8: .  .  .   .   .  .
    9: .  . 84   .   .  .
   10: .  .  . 168 105 21
   11: .  .  .   .   .  .
   12: .  .  .   .  21  .
   13: .  .  .   .   .  7
\end{Verbatim}
\begin{Verbatim}[samepage=true]
       0  1   2   3   4  5
total: 1 85 210 315 210 21
    0: 1  .   .   .   .  .
       ...
    6: . 15   .   .   .  .
    7: .  .   .   .   .  .
    8: . 70 210   .   .  .
    9: .  .   . 315 210  .
   10: .  .   .   .   . 21
\end{Verbatim}
\begin{Verbatim}[samepage=true]
       0  1   2   3   4   5
total: 1 36 315 595 420 105
    0: 1  .   .   .   .   .
    1: .  .   .   .   .   .
    2: .  .   .   .   .   .
    3: .  .   .   .   .   .
    4: . 21   .   .   .   .
    5: .  .   .   .   .   .
    6: . 15 315 595 420 105 
\end{Verbatim}
\end{multicols}

At this point, it is important to note that  $\langle r, s_\bullet, t_\bullet \rangle$ has an alternative generating set, consisting of $135$ products of linear forms
defining reflection hyperplanes. We shall present them in 
Section~\ref{sec:toricgopel}. This fact makes it obvious
that the variety consists of flats of $\mathrm{E}_7$. Colombo {\it et al.} 
used these $135$ generators  to prove the set-theoretic result
in \cite[Prop.~4.18]{CGL}. \qed
\end{remark}

\section{Equations defining the G\"opel variety} \label{sec:gopeleqns}

In this section we first determine the equations and graded Betti numbers of the G\"opel variety $\mathcal{G}$.
Subsequently, we compute the prime ideal of  the universal Coble quartic over $\mathcal{G}$.
This is the $12$-dimensional subvariety in $\PP^{14} \times \PP^7$
whose fibers over $\mathcal{G}$ are the quartics  (\ref{eq:coblequartic}).

\begin{theorem}
\label{thm:gopel}
The six-dimensional G\"opel variety $\mathcal{G}$ has degree $175$ in $\mathbb{P}^{14}$. 
The homogeneous coordinate ring of $\mathcal{G}$ is Gorenstein, it has the Hilbert series 
\begin{equation}
\label{eq:hilbertgopel}
\frac{1 + 8 z + 36 z^2 + 85 z^3 + 36  z^4 + 8 z^5 + z^6}{(1-z)^7},
\end{equation}
and its defining prime ideal  is minimally generated by $35$ cubics and $35$ quartics. The graded Betti table of this ideal in the polynomial ring  $\mathbb{Q}[r,s_{001}, \ldots,t_{111}]$ in $15$ variables equals
\tt \small \begin{Verbatim}[samepage=true]
       0  1   2    3    4    5   6  7 8
total: 1 70 609 1715 2350 1715 609 70 1
    0: 1  .   .    .    .    .   .  . .
    1: .  .   .    .    .    .   .  . .
    2: . 35  21    .    .    .   .  . .
    3: . 35 588 1715 2350 1715 588 35 .
    4: .  .   .    .    .    .  21 35 .
    5: .  .   .    .    .    .   .  . .
    6: .  .   .    .    .    .   .  . 1
\end{Verbatim}
\end{theorem}

\normalsize
\begin{remark} \rm
Before giving the proof, let us point out some geometric consequences of this theorem. 
Since the degree of the numerator is less than the degree of the denominator in the Hilbert series, we see that the Hilbert function agrees with the Hilbert polynomial for all nonnegative inputs. One can calculate directly that the Hilbert polynomial of $\mathcal{G}$ is 
\[
\frac{35}{144} t^{6}+\frac{35}{48} t^{5}+\frac{287}{144} t^{4}+\frac{133}{48} t^{3}+\frac{343}{72} t^{2}+\frac{7}{2} t+1.
\]

We see from the self-duality of the Betti table that the canonical bundle of $\mathcal{G}$ is 
\[
\omega_{\mathcal{G}} \,\,= \,\,\mathcal{E}xt^8_{\mathcal{O}_{\PP^{14}}}(\mathcal{O}_{\mathcal{G}}, \mathcal{O}_{\PP^{14}}(-15)) \,\,= \,\, \mathcal{O}_{\mathcal{G}}(-1).
\]
Hence $\mathcal{G}$ is a Fano variety. Since $\mathcal{G}$ is arithmetically Cohen--Macaulay, and the Hilbert polynomial and the Hilbert function agree, this implies 
$\mathrm{H}^i(\mathcal{G}, \mathcal{O}_{\mathcal{G}}(d)) = 0$
 if either $d \ge 0$ and $i > 0$ or if $d < 0$ and $i < 6$ \cite[Cor.~A1.15]{geom_syz}.
\qed
\end{remark}

\begin{proof}[Proof of Theorem~\ref{thm:gopel}]
Our point of departure for the derivation of Theorem \ref{thm:gopel}
is \cite[\S IX.5, Proposition 8]{DO}
and the subsequent corollary which gives a list of $63$ cubics.
We now review this construction.
For each of the $63$ non-zero half-periods 
$\epsilon/2 \in  \Lambda/2 \Lambda$, we  consider the linear involution on $\mathbb{P}^7$
that is induced from the action on second order theta functions
seen in (\ref{eq:involution}).
The fixed point set  of this involution on $\mathbb{P}^7$ is the union of
two three-dimensional planes $H_\epsilon^+$ and $H_\epsilon^-$,
and each of these planes intersects the Coble quartic in a Kummer surface.
The relations  (\ref{eq:cubicrel}) on the coefficients of these two Kummer surfaces
give the same cubic relation on $r,s_{ijk}, t_{ijk}$. This cubic lies in the
ideal of the G\"opel variety~$\mathcal{G}$.

For a concrete example consider $\epsilon = (1,0,0)$.
The corresponding involution (\ref{eq:involution}) fixes the coordinates
$x_{000}, x_{001}, x_{010}, x_{011}$ and it switches the sign on
$x_{100}, x_{101}, x_{110}, x_{111}$. The $3$-planes
$H_\epsilon^{\pm}$ are obtained by setting either of these
two groups of four variables to zero. 
This specialization erases ten of the $15$ summands in the formula
for $F_\tau$ in (\ref{eq:coblequartic}).
What remains are the terms with coefficients
  $r,s_{001}, s_{010}, s_{011}, t_{100}$.
  This represents a quartic Kummer surface in $H_\epsilon^+$
  resp.~in $H_\epsilon^-$. The relation (\ref{eq:cubicrel}) holds
  for that Kummer surface.  The argument in the previous paragraph shows
   that the following cubic lies in the ideal of~$\mathcal{G}$:
  $$ 16 r^3 + r t_{100}^2+4s_{001} s_{010} s_{011}- 4 r s_{001}^2-4 r s_{010}^2
-4 r s_{011}^2 .$$

For a more complicated example consider the half-period
$\epsilon = \tau \cdot (1,0,0)^t$.
The three-plane
$H_\epsilon^-$ is defined by
$x_{000}-x_{100} =  x_{001} - x_{101} =   x_{010}-x_{110} =  x_{011}-x_{111} = 0$,
 and the three-plane $H_\epsilon^+ $ is defined by
$x_{000}+x_{100} =  x_{001} + x_{101} =   x_{010}+x_{110} =  x_{011}+x_{111} = 0$.
The cubic relation on the Kummer surface obtained by restricting $F$ to 
$H_\epsilon^-$ or $H_\epsilon^+$ equals
\begin{align*}
16(s_{011}+2r)^3+(s_{011}+2r)(2t_{110}+2t_{101}+2t_{001}+2t_{010})^2 \\
+ 4(2s_{001}+2s_{010}+t_{100})(2s_{100}+t_{011}+2s_{111})(2s_{101}+t_{111}+2s_{110}) \\
-4(s_{011}+2r)(2s_{001} + 2s_{010}+t_{100})^2
-4(s_{011}+2r)(2s_{100} + t_{011} + 2s_{111})^2 
\\ -4(s_{011} + 2r)(2s_{101} + t_{111} + 2s_{110})^2.
\end{align*}
In this manner we obtain $63$ cubic equations for the
G\"opel variety $\mathcal{G}$, as stated in \cite[Corollary on p.~186]{DO}.
However, only
$35$ of these are linearly independent. A vector space basis consists of
those cubics that come from the $35$ nonzero even theta characteristics.  

In addition to the $35$ cubics constructed from Kummer surfaces as above,
 the  ideal of $\mathcal{G}$ has $35$  minimal generators of degree~$4$. 
 Here is an example of such a quartic generator:
$$ \begin{matrix}
48 r^2 s_{101} s_{110}-12 s_{011}^2 s_{101} s_{110} -12 s_{100}^2 s_{101} s_{110}
-4 s_{001}^2 s_{100} s_{111}-4 s_{010}^2 s_{100} s_{111}+8 s_{011}^2 s_{100} s_{111} +  \\
4 s_{100}^3 s_{111} 
{-}16 r s_{001} s_{101} s_{111}+8 s_{010} s_{011} s_{101} s_{111}+8 s_{100} 
s_{101}^2 s_{111} {-}16 r s_{010} s_{110} s_{111}
{+}8 s_{001} s_{011} s_{110} s_{111} \\ + 
8 s_{100} s_{110}^2 s_{111}+8 s_{001} s_{010} s_{111}^2 {-}
16 r s_{011} s_{111}^2 
{-} 4 s_{101} s_{110} s_{111}^2{-}4 s_{100} s_{111}^3
{-}s_{100} s_{111} t_{001}^2-s_{100} s_{111} t_{010}^2 \\ -2 s_{001} s_{010} t_{011}^2
+4 r s_{011} t_{011}^2+s_{101} s_{110} t_{011}^2+s_{100} s_{111} t_{011}^2 
{+}s_{100} s_{111} t_{100}^2 {+} s_{100} s_{111} t_{101}^2 
+s_{100} s_{111} t_{110}^2 \\
-s_{100} t_{011} t_{101} t_{110} 
 +s_{110} t_{010} t_{101} t_{111}+s_{101} t_{001} t_{110} t_{111}+2 s_{001} s_{010} t_{111}^2
-4 r s_{011} t_{111}^2
-2 s_{101} s_{110} t_{111}^2
\end{matrix}
$$
The explanation for the derivation of such quartics will come in the proof of Theorem~\ref{thm:135}.
 From the $70$ minimal generators,  a Gr\"obner basis
 for the ideal of $\mathcal{G}$ can be computed in
 {\tt Macaulay~2}, and from this one finds the degree $175$
 and the Hilbert series in Theorem~\ref{thm:gopel}.

The group ${\rm Sp}_6(\mathbb{F}_2)$ has exactly two distinct irreducible
representations of dimension $35$, and these two representations
are given respectively by the cubic generators and the quartic generators
of our ideal $\mathcal{G}$. In particular, we can obtain
all $70$ generators by lifting the two $6 \times 6$-matrices
$\mu $ and $\nu$ in Section~\ref{sec:param:gopel} to
$\mathbb{Q}[r,s_{001}, \ldots, t_{111}]$ and
applying these to one representative cubic and one 
representative quartic, for instance those displayed above. 
In Section~\ref{sec:toricgopel} we shall present an alternative model for
these two ${\rm Sp}_6(\mathbb{F}_2)$-submodules $\mathbb{Q}[r,s_{001}, \ldots, t_{111}]$, namely in terms of binomials in a polynomial ring in $135$ unknowns.

It remains to be proved that the $35$ cubics and the $35$ quartics actually
generate the prime ideal of $\mathcal{G}$. This is what we shall explain next.
The $70$ generators, as well as explicit matrix representations for $\mu$ and $\nu$ on the 15 Coble coefficients, are available in our supplementary materials,
and the reader can use these to verify all of our calculations in {\tt Macaulay~2}.

The main idea is to show that the scheme defined by this ideal is generically reduced and satisfies Serre's criterion $(S_1)$. These conditions ensure that our ideal is radical by \cite[Exercise 11.10]{eisenbud}.
By the earlier results of Coble \cite[\S 49]{Cob} and Dolgachev--Ortland \cite[\S IX.7, Theorem 5]{DO}, we already know that the cubics  alone cut out $\mathcal{G}$ set-theoretically, and hence
the variety of our ideal is irreducible of codimension eight.
So, to show `generically reduced'
only requires showing that at some point on ${\mathcal{G}}$, the Jacobian matrix has rank $8$.
This is easily done by substituting randomly chosen points from $\PP^6$ into the map $\gamma$.

To prove Serre's condition  $(S_1)$, we show that our ideal defines an arithmetically Cohen--Macaulay scheme.
This is done by adding seven random linear forms to the ideal generated by the cubics and quartics. The resulting ideal defines an artinian quotient of
$\mathbb{Q}[r,s_{001},\ldots, t_{111}]$, and we verified that its Hilbert series
is the numerator seen in (\ref{eq:hilbertgopel}). The computation shows that the seven linear forms are a regular sequence modulo our ideal, and hence we get that the scheme defined by the 35 cubics and 35 quartics is arithmetically Cohen--Macaulay. We remark that the regular sequence computation is intensive and took approximately 15 minutes to perform. An alternative method is to work over the finite field $\Z/101$. The Hilbert series remains unchanged, but the regular sequence computation takes only a few seconds to perform. This implies the result via standard flatness and semicontinuity arguments.

At this point  we have shown that the ideal generated by our cubics and quartics is prime.
To infer the Gorenstein property, we use Stanley's result \cite[Theorem 4.4]{stanley}, which states that a Cohen--Macaulay domain whose Hilbert series has a palindromic numerator is Gorenstein.
 
 We finally derive the Betti table asserted in Theorem \ref{thm:gopel}.
We have already shown that  the G\"opel variety ${\mathcal{G}}$ is arithmetically Gorenstein, which implies that the Betti table is symmetric \cite[Corollary 21.16]{eisenbud}. Multiplying the numerator of the Hilbert series  (\ref{eq:hilbertgopel})
by $(1-z)^{8}$ gives the graded Euler characteristic
(or~{\em K-polynomial}\,) of the minimal free resolution of $\mathcal{G}$. By the Cohen--Macaulay property, the Castelnuovo--Mumford regularity of ${\mathcal{G}}$ is the degree of the numerator of the Hilbert series, which is 6. Therefore, we have the following partial information about the Betti numbers:
\small
\begin{Verbatim}[samepage=true]
       0  1  2 3 4 5  6  7 8
total: 1 70  ? ? ? ?  ? 70 1
    0: 1  .  . . . .  .  . .
    1: .  .  . . . .  .  . .
    2: . 35 21 c f e  b  . .
    3: . 35  a d g d  a 35 .
    4: .  .  b e f c 21 35 .
    5: .  .  . . . .  .  . .
    6: .  .  . . . .  .  . 1
\end{Verbatim}
\normalsize
The unknown entries  satisfy the  linear equations
$$
{\tt a}-{\tt c} \,=\,588\,,\,\,
{\tt b}-{\tt d}+{\tt f} \,=\, -1715 \,,\,\,
-2{\tt e}+{\tt g} \,=\, 2350.
$$
Now, there are $21$ linear syzygies on the $35$ cubic generators,
and we compute the corresponding $35 \times 21$-matrix
whose entries are linear forms in $\mathbb{Q}[r,s_{001}, \ldots,t_{111}]$.
Evaluating this matrix at a random point in $\mathbb{Q}^{15}$ shows that it has full rank $21$
over the rational function field $\mathbb{Q}(r,s_{001}, \ldots,t_{111})$.
Hence its columns are linearly independent.
This implies ${\tt c} = 0$, and consequently
${\tt f} = {\tt e} = {\tt b} = 0$. The above linear equations 
then imply ${\tt a} = 588$,
${\tt d} = 1715$, and ${\tt g} = 2350$, so
the Betti table is as claimed.
This concludes our proof of Theorem \ref{thm:gopel}.
\end{proof}

Now we study the {\em universal Coble quartic} $\mathcal{C}$ in $\mathcal{G} \times \PP^7$. This
twelve-dimensional
variety is defined as the closure in $\PP^{14} \times \PP^7$ of the set of pairs
$\bigl( (r:s_{001}:\dots:t_{111}), (x_{000}: \dots:x_{111}) \bigr)$ 
such that $(r:s_{001}:\dots:t_{111})$ is a point in the non-hyperelliptic locus of 
% $\mathcal{A}_3(2) \subset \mathcal{G}$, 
$\mathcal{M}_3(2)$ in $\mathcal{G}$, 
% NEW022413
and the pair is a point on the hypersurface
of bidegree $(1,4)$ that is given by the polynomial \eqref{eq:coblequartic}. 

\begin{corollary}
\label{cor:gucq}
The bihomogeneous prime ideal of the universal Coble quartic in $\PP^{14} \times \PP^7$ is minimally
generated by $71$ polynomials, namely the $35+35$ equations
of bidegrees $(3,0)$ and $(4,0)$ for 
 the G\"opel variety as in Theorem~\ref{thm:gopel}, along with the bidegree $(1,4)$ 
 equation~\eqref{eq:coblequartic}.
\end{corollary}

\begin{proof}
The  equation \eqref{eq:coblequartic}
cuts out a codimension one subscheme $\mathcal{C}'$ of $\mathcal{G} \times \PP^{7}$. Since $\mathcal{G}$ is arithmetically Cohen--Macaulay in its embedding in $\PP^{14}$, this implies that \eqref{eq:coblequartic} is a non-zerodivisor on $\mathcal{G} \times \PP^{7}$,  and that the scheme $\mathcal{C}'$ is arithmetically Cohen--Macaulay as well.
 To show that our ideal is prime, it suffices 
  to show that it is irreducible and generically reduced
  (as explained in the proof of Theorem~\ref{thm:gopel}).
  This will imply    $\mathcal{C}' = \mathcal{C}$.

Consider the projection $\pi \colon \mathcal{C}' \to \mathcal{G}$. The fibers of $\pi$ have constant dimension $6$ since the Coble coefficients cannot simultaneously vanish on any point of $\mathcal{G}$. Let $\mathcal{U} \subset \mathcal{M}_3(2)$ denote the non-hyperelliptic locus. Over the subvariety $\mathcal{U} \subset \mathcal{G}$, the fibers of $\pi$ are Coble quartics, and hence irreducible. Since $\mathcal{C}'$ is arithmetically
Cohen--Macaulay, all of its irreducible components have the same dimension, namely $12$. The closure of $\pi^{-1}(\mathcal{U})$ gives an irreducible component. The complement of $\pi^{-1}(\mathcal{U})$ in $\mathcal{C}'$ is the preimage of $\mathcal{G} \backslash \mathcal{U}$. 
But $\dim \pi^{-1}(\mathcal{G} \backslash \mathcal{U}) \le 11$, so it cannot be an irreducible component.
Hence the variety $\mathcal{C}'$ is irreducible.

Now consider the Jacobian matrix of the $71$ equations defining $\mathcal{C}'$.
  Since the $70$ equations for $\mathcal{G}$ involve no ${\bf x}$ coordinate, this matrix
       can be put in block triangular form. Choose any point 
$\bigl( (r : s_{001} : \cdots : t_{111}), (x_{000} :  \cdots : x_{111})\bigr)$,
where $(r : s_{001} : \cdots  :  t_{111}) \in \mathcal{U}$ is a non-singular point of $\mathcal{G}$, 
and $(x_{000} :  \cdots :  x_{111})$ is a non-singular point on its Coble quartic.
 Then this is a non-singular point of $\mathcal{C}'$. Hence the scheme $\mathcal{C}'$ defined by
the $71$ proposed equations is also generically reduced. We conclude that
they generate the prime ideal of~$\mathcal{C}= \mathcal{C'}$.
\end{proof}

\section{A toric variety for seven points in $\PP^2$} \label{sec:toricgopel}

In this section we present an alternative embedding of the G\"opel variety, with beautiful
combinatorics. We learned this from Dolgachev and Ortland  \cite[\S IX.7]{DO}.
Here, the G\"opel variety sits in the high-dimensional 
projective space $\PP^{134}$ whose coordinates are the $135$ G\"opel functions.
The corresponding ideal  is generated by
linear trinomials,  cubic binomials and quartic binomials. We will first describe the combinatorial structure of the
G\"opel functions.
The paper \cite{RSS} features  the analogous toric varieties 
for smaller Macdonald representations.

Consider the six-dimensional vector space
$(\F_2)^6$ over the two-element field $\F_2$.
We fix the following non-degenerate symplectic form 
on this $64$-element vector space:
\begin{equation}
\label{eq:inprod}
 \langle x, y \rangle \quad = \quad
x_1 y_4 + x_2 y_5 + x_3 y_6 + x_4 y_1 + x_5 y_2 + x_6 y_3 .
\end{equation}
A linear subspace $V \subset (\F_2)^6$ is {\em isotropic}
if $\langle x,y \rangle = 0$ for all $x,y \in V$.
Clearly, all subspaces of dimension $\leq 1$ are isotropic, and
there are no isotropic subspaces of dimension $\geq 4$.
There are precisely $315$ isotropic subspaces of dimension two,
and $135$ isotropic subspaces of dimension three.
The former are called {\em isotropic planes}, and the latter
are called  {\em Lagrangians}. There are
 $63$ non-zero vectors in $(\F_2)^6$, and each 
 Lagrangian contains seven of these.  Each
 isotropic plane contains three of these, and is
 contained in precisely three Lagrangians.

The root system $\mathrm{E}_7$ consists of $126$ vectors
in a seven-dimensional inner product space. If we
take this space to be the hyperplane in $\R^8$ given by
coordinate sum zero, then the $63 = 7 + \binom{7}{2} + \binom{7}{3}$ positive roots are
$e_8 - e_i$ for $1 \le i \le 7$, 
$e_i - e_j$ for $1 \leq i < j \leq 7$,
and 
$\frac{1}{2}(e_8 + \sum_{i \in \sigma} e_i - \sum_{j \not\in \sigma} e_j)$
where $\sigma$ runs over all
three-element subsets of $\{1,2,\ldots,7\}$. We label the
positive roots of ${\rm E}_7$ by the pairs in $\{1,2,
\ldots,8\}$ and the triplets in $\{1,2,\ldots,7\}$.

\begin{table}
\begin{center} \begin{tabular}{c||c|c|c|c|c|c|c|c}
& 000 & 100 & 010 & 110 & 001 & 101 & 011 &111\\
\hline\hline
 000 & & 236  & 345 & 137 & 467 & 156 & 124 & 257 \\
 100 & 237 & 67 & 136 & 12 & 157 & 48 &  256 & 35 \\
 010 &   245 & 127 & 23 & 68 & 134 & 357 & 15 & 47 \\
 110 & 126 & 13 & 78 & 145 & 356 & 25 & 46 & 234 \\
 001 & 567 & 146 & 125 & 247 & 45 & 17 & 38 & 26 \\
  101 &  147 & 58 & 246 & 34 & 16 & 123 & 27 & 367 \\
   011 & 135 & 347 & 14 & 57 & 28 & 36 & 167 & 456 \\
  111 & 346 & 24 & 56 & 235 & 37 & 267 & 457 & 18 
\end{tabular} \end{center}
\caption{Cayley's bijection between lines in $(\F_2)^6$ and positive roots of $\mathrm{E}_7$.}
\label{fig:cayleytable}
\end{table}

We copied Table~\ref{fig:cayleytable} from Cayley's paper \cite{Cay}. It
fixes a particular bijection between the positive roots of $\mathrm{E}_7$
and the vectors in $(\F_2)^6 \backslash \{0\}$.
Its rows represent the first three coordinates
and its columns represent the last three coordinates of
a vector in $(\F_2)^6$. For instance, the triple $247$ corresponds to the
vector $ (0,0,1,1,1,0)$, and the pair $24$ corresponds to 
$(1,1,1,1,0,0)$.
This bijection has the property in \cite[Lemma IX.8]{DO}:
two positive roots in $\mathrm{E}_7$ are perpendicular in $\R^8$
if and only if the corresponding vectors $x,y \in (\F_2)^6$
satisfy $\langle x,y \rangle = 0$ in $\F_2$.
 Cayley had constructed his table in the 1870s to have
precisely this property.

The orthogonality preserving bijection between 
$(\F_2)^6 \backslash \{0\}$ and the positive roots of $\mathrm{E}_7$ 
shows that the Weyl group of $\mathrm{E}_7$ (modulo its
two-element center) is isomorphic to the 
group  ${\rm Sp}_6(\F_2)$ of $6 \times 6$-invertible
matrices over $\F_2$ that preserve
the symplectic form in \eqref{eq:inprod} (see \cite[\S VI.4, Exercice 3]{bourbaki}).
The order of that group is easily seen (cf.~\cite[\S II.23, (7)]{Cob}) to be
\begin{equation}
\label{eq:1451520}
|{\rm Sp}_6(\F_2)| \,\,\, = \,\,\,
 (2^6-1)\cdot 2^5 \cdot (2^4-1) \cdot 2^3 \cdot (2^2-1) \cdot 2^1 
\,\,\, = \,\,\, 36 \cdot 8 !  \,\, = \,\, 1451520. 
\end{equation}

With the relabeling given by Cayley's table, the $135$ Lagrangians
fall into two classes with respect to permutations of the set
$\{1,2,3,4,5,6,7\}$.
First, there are $30$ {\em Fano configurations}
\begin{equation}
\label{eq:fano}
 \bigl\{ 124, \, 235, \, 346, \, 457, \, 561, \, 672,\, 713 \bigr\}.
 \end{equation}
 We denote this configuration by $f_{1234567}$.
 Second, there are $105$ {\em Pascal configurations} like
\begin{equation}
\label{eq:pascal}  \bigl\{ 12, 34, 56, 78, 127, 347, 567 \}.
\end{equation}
We denote this configuration by $g_{7123456}$.
If the configurations are changed by a permutation of
$\{1,2,3,4,5,6,7\}$ then the indices of the label $f_\bullet$ or $g_\bullet$
are permuted accordingly. Hence the labeling is not unique. For example, the Fano
configuration $f_{1234567}$ in (\ref{eq:fano}) can also be labeled $f_{2345671}$, and the 
Pascal configuration  $g_{7123456}$ in (\ref{eq:pascal}) can also be labeled $g_{7345612}$.
Some permutations give rise to a sign change in
$f_\bullet$ or $g_\bullet$.

 The $315$ isotropic planes in $(\mathbb{F}_2)^6$ fall into three classes.
Each class has cardinality $105$. Using
the labeling given by Cayley's Table~\ref{fig:cayleytable},
 representatives of these three classes are
\begin{equation}
\label{eq:planes}
 \{12,\,123,\,38\} \quad \hbox{and} \quad
 \{123, \,145, \, 167 \} \quad \hbox{and} \quad
\{ 12,\, 34 ,\, 567 \}. 
\end{equation}

Consider the six-dimensional variety of unlabeled configurations of seven points
in the projective plane $\PP^2$. There are natural correspondences, described
in \cite[\S IX]{DO}, that take such configurations to Cayley octads,
hence to plane quartic curves, and hence to abelian threefolds. Here is how
to make this completely explicit.
We associate with each  Lagrangian in $(\mathbb{F}_2)^6$
 a homogeneous symmetric polynomial in the brackets $[ijk]$, which are the $3 \times 3$-minors
of the $3 \times 7$-matrix of homogeneous coordinates on $(\PP^2)^7$.
These bracket polynomials are called {\em G\"opel functions} in \cite{Cob, DO}.
Each of the $30$ Fano configurations translates verbatim into a bracket monomial.
For instance, the Fano configuration (\ref{eq:fano}) translates into
\begin{equation}
\label{eq:fano2}
f_{1234567} \quad = \quad   [124] [235][346][457][561][672][713].
 \end{equation}
  The $105$ Pascal configurations represent  six points lying on a conic, 
  e.g.~let $Q_7$ denote the quartic bracket polynomial
that vanishes when the points $1,2,3,4,5,6$ lie on a conic:
\begin{equation}
\label{eq:pascal3}
Q_7 \quad = \quad  
[134][156][235][246] -  [135][146][234][256].
\end{equation}
Then the Pascal configuration (\ref{eq:pascal}) translates into the bracket binomial
\begin{equation}
\label{eq:pascal2}
 g_{7123456} \quad = \quad  [127][347][567] \cdot Q_7 . 
 \end{equation}
We have thus defined a coordinate system on $\PP^{134}$, 
consisting of $30$ coordinates $f_{\bullet}$ and 
$105$ coordinates  $g_{\bullet}$, and we have defined
a rational map 
\begin{align} \label{eqn:gopelmap}
(\PP^2)^7  \dashrightarrow \PP^{134}
\end{align}
whose coordinates are multi-homogeneous bracket polynomials of degree  
$(3,3,3,3,3,3,3)$. This map factors through the six-dimensional space of configurations of seven points in $\PP^2$.
 
One desirable property of this coordinate system is that it fits well with $\mathrm{E}_7$. Consider the action of the symplectic group ${\rm Sp}_6(\mathbb{F}_2)$ on the $135$ Lagrangians in $(\mathbb{F}_2)^6$.
Then the induced action on our $135$ G\"opel coordinates
$f_{\bullet}$ and   $g_{\bullet}$ is realized by signed permutations.

Dolgachev and Ortland \cite[\S IX.7, Proposition 9]{DO} describe 
the following $315$ linear relations among the G\"opel functions.  For each of the $315$ isotropic planes, there is a linear relation among the three G\"opel coordinates indexed by the three Lagrangians containing 
that plane.  For instance, the three isotropic planes in (\ref{eq:planes}) determine the linear relations
$$
\begin{matrix}
g_{3124567}-g_{3124657}+g_{3124756} & 
g_{1234567}-f_{1243765}+f_{1243675}  &
g_{5123467}-g_{6123457}+g_{7123456}  \\
& & \\
\{12,38,123,45,67,345,367\} & 
\{123,145,167,18,23,45,67\} & 
\{12,34,567, 58,67,125,345\}
 \\
\{12,38,123,46,57,346,357\}  & 
\{123,\!145, \! 167, \! 247, \! 256, \! 346, \! 357\} & 
\{12,34,567, 68,57,126,346\}
\\
\{12,38,123,47,56,347,356\} & 
\{123, \! 145, \! 167, \! 246, \! 257, \! 347, \! 356\} & 
\{12,34,567,78,56, 127,347\}
\\
\end{matrix}
$$ 
In each column we list the corresponding triple of Lagrangians.
These linear forms arise from  the three-term quadratic Pl\"ucker relations among the brackets $[ijk]$. These linear forms span a $120$-dimensional space, so they cut out a $14$-dimensional linear subvariety $L \subset \mathbb{P}^{134}$. Thus, the rational map \eqref{eqn:gopelmap} factors through this linear space:
\begin{equation}
\label{eq:7to14to134}
(\mathbb{P}^2)^7\dashrightarrow L\subset \mathbb{P}^{134}.
\end{equation}

Returning to the Coble quartic $F_\tau$ in (\ref{eq:coblequartic}),
its $15$ coefficients span the same irreducible ${\rm Sp}_6(\F_2)$-module as the
$135$ G\"opel functions.
 Coble states in 
 \cite[\S IV.49,(16)]{Cob} that
\[
r,s_{001},s_{010},s_{011},s_{100},s_{101},s_{110},s_{111},t_{001},t_{010},
t_{011}, t_{100}, t_{101} , t_{110}, t_{111}
\]
can be expressed as integer linear combinations of the
G\"opel functions $f_{\bullet{}}$ and $g_{\bullet{}}$.

Constructing such linear combinations turned out to be
a non-trivial undertaking, but we succeeded in obtaining 
formulas for writing $r,s,t$ in the G\"opel coordinates $f_\bullet$, $g_\bullet$ by following exactly the derivation described by Coble \cite[\S IV.49,(16)]{Cob}. Combining with $120$ independent linear trinomial relations among the G\"opel functions described above, we  inverted
 the linear relations and found an
explicit list of 135 transformation formulas such~as
\begin{equation}
\label{eq:fgTOrst}
 \begin{matrix}
f_{1234657} &=&  - 4 r-2 s_{001} - 2 s_{010} - 2 s_{011} - 2 s_{100}
-2 s_{101} - 2 s_{110} - 2 s_{111} \\ & & - t_{001}    - t_{010} - t_{011} -
t_{100} - t_{101} - t_{110} - t_{111} \\
f_{1237654} & =& -4 s_{101} - 8 r - 4 s_{001} - 4 s_{100} - 2 t_{010}
\qquad \qquad \qquad  \qquad \qquad \\
\cdots & & \cdots \qquad \qquad \cdots \qquad \qquad \cdots \qquad
\end{matrix}
\end{equation}
Of course, these $135$ linear forms in the $15$ coefficients of $F$
satisfy the $120$ linear trinomials above. The complete list of transformation formulas is posted
on our supplementary materials website.

It is important to note that the change of basis in (\ref{eq:fgTOrst})
is unique up to scalar multiple, since we require it
to be ${\rm Sp}_6(\mathbb{F}_2)$-equivariant, and the 15-dimensional ${\rm Sp}_6(\F_2)$-module appears with multiplicity 1 in the permutation representation given by the G\"opel functions.
Consequently, the expressions in (\ref{eq:fgTOrst}) are not just some random transition maps. 
In finding them explicitly, we resolved an issue that was left open by Coble in \cite[\S IV.49,~(16)]{Cob}.

We now present a completely explicit formula for the map
$\PP^6 \dashrightarrow \mathcal{G}\subset \PP^{134}$.
It factors  through the configuration space of
seven points in $\PP^2$ as follows.
A point $(d_1 : d_2 : \cdots : d_7)$ in $\mathbb{P}^6$ is sent to the configuration
given by the columns of the matrix $D$ in (\ref{eq:threebyseven}).
We specialize the $35$ Pl\"ucker coordinates $[ijk]$
to the $3 \times 3$-minors of $D$. The result is the product
$$ \quad [ijk] \, \, = \,\, (d_i-d_j)(d_i-d_k)(d_j-d_k)(d_i+d_j+d_k) . $$
Under this specialization, the quartic bracket polynomial $Q_i$
for the condition (\ref{eq:pascal3}) that six points lie on a conic
maps to a product of $16$ linear forms.
Substituting these expressions into the G\"opel functions
 (\ref{eq:fano2}) and (\ref{eq:pascal2}), we obtain
a list of $135$ polynomials in $d_1,d_2,\ldots,d_7$,
each of which is a product of $28$ linear forms.
All $135$ such products share a common factor of
degree $21$, namely, $ \prod_{i < j} (d_i-d_j)$. Removing
that common factor, we recover the formulas for
the G\"opel coordinates as products of seven linear forms. For instance, we find
\begin{equation}\label{equation:gopelproduct}
\begin{split}
 f_{1234567}  &= 
(d_1{+}d_2{+}d_4)(d_2{+}d_3{+}d_5)(d_3{+}d_4{+}d_6)(d_4{+}d_5{+}d_7)\\
&(d_5{+}d_6{+}d_1)(d_6{+}d_7{+}d_2)(d_7{+}d_1{+}d_3), \\
g_{7123456}  &= 
(d_1{+}d_2{+}d_7)(d_3{+}d_4{+}d_7)(d_5{+}d_6{+}d_7)\\
&(d_1{-}d_2)(d_3{-}d_4)(d_5{-}d_6)({-}d_1{-}d_2{-}d_3{-}d_4{-}d_5{-}d_6).
\end{split}
\end{equation}
These formulas also appear in \cite[\S 3]{CGL}.

We now introduce variables for  the $63$
reflection hyperplanes of $\mathrm{E}_7$.
Denote
\begin{equation} \label{eq:rootsys}
\begin{split}
x_i &  =  d_i - (d_1 + \cdots + d_7),\\
x_{ij}  & = d_i - d_j, \\
x_{ijk} & =  d_i+d_j+d_k. 
\end{split}
\end{equation}
Then, we obtain
\begin{equation}
\label{eq:monopara}
 \begin{array}{ll}
f_{1234567} = x_{124} x_{137} x_{156} x_{235} x_{267} x_{346} x_{457} , & 
 f_{1234576} = x_{124} x_{136} x_{157} x_{235} x_{267} x_{347} x_{456} ,\\
f_{1234657} = x_{124} x_{137} x_{156} x_{236} x_{257} x_{345} x_{467} ,& 
 f_{1234675} = x_{124} x_{135} x_{167} x_{236} x_{257} x_{347} x_{456},\\
\cdots & 
 f_{1243765} = x_{123} x_{145} x_{167} x_{247} x_{256} x_{346} x_{357}, \\
g_{1234567} = x_{1} x_{23} x_{45} x_{67} x_{123} x_{145} x_{167}, &
g_{1234657} = x_{1} x_{23} x_{46} x_{57} x_{123} x_{146} x_{157}, \\ 
g_{1234756} = x_{1} x_{23} x_{47} x_{56} x_{123} x_{147} x_{156},  &
g_{1243567} = x_{1} x_{24} x_{35} x_{67} x_{124} x_{135} x_{167}, \\
\cdots  & 
g_{7162534} = x_7 x_{16} x_{25} x_{34} x_{167} x_{257} x_{347}.
\end{array}
\end{equation}
Each of these $x$-monomials is squarefree of degree $7$, and
represents one of the $30$  Fano configurations like (\ref{eq:fano}),
or one of the $105$ Pascal configurations like (\ref{eq:pascal}).

\smallskip

This is the moment when we  come to the object promised in the title of this section.
 If we regard $x_i$, $x_{ij}$, $x_{ijk}$ as formal variables, 
 the formulas in \eqref{eq:monopara} define a monomial map
 \[
m \colon \PP^{62}   \dashrightarrow \PP^{134}.
\]
The closure of the image of $m$ is a toric variety $\mathcal{T}$ in $\PP^{134}$.
We call $\mathcal{T}$ the {\em G\"opel toric variety}.

Let $\mathcal{A}$ denote the $63 \times 135$-matrix with entries in $\{0,1\}$
representing the monomial map $m$.
The rows of $\mathcal{A}$ are labeled by the
$63$ parameters $x_i,x_{ij}, x_{ijk}$, or, using Cayley's bijection,
by the $63$ vectors in  $(\mathbb{F}_2)^6 \backslash \{0\}$.
Here we erase all occurrences of  the index ``8'' in
Table~\ref{fig:cayleytable} so as to identify  $(\mathbb{F}_2)^6 \backslash \{0\}$ with the 
subsets of $\{1,2,3,4,5,6,7\}$ that have cardinality $1$, $2$ or $3$.
The columns of the matrix $\mathcal{A}$ are labeled by the
$135$ G\"opel coordinates $f_\bullet, g_\bullet$, and thus
 by the $135$ Lagrangians in  $(\mathbb{F}_2)^6$.
Each row of $\mathcal{A}$ has $15$ entries $1$; the other $120$ entries are $0$.
 Each  column of  $\mathcal{A}$ has $7$ entries $1$; the other $56$ entries are $0$.
 By computing the Smith normal form, one checks that 
 our $63 \times 135$-matrix $\mathcal{A}$ has rank $36$ over any field.
 
The convex hull of the columns of  $\mathcal{A}$
is a convex polytope of dimension $35$ in $\R^{63}$. We denote this polytope also by $\mathcal{A}$
and we call it the {\em G\"opel polytope}. 
The G\"opel polytope $\mathcal{A}$ has 135 vertices and 63 distinguished facets, each given by the vanishing of one of the coordinates in the ambient $\R^{63}$. However, these are not all of the facets of $\mathcal{A}$. In fact, it is a difficult computational problem, which we could not solve, to determine the number of facets of $\mathcal{A}$.
Each distinguished facet contains $135-15= 120$ of the vertices, and is
 indexed by a vector $v \in (\mathbb{F}_2)^6 \backslash \{0\}$.
Thus each distinguished facet of $\mathcal{A}$ corresponds combinatorially to the complement of a Lagrangian in $ (\mathbb{F}_2)^6 \backslash \{0\}$.  These facets are indexed by the $63$ $x$-variables.

By construction, the G\"opel polytope $\mathcal{A}$ is the
polytope associated with the projective toric variety $\mathcal{T}$. We summarize the above discussion:

\begin{theorem} 
\label{thm:toricvariety}
The G\"opel toric variety $\mathcal{T}$ has dimension $35$ in $\PP^{134}$. 
It has  $63$ distinguished boundary divisors, each given by the vanishing of one parameter $x_i$, $x_{ij}$, or $x_{ijk}$.
\end{theorem}

The ideal for the embedding $\mathcal{T} \subset \PP^{134}$
is the {\em toric ideal} associated with the matrix $\mathcal{A}$. We computed all minimal generators of this ideal of degree at most $4$ by writing all products of at most $4$ G\"opel functions in terms of $x_i$, $x_{ij}$, or $x_{ijk}$ and comparing them.
There are no linear forms or quadrics in the toric ideal of $\mathcal{A}$. There are $630$ cubic binomials of the form
\begin{equation}
\label{eq:cubibino}
g_{5132467} g_{6123457} f_{1235746} - f_{1236745} g_{5123467} g_{6132457}.
\end{equation}
They form a single orbit under the ${\rm Sp}_6(\F_2)$-action by permuting G\"opel coordinates. Further, there are precisely $12285$ quartic binomials that are not in the ideal generated by the $630$ cubics.
These form two  ${\rm Sp}_6(\F_2)$-orbits of sizes 945 and 11340.
Representatives for these two orbits are, respectively,
\begin{equation}
\label{eq:quarbino}
\begin{matrix}
 f_{1234567} f_{1234675} g_{1243657} g_{2143756} & - & f_{1234576} f_{1234765} g_{1243756} g_{2143657},\\
f_{1234576} f_{1234765} f_{1235647} f_{1236457} & - & f_{1234675} f_{1234756} f_{1235467} f_{1236547}  .
 \end{matrix}
  \end{equation}

It would be desirable to find a {\it Markov basis} of $\mathcal{T}$, i.e. a set of minimal generators of the toric ideal $I_{\mathcal{A}}$ of
the matrix $\mathcal{A}$. This turns out to be nontrivial. In fact, we need binomials of up to degree at least $6$ to generate the toric ideal. We constructed examples of binomials of degree $5$ and $6$
that lie in $I_{\mathcal{A}}$ but are not generated by elements of lower degrees:
$$ \begin{matrix}
f_{1234576}g_{2163745}g_{4132657}g_{6142537}g_{7132456} & -& 
f_{1236574}g_{2143756}g_{4162537}g_{6132457}g_{7132645} \\
f_{1234576}f_{1243657}g_{2153467}g_{3152746}g_{4162537}g_{6123547} & -&
 f_{1236475}f_{1243576}g_{2153746}g_{3162547}g_{4123567}g_{6152734}
 \end{matrix}
 $$
Both of these binomials are {\it indispensable}: they represent two-element
fibers of the semigroup map $\mathcal{A} \colon \mathbb{N}^{135} \rightarrow \mathbb{N}^{63}$.
This implies  that they must appear in every
Markov basis of $I_{\mathcal{A}}$. 

We also do not know whether $\mathcal{T}$ is
projectively normal, or arithmetically Cohen--Macaulay. If this holds then $\mathcal{T}$ is the toric variety of
the polytope $\mathcal{A}$ in the strict sense of \cite[\S 2.3]{CLS}.
Also we do not know the Hilbert polynomial (or Ehrhart polynomial).
At the present time, our $63 \times 135$-matrix $\mathcal{A}$ seems too big
for the packages {\tt 4ti2} (for computing Markov bases),
{\tt polymake} (for computing facets),
 and {\tt normaliz} (which is discussed in \cite[Appendix B.3]{CLS}).
 
Now we are ready to connect the map \eqref{eqn:gopelmap} with the map \eqref{eqn:param:weyl}.

\begin{theorem}
\label{thm:135}
The G\"opel variety $\mathcal{G}$ in $\PP^{14}$ is linearly isomorphic to the closure of the image of the map  $(\PP^2)^7 \dashrightarrow \PP^{134}$ 
described prior to \eqref{eqn:gopelmap}, which is the (ideal-theoretic)
 intersection of $\mathcal{T}$ with the linear space
 $L \simeq \PP^{14}$ in~(\ref{eq:7to14to134}).
Its prime ideal in the polynomial ring in $135$ unknowns
$f_{\bullet{}}$ and $g_{\bullet{}}$ is minimally
generated by $120$ linear trinomials,
$35$ cubic binomials, and $35$ quartic binomials.
\end{theorem}

\begin{proof}
The composition of the maps (\ref{eq:fromctod}), (\ref{eqn:param:weyl}) and (\ref{eq:fgTOrst}) is exactly to (\ref{equation:gopelproduct}). Thus, the transformation (\ref{eq:fgTOrst}) defines an isomorphism $L \simeq
\mathbb{P}^{14}$ that is compatible with the rational map $(\mathbb{P}^2)^7\dashrightarrow{}L$ in
 \eqref{eq:7to14to134} and the rational map $(\PP^2)^7 \dashrightarrow \mathcal{G} \subset \PP^{14}$ in
 \eqref{eqn:param:weyl}.
Therefore, the image of our map $(\PP^2)^7 \dashrightarrow \PP^{134}$ in \eqref{eqn:gopelmap} is linearly isomorphic to the G\"opel variety~$\mathcal{G}$. Moreover, under this projective transformation, the
ideal of the image of the rational map $(\mathbb{P}^2)^7\dashrightarrow{}\mathbb{P}^{134}$ in $L$
is mapped to the ideal of $\mathcal{G}$ in $\mathbb{P}^{14}$, which is the
Gorenstein prime ideal generated by the
$35$ cubics and $35$ quartics in Theorem \ref{thm:gopel}. 

It remains to be shown that the cubic and quartic generators
can be represented by binomials modulo the $120$ trinomials. Indeed, if we take the $630$ cubic binomials described above and write them in terms of $r,s_{\dot{}},t_{\dot{}}$ using (\ref{eq:fgTOrst}), they span exactly the $35$-dimensional vector space of cubics in the ideal of $\mathcal{G} \subset \PP^{14}$. Moreover, if we take the quartics generated by the $630$ cubic binomials and the $12285$ quartic binomials described above and write them in terms of $r,s_{\dot{}},t_{\dot{}}$, they span exactly the vector space of quartics in the ideal of $\mathcal{G} \subset \PP^{14}$. Therefore, the image under \eqref{eq:fgTOrst} of the ideal they generate coincides with the ideal of
 $\mathcal{G} \subset \PP^{14}$ in Theorem~\ref{thm:gopel}.
\end{proof}

\section{The universal Coble quartic in $\PP^7\times\PP^7$} \label{sec:univcoble}

Recall that for each $\tau\in\mathfrak{H}_3$ corresponding to a smooth non-hyperelliptic curve of genus three, the Coble quartic $C_\tau$ is the unique hypersurface of degree $4$ in $\PP^7$ whose singular locus is the Kummer threefold corresponding to that curve. The
{\it universal Coble quartic} $\mathcal{C}(2,4)$ is the Zariski closure in $\mathbb{P}^7\times{}\mathbb{P}^7$ of the set of  pairs 
$(\mathbf{u},\mathbf{x})$ where $\mathbf{u}\in \vartheta(\mathcal{A}_3(2,4))$ corresponds to a non-hyperelliptic curve  and $\mathbf{x}$ lies in the Coble quartic hypersurface corresponding to $\mathbf{u}$. In this section we derive generators for the ideal of $\mathcal{C}(2,4)$ in the theta coordinates $\mathbf{u}$ and $\mathbf{x}$. 

We shall derive formulas for the coefficients $r,s_\bullet,t_\bullet$ of the 
quartic polynomial $F_\tau$ in (\ref{eq:coblequartic})
in terms of the second order theta constants $\mathbf{u}$. 
Our formulas are fairly big, and they
define an explicit rational map of degree $64$
from the Satake hypersurface to the G\"opel variety:
\begin{equation}
\label{eq:64to1revisited}
\mathcal{S} \buildrel{64:1}\over{\dashrightarrow} \mathcal{G}.
\end{equation}
This is realized concretely as a map $\PP^7 \dashrightarrow \PP^{14}$
by listing $15$ polynomials 
 $r(\mathbf{u}),s_\bullet(\mathbf{u}),t_\bullet(\mathbf{u})$ in the variables
 ${\bf u} = (u_{000}, \ldots, u_{111})$.
It agrees with the abstract map (\ref{eq:64to1})
over the non-hyperelliptic locus
of $\mathcal{M}_3(2,4)$ and collapses each of the $36$ components
of the hyperelliptic locus to a single point (see, for example, \eqref{eq:squareofquadric}).
 We also determine the zero set of these polynomials, which is
the indeterminacy locus of (\ref{eq:64to1revisited})
in the hypersurface  $\mathcal{S} = \overline{\vartheta(\mathcal{A}_3(2,4))}$.

 The following theorem summarizes the results to be proved in this section:

\begin{theorem} \label{thm:thetacoble}
\begin{compactenum}[\rm (a)]
\item The coefficients $r, s_\bullet, t_\bullet$ of the
 Coble quartic can be expressed as polynomials of degree $28$ in the eight theta constants $\mathbf{u}$. 
 The resulting polynomial $F$ from (\ref{eq:coblequartic})
is the sum of $372060$ monomials  of bidegree $(28,4)$ in $({\bf u},{\bf x})$.
\item The locus in the Satake hypersurface $\mathcal{S}$ that is cut out by the coefficients
$r, s_\bullet, t_\bullet$  equals the Torelli boundary
 ${\rm Sing}(\mathcal{S}) =\overline{\vartheta(\mathcal{A}_3(2,4))}\setminus \vartheta(\mathcal{M}_3(2,4))$
 \hfill
(cf.~Proposition \ref{prop:torellibdr}).
 \item The prime ideal of the universal Coble quartic 
equals $\langle F, \mathcal{S} \rangle$,
where $\mathcal{S}$ is the Satake polynomial of bidegree $(16,0)$.  Hence 
$\mathcal{C}(2,4)$  is a complete intersection of codimension~$2$ in  $\PP^7\times\PP^7$, and its bidegree equals $16U(28U+4X)$.
Here we write the cohomology ring of $\PP^7 \times \PP^7$ as
 $\mathbb{Z}[U,X]/\langle U^7, X^7\rangle$, with $U$ and $X$ representing the two hyperplane classes.
\end{compactenum}
\end{theorem}

% NEW022413
We note that, in response to this theorem, Grushevsky and Salvati Manni \cite{gsnew} developed a conceptual geometric approach that yields a shorter representation of the same polynomial.

\begin{proof}
(a) We will construct $15$ polynomials of degree $28$ in the eight unknowns $u_{000},\ldots,u_{111}$.
The polynomial $r({\bf u})$ has  $5360$ terms,
each polynomial $s_{ijk}({\bf u})$ has $7564$ terms when $i+j+k \ne 2$ and has $7880$ terms otherwise, and each polynomial $t_{ijk}({\bf u})$ has
$8114$ terms. They are available on our supplementary materials website.

We begin by outlining the derivation of these polynomials.
Consider the Coble quartic $C_\tau$ corresponding to some  non-hyperelliptic $\tau$.
Let $\epsilon$ be a non-zero length $3$ binary vector, so $\epsilon/2 \in A_\tau[2]$ represents a 2-torsion point on the abelian threefold $A_\tau$.    As in the proof of Theorem \ref{thm:gopel}, we write $H^{\pm}_\epsilon$ for the two three-dimensional fixed spaces of the involution
  (\ref{eq:involution}) on $\mathbb{P}^7$ defined by 
$\Theta_2[\sigma](\tau;z) \mapsto \Theta_2[\sigma](\tau;z+\epsilon/2)$.
The fixed space $H_\epsilon^+$ is given by
\begin{equation}
\label{eq:fixedspace}
x_{\delta{}}=0, \qquad  \delta{}\notin\epsilon^\perp,
\end{equation}
where the orthogonal complements are taken with respect to the usual Euclidean form on $(\mathbb{F}_2)^3$.
On the subspace $H_\epsilon^+ \simeq \PP^3$ we use the homogeneous coordinates
\[
(x_{000}:x_{\delta_1}:x_{\delta_2}:x_{\delta_3}), \qquad \delta_i\in\epsilon^\perp.
\]
Substituting (\ref{eq:fixedspace}) into $F_\tau$,
 we get the equation of a Kummer surface $K_2^{\epsilon,+} = C \cap H^+_\epsilon$:
\begin{equation}
\label{intersect} \qquad
r\cdot\sum_{\sigma\in\epsilon^\perp}x^4_\sigma+\sum_{\sigma\in \epsilon^\perp}s_\sigma\cdot m_\sigma +t_\epsilon\prod_{\sigma\in\epsilon^\perp}x_\sigma=0
\qquad \hbox{where} \quad
m_\sigma=\frac{1}{2}\sum_{\nu\in\epsilon^\perp}x^2_{\nu}x^2_{\nu+\sigma}.
\end{equation}
Its singular locus is $H_\epsilon^+\cap \mathrm{Sing}(C_\tau)$. We claim that one of the $16$ singular points of $K_2^{\epsilon,+}$ is
\begin{equation}
\label{eq:peps}
p_\epsilon\,\,=\,\,
\kappa{}_{\tau{}}(\frac{\epsilon}{4})\,\,=\,
\,\bigl(\Theta_2[000](\tau; \frac{\epsilon}{4}) : \Theta_2[001](\tau;\frac{\epsilon}{4}) : 
\,\cdots\, : \Theta_2[111](\tau; \frac{\epsilon}{4}) \bigr).
\end{equation}
Since this point lies in the image of the Kummer map $\kappa_\tau$, it is by definition in the Kummer 
threefold $K_3 = \kappa_\tau(A_\tau) = {\rm Sing}(C_\tau)$. That $p_\epsilon$ is in the fixed space $H_\epsilon^+$ can be seen from the transformation properties of the second order theta functions under $A_\tau[2]$.   Indeed,
\begin{equation}
\quad \Theta_2[\delta{}](\tau; \frac{\epsilon}{4}) \,\,\, = \,\,\,
\theta \bigl(2\tau; \frac{\epsilon{}}{2}+\tau\delta \bigr) \cdot \exp \left[ \pi i\left(\frac{\delta^t\tau\delta}{2}+2\delta^t z \right) \right]
\qquad
\hbox{ for any $\delta{}\in{}(\mathbb{F}_2)^3$.}
\end{equation}
Furthermore,
\begin{align*}
\theta \bigl(2\tau; \frac{\epsilon{}}{2}+\tau\delta \bigr) \,&=\, \theta \bigl(2\tau; -(\frac{\epsilon{}}{2}+\tau\delta) \bigr) = \,\theta \bigl(2\tau; (\frac{\epsilon{}}{2}+\tau\delta)-\epsilon{}-(2\tau{})\delta{} \bigr)\\
&= \,\theta \bigl( 2\tau; \frac{\epsilon{}}{2}+\tau\delta \bigr) \cdot{} \exp \left[2\pi{}i(\delta{}^t(\frac{\epsilon{}}{2}+\tau\delta)-(1/2)\delta{}^t(2\tau{})\delta{})\right]\\
&= \,\theta \bigl(2\tau; \frac{\epsilon{}}{2}+\tau\delta\bigr) \cdot{} \exp \left[\pi{}i\delta{}^t\epsilon{}\right].
\end{align*}
If $\delta{}\notin\epsilon^\perp$, then $\exp \left[\pi{}i\delta{}^t\epsilon{}\right] = -1$. So $\theta \bigl(2\tau; \frac{\epsilon{}}{2}+\tau\delta \bigr)=0$. Thus, $p_\epsilon$ is in $H_\epsilon^+\cap{}K_3$.

\smallskip

Now recall (from Example~\ref{ex:Kummsur})
 that the equation of a Kummer surface can be expressed in terms of the coordinates of any of its $16$ singular points. We apply this to
 the Kummer surface $K_2^{\epsilon,+}$. Namely, let $L_i^\epsilon$ denote the signed $4\times 4$-minor of the corresponding matrix \eqref{universalk2} 
 which is obtained by deleting row $1$ and column $i+1$. Then the coefficients of (\ref{intersect}) are proportional to $L_0^\epsilon , L_1^\epsilon, L_2^\epsilon ,L_3^\epsilon,L_4^\epsilon$. Explicitly, we have the relations
\begin{align*}
&(r : s_{001} : s_{010} :s_{011} :t_{100}) \,\,=\,\,  (L_0^{100} :L_1^{100} :L_2^{100} :L_3^{100} :L_4^{100}), \\
& (r : s_{001} :s_{100} : s_{101} : t_{010}) \,\, = \,\,
(L_0^{010} :L_1^{010} :L_2^{010} :L_3^{010} :L_4^{010}) ,\\
&(r:s_{010}:s_{100}:s_{110}:t_{001})\,\,=\,\,
(L_0^{001} :L_1^{001} :L_2^{001} :L_3^{001} :L_4^{001}) ,\\
&(r:s_{001}:s_{110}:s_{111}:t_{110})\,\,=\,\,
(L_0^{110} :L_1^{110} :L_2^{110} :L_3^{110} :L_4^{110}) ,\\
&(r:s_{010}:s_{101}:s_{111}:t_{101})\,\,=\,\,
(L_0^{101} :L_1^{101} :L_2^{101} :L_3^{101} :L_4^{101}) ,\\
&(r:s_{011}:s_{100}:s_{111}:t_{011})\,\,=\,\,
(L_0^{011} :L_1^{011} :L_2^{011} :L_3^{011} :L_4^{011}) ,\\
&(r:s_{011}:s_{101}:s_{110}:t_{111})\,\,=\,\,
(L_0^{111} :L_1^{111} :L_2^{111} :L_3^{111} :L_4^{111}) .
\end{align*}
Combining these formulas, we see that with respect to the normalization $r=1$, the coefficients of the 
Coble quartic (\ref{eq:coblequartic}) are given by
\begin{align*}
&s_{001}=\frac{L^{100}_1}{L^{100}_0}, \
s_{010}=\frac{L^{100}_2}{L^{100}_0}, \
s_{011}=\frac{L^{100}_3}{L^{100}_0}, \
s_{100}=\frac{L^{010}_2}{L^{010}_0}, \,
s_{101}=\frac{L^{010}_3}{L^{010}_0}, \\ &
 \ s_{110}=\frac{L^{001}_3}{L^{001}_0},\
s_{111}=\frac{L^{011}_3}{L^{011}_0} \,,\quad \hbox{and} \quad 
\,t_\sigma=\frac{L_4^\sigma}{L_0^\sigma}\quad
\hbox{for $\sigma\in \F_2^3 \backslash \{0\}$.}
\end{align*}
The $L_i^{\epsilon{}}$  are polynomials of degree twelve in the
eight coordinates of the point $p_\epsilon$ in (\ref{eq:peps}).
The next step is to relate these coordinates $\Theta_2[\sigma](\tau; \frac{\epsilon}{4})$ to the variables $\mathbf{u}$ on our $\PP^7$. 

We introduce the following notation for the first order theta constants:
\begin{equation}
\label{eq:bigT}
T^\epsilon_{\epsilon'}\,\,=\,\,\theta[\epsilon|\epsilon'](\tau; 0).
\end{equation}
By the inverse addition formula \eqref{eq:additioninv}, we have
\begin{align}
\label{l2}
8\Theta_2[\sigma](\tau; \frac{\epsilon}{4}) \Theta_2[\sigma+\delta](\tau; \frac{\epsilon}{4}) &\,\,= \sum_{\epsilon'\in(\mathbb{Z}/2)^g} (-1)^{\sigma\cdot\epsilon'}T^\delta_{\epsilon'} T^\delta_{\epsilon'+\epsilon}.
\end{align}
In the special case $\delta = 0$, the left hand side becomes a square:
\begin{align}
\label{l1}
8\Theta^2_2[\sigma](\tau; \frac{\epsilon}{4}) &\,\,=\sum_{\epsilon'\in(\mathbb{Z}/2)^g}(-1)^{\sigma\cdot\epsilon'}T^0_{\epsilon'}T^0_{\epsilon'+\epsilon}.
\end{align}
Now, for $0\leq i\leq 3$, the quantity $L^\epsilon_i$ is actually a polynomial of degree $6$ in $\Theta^2_2[*](\tau; \frac{\epsilon}{4})$. Hence formula \eqref{l1} allows us to write these  $L^\epsilon_i$ as polynomials in the variables $T^0_{\epsilon'}$.  For example,
\begin{align*}
L_0^{100} &= \,\,-\frac{1}{64}(T_{000}^0T_{010}^0T_{100}^0T_{110}^0 - T_{001}^0T_{011}^0T_{101}^0T_{111}^0)
(T_{000}^0T_{001}^0T_{100}^0T_{101}^0 - T_{010}^0T_{011}^0T_{110}^0T_{111}^0)\\*
&\qquad \qquad \cdot (T_{000}^0T_{011}^0T_{100}^0T_{111}^0 - T_{001}^0T_{010}^0T_{101}^0T_{110}^0),\\
L_1^{100} &= \,\,\frac{1}{32}(T_{000}^0T_{010}^0T_{100}^0T_{110}^0 + T_{001}^0T_{011}^0T_{101}^0T_{111}^0)
(T_{000}^0T_{001}^0T_{100}^0T_{101}^0 - T_{010}^0T_{011}^0T_{110}^0T_{111}^0)\\*
&\qquad \qquad \cdot (T_{000}^0T_{011}^0T_{100}^0T_{111}^0 - T_{001}^0T_{010}^0T_{101}^0T_{110}^0),\\
L_2^{100} &=\,\, \frac{1}{32}(T_{000}^0T_{010}^0T_{100}^0T_{110}^0 - T_{001}^0T_{011}^0T_{101}^0T_{111}^0)
(T_{000}^0T_{001}^0T_{100}^0T_{101}^0 + T_{010}^0T_{011}^0T_{110}^0T_{111}^0)\\*
&\qquad\qquad \cdot (T_{000}^0T_{011}^0T_{100}^0T_{111}^0 - T_{001}^0T_{010}^0T_{101}^0T_{110}^0),\\
L_3^{100} &=\,\, \frac{1}{32}(T_{000}^0T_{010}^0T_{100}^0T_{110}^0 - T_{001}^0T_{011}^0T_{101}^0T_{111}^0)
(T_{000}^0T_{001}^0T_{100}^0T_{101}^0 - T_{010}^0T_{011}^0T_{110}^0T_{111}^0)\\*
&\qquad \qquad \cdot (T_{000}^0T_{011}^0T_{100}^0T_{111}^0 + T_{001}^0T_{010}^0T_{101}^0T_{110}^0).
\end{align*}

The quantity $L^\epsilon_4$ can be written in the form
\begin{equation}\label{equation:L4}
L^\epsilon_4 = \prod_{\delta\in\epsilon^\perp} \Theta_2[\delta](\tau; \frac{\epsilon}{4})\cdot M^\epsilon
\end{equation}
where $M^\epsilon$ is a polynomial of degree $4$ in $\Theta^2_2[*](\tau; \frac{\epsilon}{4})$. Therefore, we can apply \eqref{l1} to write $M^\epsilon$ as a polynomial of degree $8$ in the $T^\epsilon_{\epsilon'}$. To deal with expressions of the form $\prod_{\delta\in\epsilon^\perp} \Theta_2[\delta](\tau; \frac{\epsilon}{4})$, we group them into two products of pairs of theta functions and apply \eqref{l2}. Note that some of the $T^\epsilon_{\epsilon'}$ on the right hand side of \eqref{l2}, namely the ones such that $\epsilon'\in{}\epsilon{}^{\perp{}}$, correspond to odd characteristics and thus vanish. Hence we can write, for example,
$$
L_4^{100} \,\,\,=\,\,\, \frac{1}{8}\cdot\bigl(\prod_{\epsilon'\in \mathbb{F}_2^3}T^{0}_{\epsilon'}\bigr) 
\cdot \bigl(\,(T_{000}^{001})^2(T_{100}^{001})^2-(T_{010}^{001})^2(T_{110}^{001})^2 \,\bigr).
$$
In this manner, we express all of the coefficients of the Coble quartic in terms of the $T^{\epsilon}_{\epsilon{}'}$. Clearing the denominators, we get the following expressions of degree $28$
 for $r,s_{\sigma},t_{\sigma}$:
\begin{align}\label{equation:coble_r}
r &=\prod_{\epsilon\in\mathbb{F}_2^3\setminus \{0\}}\left(\prod_{\epsilon'\in\epsilon^\perp}T^0_{\epsilon'}-\prod_{\epsilon'\notin\epsilon^\perp}T^0_{\epsilon'} \right),\\ \label{equation:coble_s}
s_\sigma&=  (-2)\cdot\left(\prod_{\epsilon'\in\sigma^\perp}T^0_{\epsilon'}+\prod_{\epsilon'\notin\sigma^\perp}T^0_{\epsilon'} \right)\cdot
\prod_{\epsilon\in\mathbb{F}_2^3\setminus\{0,\sigma\}}\left(\prod_{\epsilon'\in\epsilon^\perp}T^0_{\epsilon'}-\prod_{\epsilon'\notin\epsilon^\perp}T^0_{\epsilon'} \right),\\
\label{equation:coble_t}
t_\sigma&= 8\cdot\prod_{\epsilon'\in\mathbb{F}_2^3}T^{0}_{\epsilon'}\cdot\prod_{\epsilon\notin\sigma^\perp}\left(\prod_{\epsilon'\in\epsilon^\perp}T^0_{\epsilon'}-\prod_{\epsilon'\notin\epsilon^\perp}T^0_{\epsilon'} \right)\cdot W_\sigma , 
\end{align}
where $W_\sigma$ is a polynomial of degree $4$ in the $T^\epsilon_{\epsilon'}$ of the form $(T_*^*)^2(T_*^*)^2-(T_*^*)^2(T_*^*)^2$.

To write these coefficients as polynomials in the $u_\sigma$, we'd like to 
use the addition formula
\begin{align}
\label{addition}
\left(T^\epsilon_{\epsilon'}\right)^2\,\,\,\,=\,\,\,
\sum_{\sigma\in\mathbb{F}_2^3}(-1)^{\sigma\cdot\epsilon'}u_\sigma\cdot u_{\sigma+\epsilon}.
\end{align}
In order to do this, we must write the coefficients of the Coble quartic as polynomials in the squares of the $T^\epsilon_{\epsilon'}$. The expressions for $r,s_\sigma,t_\sigma$ above are not yet in this form. However, by expanding the products we observe that these expressions can be written as polynomials in the squares of the $T^\epsilon_{\epsilon'}$ together with $\prod_{\epsilon\in\mathbb{F}_2^3}T^{0}_{\epsilon}$. Thus, it remains to express the quantity $\prod_{\epsilon\in\mathbb{F}_2^3}T^{0}_{\epsilon}$ as a polynomial in the squares of the $T^\epsilon_{\epsilon'}$. To do this, we make use of the following special case of {\em Riemann's theta relations}
 \cite[Exercise 7.9]{BL}:
\begin{align} \label{eqn:riemannthetarel}
\prod_{\epsilon'\in\epsilon^\perp}T^0_{\epsilon'}\,-\prod_{\epsilon'\notin\epsilon^\perp}T^0_{\epsilon'}
\quad = \quad \prod_{\epsilon'\in\epsilon^\perp}T^{\epsilon}_{\epsilon'}.
\end{align}
Squaring this formula, we get
\begin{align}
\prod_{\epsilon'\in\epsilon^\perp}(T^0_{\epsilon'})^2 + \prod_{\epsilon'\notin\epsilon^\perp}(T^0_{\epsilon'})^2 - 2\prod_{\epsilon'\in\mathbb{F}_2^3}T^{0}_{\epsilon'}\,\,\, = \,\,\,
\prod_{\epsilon'\in\epsilon^\perp}(T^{\epsilon}_{\epsilon'})^2.
\end{align}
Thus, we get a formula for $\prod_{\epsilon\in\mathbb{F}_2^3}T^{0}_{\epsilon}$ as a polynomial in the squares of the $T$ variables. Then, we apply the formula \eqref{addition} to the expressions \eqref{equation:coble_r}, \eqref{equation:coble_s}, \eqref{equation:coble_t} for $r,s_\sigma,t_\sigma$. We obtain the 
$15$ polynomials $r(\mathbf{u}),s_\sigma(\mathbf{u}), t_\sigma(\mathbf{u})$ 
of degree $28$ in the variables $\mathbf{u}$ that had been promised.

\smallskip

(b) We next argue that, set-theoretically, the variety of the ideal $\langle r,s_\sigma, t_{\sigma} \rangle$ equals the Torelli boundary in the Satake hypersurface $\mathcal{S}$. First, suppose that $p \in \mathcal{S}$ lies in the zero locus of $\langle r, s_\sigma, t_\sigma \rangle$. 
By \eqref{equation:coble_r} and Riemann's theta relations \eqref{eqn:riemannthetarel}, the quantity 
\begin{align}
\label{riemannrel}
\prod_{\epsilon'\in\epsilon^\perp}T^{\epsilon}_{\epsilon'}=
\prod_{\epsilon'\in\epsilon^\perp}T^0_{\epsilon'}-\prod_{\epsilon'\notin\epsilon^\perp}T^0_{\epsilon'}
\end{align} vanishes at $p$ for at least one $\epsilon \in \mathbb{F}_2^3 \setminus \{0\}$. By 
the proof of Proposition \ref{prop:torellibdr}
or \cite[Theorem 3.1]{glass},
 in order to show that $p$ lies in the Torelli boundary, it suffices to show that at least $2$ first order even theta constants $T^{\epsilon{}}_{\epsilon{}'}$ vanish at $p$. We distinguish the following two cases:

{\it Case 1.}
 Suppose $\prod_{\epsilon'\in\epsilon^\perp}T^{\epsilon}_{\epsilon'}$ vanishes for 
two distinct $\epsilon{}=\epsilon{}_1,\epsilon{}_2\in{}\mathbb{F}_2^3\setminus \{0\}$. 
Since, these monomials have disjoint sets of variables,
 at least two first order even theta constants vanish.

{\it Case 2.} Suppose that $\left(\prod_{\epsilon'\in\epsilon^\perp}T^0_{\epsilon'}-\prod_{\epsilon'\notin\epsilon^\perp}T^0_{\epsilon'} \right)$ vanishes for exactly one $\epsilon{}\in{}\mathbb{F}_2^3 \setminus \{0\}$. Then by  \eqref{equation:coble_s}, $s_\epsilon{}$ is a product of $\left(\prod_{\epsilon'\in\epsilon^\perp}T^0_{\epsilon'}+\prod_{\epsilon'\notin\epsilon^\perp}T^0_{\epsilon'} \right)$ with some nonzero factors. Thus, $\left(\prod_{\epsilon'\in\epsilon^\perp}T^0_{\epsilon'}+\prod_{\epsilon'\notin\epsilon^\perp}T^0_{\epsilon'} \right)$ vanishes. But then the hypothesis of Case 2 implies
\begin{equation}\label{equation:two_terms_vanish}
\prod_{\epsilon'\in\epsilon^\perp}T^0_{\epsilon'} = \prod_{\epsilon'\notin\epsilon^\perp}T^0_{\epsilon'} = 0.
\end{equation}
and again we conclude that at least two first order even theta constants vanish.
This shows that the variety cut out by the polynomials $r,s_\bullet,t_\bullet$ is contained in the Torelli boundary.  

\medskip

We now show the reverse containment. Suppose a point $p$ lies in the Torelli boundary. By 
\cite[Theorem 3.1]{glass} at least $6$ first order even theta constants $T^{\epsilon{}}_{\epsilon{}'}$ vanish at $p$. In order to show  that all $15$ polynomials $r,s_\bullet,t_\bullet$ 
   vanish at $p$, we divide our analysis into three cases:

{\it Case 1.} Suppose that at least two $T^0_{\epsilon{}_1'}$, $T^0_{\epsilon{}_2'}$ vanish. Then there exists $\epsilon{}\in{}\mathbb{F}_2^3 \setminus \{0\}$ such
 that one of $\epsilon{}_1',\epsilon{}_2'$ is in $\epsilon{}^\perp{}$ and the other is not. For this choice of $\epsilon$, equation (\ref{equation:two_terms_vanish}) holds. Since the expressions (\ref{equation:coble_r}) and (\ref{equation:coble_s}) for $r$ and the $s_\sigma{}$ each contain a factor of either $\left(\prod_{\epsilon'\in\epsilon^\perp}T^0_{\epsilon'}+\prod_{\epsilon'\notin\epsilon^\perp}T^0_{\epsilon'} \right)$ or $\left(\prod_{\epsilon'\in\epsilon^\perp}T^0_{\epsilon'}-\prod_{\epsilon'\notin\epsilon^\perp}T^0_{\epsilon'} \right)$, we see that $r$ and all $s_\sigma{}$ vanish. Also, each $t_\sigma{}$ vanishes because the expression (\ref{equation:coble_t}) has a factor $\prod_{\epsilon'\in\mathbb{F}_2^3}T^{0}_{\epsilon'}$.

{\it Case 2.} Suppose that exactly one $T^0_{\epsilon{}'}$ vanishes. Since at least six theta
 constants vanish, at least $5$ of the $T^{\epsilon{}}_{\epsilon{}'}$ with $\epsilon{}\neq{}0$ 
 vanish. So there exist two vanishing $T^{\epsilon{}_1}_{\epsilon{}'_1}$, $T^{\epsilon{}_2}_{\epsilon{}'_2}$ with $\epsilon{}_1,\epsilon{}_2\neq{}0$ and $\epsilon{}_1\neq{}\epsilon{}_2$. By Riemann's theta relations \eqref{eqn:riemannthetarel}, $\left(\prod_{\epsilon'\in\epsilon^\perp}T^0_{\epsilon'}-\prod_{\epsilon'\notin\epsilon^\perp}T^0_{\epsilon'} \right)$ vanishes for $\epsilon{} = \epsilon{}_1,\epsilon{}_2$. Since the expression  \eqref{equation:coble_r} for $r$ contains all 
seven such expressions as factors and each $s_\sigma{}$ contains all but one of these as factors, we see that $r$ and all the $s_\sigma{}$ vanish. The $t_\sigma{}$ vanish for the same reason as in Case 1.

{\it Case 3.} Suppose that none of the $T^0_{\epsilon{}'}$ vanishes. Then $r$ and the $s_\sigma{}$ vanish by the argument in Case 2. Using (\ref{riemannrel}), we rewrite equation (\ref{equation:coble_t})  for the $t_\sigma{}$ as follows:
\begin{equation}\label{equation:coble_t_simplified}
t_\sigma=C\cdot\prod_{\epsilon\notin\sigma^\perp}\prod_{\epsilon'\in\epsilon^\perp}T^{\epsilon}_{\epsilon'}\cdot W_\sigma,
\end{equation}
where $C$ is nonzero. Let $T^{\epsilon{}_j}_{\epsilon{}_j'}$, $j=1,\dotsc{},6$ be vanishing first order even theta constants whose characteristics form an azygetic $6$-set. The vanishing of each of the $T^{\epsilon{}_j}_{\epsilon{}_j'}$ implies the vanishing of $4$ of the $t_\sigma{}$, namely those with $\sigma{}\notin{}\epsilon{}_j^\perp{}$. If the $\mathbb{F}_2$-vector space spanned by the $\epsilon{}_j$ is the whole space $\mathbb{F}_2^3$, then all $7$ $t_\sigma{}$ vanish in this way. If not, without loss of generality, we may assume that $\{\epsilon{}_1,\dotsc{},\epsilon{}_6\}\subset{}\{001,010,011\}$. In this case, it suffices to show that $t_{100}$ vanishes.

We note that there are many ways to write $W_\sigma{}$ as a polynomial in the $T^\epsilon_{\epsilon'}$, coming from different ways of grouping the $4$ theta constants in (\ref{equation:L4}) into two pairs. For $W_{100}$, we have
\begin{align*}
W_{100} &= (T^{001}_{000})^2(T^{001}_{100})^2 - (T^{001}_{010})^2(T_{001}^{110})^2\\
&= (T^{010}_{000})^2(T^{010}_{100})^2 - (T^{010}_{001})^2(T_{010}^{101})^2\\
&= (T^{011}_{000})^2(T^{011}_{100})^2 - (T^{011}_{011})^2(T_{011}^{111})^2.
\end{align*}
By our assumption, six of the twelve $T$-variables in the expression for $W_{100}$ vanish. The only way $W_{100}$ can be nonzero is for the following three conditions to hold:
\begin{align*}
T^{001}_{000}=T^{001}_{100}=0 &\,\,\text{ or }\,\, T^{001}_{010}=T_{001}^{110}=0;\\
T^{010}_{000}=T^{010}_{100}=0 &\,\,\text{ or }\,\, T^{010}_{001}=T_{010}^{101}=0;\\
T^{011}_{000}=T^{011}_{100}=0 &\,\,\text{ or }\,\, T^{011}_{011}=T_{011}^{111}=0.
\end{align*}
This gives eight cases in total. It can be verified that in each of these eight cases, none of the characteristics of the 
six vanishing first order even theta constants are azygetic. We conclude that all $t_\sigma{}$ vanish, and this completes our proof that the Torelli boundary is contained in the common zero set of 
our coefficient polynomials $r(\mathbf{u}),s_\sigma(\mathbf{u}),t_\sigma(\mathbf{u})$
for the Coble quartic.

\smallskip

(c) Our next goal is to show that $\langle{}F,\mathcal{S}\rangle{}$ is a prime ideal.
Let $\mathcal{C}'$ denote the subscheme of $\mathbb{P}^7\times{}\mathbb{P}^7$
defined by this ideal. We begin by proving that $\mathcal{C}'$  is reduced and
   irreducible. Consider $\mathcal{C}'$ as a family $\pi \colon \mathcal{C}' \to \overline{\vartheta(\mathcal{A}_3(2,4))}$. Let $\mathcal{U}$ denote the non-hyperelliptic locus in $\overline{\vartheta(\mathcal{A}_3(2,4))}$. The fiber over each closed point in $\mathcal{U}$ is an
   irreducible Coble quartic hypersurface. Thus, $\pi^{-1}(\mathcal{U})$ is a family over an irreducible 
   base whose fibers are irreducible and have the same dimension. Therefore, $\pi{}^{-1}(\mathcal{U})$ is irreducible by \cite[Exercise 14.3]{eisenbud}. Since $F$ and $\mathcal{S}$ have no common factor,
   the ideal $\langle{}F,\mathcal{S}\rangle{}$ is a complete intersection. Hence it is Cohen--Macaulay, and all of its 
    minimal primes have the same dimension~$12$. 

We claim that $\mathcal{C}'$ is the closure of $\pi{}^{-1}(\mathcal{U})$ in $\mathbb{P}^7\times{}\mathbb{P}^7$. Suppose it is not. Then it has a twelve-dimensional component contained in $(\overline{\vartheta(\mathcal{A}_3(2,4))}\setminus \mathcal{U}) \times \mathbb{P}^7$. Since $\dim(\overline{\vartheta(\mathcal{A}_3(2,4))}\setminus \mathcal{U}) = 5$, this can only happen if there is a five-dimensional subvariety $Z$ of $\overline{\vartheta(\mathcal{A}_3(2,4))}\setminus \mathcal{U}$ such that $Z\times{}\mathbb{P}^7\subset{}\mathcal{C}'$, i.e. all of the $r,s_\sigma{},t_\sigma{}$ vanish on $Z$. This is impossible because the zero locus of $r,s_\sigma{},t_\sigma{}$ 
is the Torelli boundary, which has dimension four. Therefore, $\mathcal{C}'$ is irreducible and $\mathcal{C}=\mathcal{C'}$. We now know that the radical of $\langle{}F,\mathcal{S}\rangle{}$ is a prime ideal.

The Cohen--Macaulay property implies that $\langle{}F,\mathcal{S}\rangle{}$  satisfies Serre's criterion $(S_1)$. Since the variables ${\bf x}$ do not appear in $\mathcal{S}$, the Jacobian matrix of $\langle F,\mathcal{S} \rangle$ is a $2 \times 16$-matrix in block  triangular form. If we pick any point $({\bf u}, {\bf x})$ with ${\bf u} \in \mathcal{S}$  non-hyperelliptic  and ${\bf x}$  non-singular on its Coble quartic, then the Jacobian matrix has rank 2 at this point.
So, this point is non-singular on the scheme defined by $F=\mathcal{S}=0$, and hence $\langle F,\mathcal{S} \rangle$ is generically reduced.
Therefore, $\langle F,\mathcal{S}\rangle$ is radical, and it is the prime ideal defining
the universal Coble quartic $\mathcal{C} $ in $\PP^7 \times \PP^7$.
 The variety $\mathcal{C}$  has bidegree $16U(28U+4X)$ because it is a complete intersection defined by two polynomials of bidegree $(16,0)$ and $(28,4)$.
 This completes the proof.
 \end{proof}

\begin{remark} \rm
The {\em Satake ring} $\mathbb{C}[\mathbf{u}]/\langle{}\mathcal{S}\rangle{}$ is not a 
unique factorization domain (UFD). Our degree $28$
polynomial $r({\bf u})$ has  distinct  factorizations.
For any $\sigma{}\in{}\mathbb{F}_2^3$, we have
  $r  = c_\sigma q_\sigma$,~where
$$
c_\sigma{} \,\,=\, \prod_{\epsilon\in\sigma{}^{\perp{}}\setminus0}\left(\prod_{\epsilon'\in\epsilon^\perp}T^0_{\epsilon'}-\prod_{\epsilon'\notin\epsilon^\perp}T^0_{\epsilon'} \right)
\quad \hbox{and} \quad
q_\sigma{} \,=\, \prod_{\epsilon\notin\sigma{}^{\perp{}}}\left(\prod_{\epsilon'\in\epsilon^\perp}T^0_{\epsilon'}-\prod_{\epsilon'\notin\epsilon^\perp}T^0_{\epsilon'} \right)\\
$$
Both $c_\sigma{}$ and $q_\sigma{}$ are polynomials in the square of the $T^\epsilon_{\epsilon'}$ and $\prod_{\epsilon\in\mathbb{F}_2^3}T^{0}_{\epsilon}$. As before, we can rewrite them
 as polynomials in $\mathbf{u}$. Therefore, as polynomials in $\mathbf{u}$,  we have
$r=c_\sigma{}q_\sigma{}\mod{}\langle \mathcal{S} \rangle$.
Hence $c_\sigma{}$ is a factor of $r$ in 
$\mathbb{C}[\mathbf{u}]/\langle{}\mathcal{S}\rangle{}$. 
It can also be verified that these $c_\sigma{}$ are pairwise relatively prime in $\mathbb{C}[\mathbf{u}]$, thus pairwise relatively prime in $\mathbb{C}[\mathbf{u}]/\langle{}\mathcal{S}\rangle{}$ since $\deg(c_\sigma) = 12 < 16 = \deg(\mathcal{S})$. 
This is impossible for a UFD. Compare this to results of Tsuyumine \cite{tsuyumine} 
and others on the UFD property for
 coordinate rings representing the moduli space~$\mathcal{A}_3$. \qed
\end{remark}

We close with a remark that highlights the utility 
of the formulas derived in this paper.

\begin{remark}  \rm
Our supplementary files enable the reader to write
geometric properties of plane quartic curves
explicitly in terms of the theta constants (\ref{eq:2ndorder}).
For example, consider the condition that a ternary quartic is the sum of
five reciprocals of linear forms. Such quartics are known as
{\em L\"uroth quartics}. 
A classical result of Morley states that L\"uroth quartics
form a hypersurface of degree $54$ in the $\PP^{14}$ of all quartics.
See \cite{OS} for a modern exposition.

L\"uroth quartics are also characterized by the vanishing of the
following {\em Morley invariant}:
\[
\begin{matrix} \phantom{+}
f_{1234567}{+}f_{1234576}{+}f_{1234657}{+}f_{1234675}{+}f_{1234756}{+}
f_{1234765}{+}f_{1235467}{+}f_{1235476}{+}f_{1235647}{+}f_{1235674} \\ {+}
f_{1235746}{+}f_{1235764}{+}f_{1236457}{+}f_{1236475}{+}f_{1236547}
{+}f_{1236574}{+} f_{1236745}{+}f_{1236754}{+}f_{1237456}{+}f_{1237465} \\
{+}f_{1237546}{+}f_{1237564}{+}f_{1237645}{+}f_{1237654}{+}
f_{1243567}{+}f_{1243576}{+}f_{1243657}{+}
f_{1243675}{+}f_{1243756}{+}f_{1243765}
\end{matrix}
\]
This expression is found in \cite[page 379, after Figure 1]{OS}.
Using the transformation derived in (\ref{eq:fgTOrst}), the Morley
invariant translates into the following linear form in Coble coefficients:
\[
6 r + s_{001} + s_{010} + s_{100} + s_{011} + s_{101} + s_{110} + s_{111}.
\]
We note that, for $S_7 \subset W(\mathrm{E}_7)$ embedded as a parabolic subgroup, the 
$15$-dimensional space of Coble coefficients decomposes 
as a $14$-dimensional irreducible $S_7$-module
plus the trivial representation. The latter is spanned by the Morley invariant.
Now, substituting the polynomials $r({\bf u})$ and $s_{ijk}({\bf u})$ 
from Theorem \ref{thm:thetacoble} into this linear form, 
we obtain a polynomial of degree $28$ that has
$59256$ terms. That expression in $u_{000}, u_{001}, \ldots,u_{111}$ represents the condition that a matrix $\tau$ in the Siegel upper halfspace $\mathfrak{H}_3$ comes from a L\"uroth quartic. \qed
\end{remark}

\section{Equations for universal Kummer threefolds} \label{sec:kumeqn}

We now turn to the object that gave our paper its title. 
The Kummer threefold is the singular locus of the Coble quartic.
In the past  sections we found the ideals
for three variants of the universal Coble quartic.
Each is a twelve-dimensional projective variety,
over a six-dimensional base.
First, in Section \ref{sec:param:gopel}, the base was $\PP^6$.
Next,  in Corollary \ref{cor:gucq},
the base was the G\"opel variety $\mathcal{G}$.
Finally, in Theorem \ref{thm:thetacoble}, the base
was the Satake hypersurface ~$\mathcal{S}$.

For each of the 3 versions of the universal Coble quartic we have
a universal Kummer threefold.
Each of them is a 9-dimensional irreducible variety.
Their respective ambient spaces are $\PP^6 \times \PP^7$, 
$\PP^{14} \times \PP^7$, and $\PP^7 \times \PP^7$.
We discuss their defining equations in this order.

First, consider the parametrization of the Coble quartic (\ref{eq:coblequartic})  in terms of $c_1,c_2, \dots, c_7$ 
given in \eqref{eqn:param:weyl}. This defines what we call the {\em flex version} of the universal Kummer variety:
\[
\mathcal{K}_3^{\rm flex} \subset \mathbb{P}^6 \times \mathbb{P}^7.
\]
For $(c_1: \cdots: c_7)$ not lying on any of the reflection hyperplanes of $\mathrm{E}_7$, 
we consider all points
 $(\mathbf{c},\mathbf{x}) = \bigl((c_1: \cdots: c_7),( x_{000}: \cdots: x_{111})\bigr)$ 
 such that ${\bf x}$ lies in the Kummer variety defined by ${\bf c}$.
The variety $\mathcal{K}_3^{\rm flex}$ is the Zariski closure
of the set of all such points $({\bf c},\bf{x})$ in $\mathbb{P}^6 \times \mathbb{P}^7$.

The label `flex' refers to the fact that general points ${\bf c}$ in $\mathbb{P}^6$,
modulo the action by $W({\rm E}_7)$,
represent plane quartic curves with a distinguished 
inflection point, as seen in the proof of
Theorem \ref{thm:24to1}. We shall construct two
sets of polynomials for the ideal of $ \mathcal{K}_3^{\rm flex} $.
The first set consists of eight polynomials.  They are
 the partial derivatives, with respect to $x_{000}, \ldots,x_{111}$,
of the Coble quartic (\ref{eq:coblequartic}).
Here $r,s_\bullet,t_\bullet$ are the degree $7$ polynomials
in $c_1,\ldots,c_7$ listed after (\ref{eqn:param:weyl}).
Each of these  Coble derivatives is the sum of
$323$ terms of bidegree $(7,3)$~in~$({\bf c},{\bf x})$.

Our second set of equations for $\mathcal{K}_3^{\rm flex}$
consists of $70$ polynomials of bidegree $(6,4)$ in $({\bf c},{\bf x})$.
Their construction is considerably more difficult,
and we shall now explain it. The idea is to use the methods for degeneracy loci
arising from Vinberg's $\theta$-groups, as 
developed in \cite{gsw}.

We start with $\Omega{}^3(4)$, the third exterior power of the cotangent bundle of $\mathbb{P}^7$ twisted by $\mathcal{O}(4)$. Its global sections are homogeneous differential $3$-forms of degree $4$ on $\mathbb{P}^7$. The space of global sections $\mathrm{H}^0(\mathbb{P}^7,\Omega{}^3(4))$ is isomorphic to $\bigwedge^4 \mathbb{C}^8$ via the map
$$
\bigwedge^4 \mathbb{C}^8 \to \mathrm{H}^0(\mathbb{P}^7,\Omega^3(4)) \,,\,\,\,
a_i \wedge a_j \wedge a_k \wedge a_\ell \, \mapsto
\begin{matrix} \,\,\,\,a_\ell(da_i\wedge{}da_j\wedge{}da_k) - a_k(da_i\wedge{}da_j\wedge{}da_\ell) \\
+ a_j(da_i\wedge{}da_k\wedge{}da_\ell) - a_i(da_j\wedge{}da_k\wedge{}da_\ell)
\end{matrix}
$$
Here the basis $\{a_1,a_2,\ldots,a_8\}$ is denoted as in Remark \ref{rmk:gsw}.
Let $U_1,U_2,\dotsc{},U_8$ be the corresponding affine 
open charts on $\mathbb{P}^7$. We write
$U_i = {\rm Spec}(\mathbb{C}[z_1,\dotsc, z_{i-1}, z_{i+1},\dotsc,z_8])$,
where $z_s=a_s/a_i$.
The restriction $\Omega{}^3(4)|_{U_i}$ is generated
by the $35$ forms $dz_j\wedge{}dz_k\wedge{}dz_\ell$ where  $j,k,\ell\neq i$. 
The fiber of the rank $35$ bundle $\Omega{}^3(4)$
over any point ${\bf x} \in{}U_i$ is identified with~$\bigwedge^3\mathbb{C}^7$. 
The value of an element $\alpha{}\in{}\Omega{}^3(4)$  on that fiber
can be written as a linear combination of the 
forms $dz_j\wedge{}dz_k\wedge{}dz_\ell$. The $35$ coefficients
are homogeneous linear expressions in ${\bf x}$, or non-homogeneous linear expressions in ${\bf z}$. We regard these as the coordinates of $\alpha$.

Consider the action of $\mathrm{GL}_7(\mathbb{C})$ on $\bigwedge^3 \mathbb{C}^7$.
As explained in \cite[\S 2.6]{gsw}, this action has unique orbits
of codimensions $1$, $4$, and $7$, respectively.  We denote the closures
of these orbits by $O_1, O_4, O_7$, to indicate their codimension in $\bigwedge^3 \mathbb{C}^7$.
The orbit closure $O_1$ is a hypersurface of degree $7$.
Its defining polynomial $f$ generates the ring of
  $\mathrm{SL}_7(\mathbb{C})$-invariant polynomial functions on 
  $\bigwedge^3 \mathbb{C}^7$. We constructed the invariant $f$
using the method in \cite[Remark 4.4]{kimura}.  Its 
expansion into monomials has $10680$ terms.
It is presented in our supplementary files.
 The ideal for $O_4$ is generated by the $35$ partial derivatives of $f$.
 These are polynomials of degree $6$ in $35$ unknowns,
 and we shall now introduce a certain specialization of these.
 
 Consider a section $v=c_1h_1+\cdots+c_7 h_7\in{}\bigwedge^4 \mathbb{C}^8$ 
 where $h_1,\ldots,h_7$ are the tensors in Remark~\ref{rmk:gsw},
 and $(c_1:\dotsb{}:c_7)\in{}\mathbb{P}^6$ is generic.
 The section $v$ is a linear combination of the $56$ forms
 $da_i \wedge da_j \wedge da_k$ where each coefficient is a monomial
 of bidegree $(1,1)$ in $({\bf c},{\bf x})$. On $U_i$
 we set $a_i = 1$ and $a_j=z_j$ for $j \not=i$.
 Then $21$ of the $56$ summands vanish since $d a_i = 0$
 and $d a_j = d z_j$ for $j \not= i$.
 What is left is a linear combination of the 
  $35$ forms $dz_j\wedge{}dz_k\wedge{}dz_\ell$,
where each coefficient is a monomial of bidegree $(1, 1)$ or $(1,0)$ in $({\bf c},{\bf z})$. These $35$ coefficients are the coordinates of $v$. Shortly, we shall plug them into the derivatives of the  polynomial~$f$.
     
The following  {\it degeneracy locus} has codimension 4 in the 7-dimensional affine space $U_i$:
\begin{equation}
\label{eq:degenloc}
\bigl\{ \, {\bf x} \in{}U_i \,\,| \,\,v({\bf x}) \in{}O_4 \bigr\}.
\end{equation}
It was shown in \cite[\S 6.2]{gsw} that its closure in $\mathbb{P}^7$
is precisely the Kummer threefold over ${\bf c}$.
Here the coordinates on $\PP^7$ need to be relabeled $a_1=x_{000}, \dots, a_8=x_{111}$ as in Remark~\ref{rmk:gsw}. For completeness, we mention that the corresponding degeneracy locus for $O_7$ 
is the set of
$64$ singular points of the Kummer threefold, but we won't need to use this.

By plugging the $35$ coordinates of $v$ into the partial derivatives of $f$, 
we obtain $35$ polynomials
in $({\bf c},{\bf z})$. These polynomials are homogeneous of
degree $6$ in ${\bf c}$, and they are non-homogeneous
of degree $4$ in ${\bf z}$.
For generic ${\bf c}$, these equations define the affine Kummer threefold in $U_i = \{a_i \not= 0\} 
\simeq \mathbb{A}^7$.
If we homogenize these equations, then we obtain $35$ bihomogeneous
polynomials of degree $(6,4)$ in $({\bf c},{\bf x})$.
These all lie in the prime ideal of  $\mathcal{K}_3^{\rm flex}$.

We now repeat this process for the seven other affine charts $U_j$.
This leads to $8 \cdot 35 = 280$ polynomials of bidegree $(6,4)$ in $({\bf c},{\bf x})$,
but only $120$ of them are distinct. Some of these polynomials
have $332$ terms, while the others have $362$.
They span a $70$-dimensional
vector space over $\mathbb{Q}$, and we select a basis for that space.
We conjecture that this basis suffices:

\begin{conjecture} \label{conj:uno}
The  prime ideal 
of the universal Kummer threefold in $\PP^6 \times \PP^7$
is minimally generated by the $78$ polynomials above,
namely, $8$ of bidegree $(7,3)$ and $70$ of bidegree $(6,4)$.
\end{conjecture}

Second, let us consider the {\em G\"opel version} of the
 universal Kummer variety:
 \[
\mathcal{K}_3^{\rm gopel} \,\,\subset \,\,\,\mathcal{G} \,\,\times
\,\, \mathbb{P}^7 \,\,\subset \,\,\mathbb{P}^{14} \times \mathbb{P}^7.
\]
Here, the universal Coble quartic gives an equation of bidegree $(1,4)$.
Its eight partial derivatives are polynomials of bidegree $(1,3)$. 
These polynomials are not sufficient to generate the Kummer ideal, even
over a general point in $\mathcal{G}$, because we need
$70$ quartics. The eight derivatives only give $64$ equations of bidegree $(1,4)$,
so we need at least six more polynomials of bidegree $(?,4)$.
We do not know how to produce these extra quartics.
In other words, we do not know how to lift the
degeneracy locus construction of (\ref{eq:degenloc})
to the G\"opel variety~$\mathcal{G}$.
In light of the beautiful combinatorics in Section \ref{sec:toricgopel},
it is desirable to study  this further.

\smallskip

Last but not least, we return to the object that started this project.
The {\em theta version} of our variety is the Zariski closure of the image of the universal Kummer map
$\kappa$ in (\ref{univkummer}):
\begin{equation}
\label{eq:univkummer4}
\mathcal{K}_3(2,4) \,\,\subset \,\,\,\mathcal{S} \,\,\times
\,\, \mathbb{P}^7 \,\,\subset \,\,\mathbb{P}^{7} \times \mathbb{P}^7.
\end{equation}
As before, we use the coordinates ${\bf u} = (u_{000}: u_{001}: \cdots : u_{111})$ on the 
first $\PP^7$ to parameterize the moduli of Kummer threefolds, and the coordinates 
${\bf x} = (x_{000}: x_{001}: \cdots : x_{111})$
on the second copy of $\PP^7$ to parameterize points of a particular Kummer threefold.
The theta version of the universal Kummer threefold has codimension five, and we have already constructed several polynomials in its defining bihomogeneous prime ideal
$\mathcal{I}_3$.
  First, the ideal   $\mathcal{I}_3$ contains the bidegree $(16,0)$ polynomial of the Satake hypersurface $\mathcal{S}$.  Second, since the Kummer threefold is the singular locus of the Coble quartic hypersurface, one has the eight partial derivatives of $F$ as in Theorem \ref{thm:thetacoble} with respect to $x_{ijk}$.
  These  have bidegree $(28,3)$.

Third and most important, there are additional generators of bidegree $(16,4)$
in $({\bf u} , {\bf x})$.
These play the same role as the $70$ equations of bidegree $(6,4)$
in $({\bf c} ,{\bf x})$ for the flex version in Conjecture~\ref{conj:uno}. We will show that there are $882$ such  additional minimal generators.

\begin{lemma}
\label{lem:july27}
There exists a polynomial $f$ of bidegree $(16,4)$
in the universal Kummer ideal $\mathcal{I}_3$, having the same form (\ref{eq:coblequartic}) as the Coble quartic, 
but now $r, s_\bullet, t_\bullet$ are polynomials of degree $16$ in ${\bf u}$.
These can be given explicitly by the formulas
\begin{align*}
& &  r \,=\,  s_{001} \,=\, s_{100} \,=\,  s_{101} \,= \, s_{110} \,=\,  s_{111} \,= \,
t_{010} \,=\,  t_{110} \,=\,  0  \qquad \qquad \qquad \qquad \\
s_{010} & = & 
(u_{000} u_{011} + u_{001} u_{010} -u_{100} u_{111} -u_{101} u_{110}) 
(u_{000} u_{011} + u_{001} u_{010} +u_{100} u_{111} +u_{101} u_{110})  \\* & &
(u_{000} u_{011} -u_{001} u_{010} +u_{100} u_{111}-u_{101} u_{110})
(u_{000} u_{011} -u_{001} u_{010} -u_{100} u_{111}+u_{101} u_{110})  \\* & & 
(u_{000} u_{010} u_{101} u_{111} - u_{001} u_{011} u_{100} u_{110}) 
(u_{000} u_{010} u_{100} u_{110} -u_{001} u_{011} u_{101} u_{111}) \\
s_{011} & = & 
-(u_{000} u_{010} +u_{001} u_{011}+u_{100} u_{110}+u_{101} u_{111})
  (u_{000} u_{010}+u_{001} u_{011}-u_{100} u_{110}-u_{101} u_{111}) \\*  & & 
\phantom{-}   (u_{000} u_{010}-u_{001} u_{011}+u_{100} u_{110}-u_{101} u_{111})
   (u_{000} u_{010}-u_{001} u_{011}-u_{100} u_{110}+u_{101} u_{111}) \\* & & 
    (u_{000} u_{011} u_{100} u_{111}-u_{001} u_{010} u_{101} u_{110}) 
    (u_{000} u_{011} u_{101} u_{110}-u_{001} u_{010} u_{100} u_{111}) \\ 
t_{001} & = & 
-(u_{000} u_{011}-u_{001} u_{010}-u_{100} u_{111}+u_{101} u_{110})
  (u_{000} u_{011}-u_{001} u_{010}+u_{100} u_{111}-u_{101} u_{110}) \\* & & 
 \phantom{-}  (u_{000} u_{011}+u_{001} u_{010}+u_{100} u_{111}+u_{101} u_{110})
    (u_{000} u_{011}+u_{001} u_{010}-u_{100} u_{111}-u_{101} u_{110}) \\* & & 
 (u_{000} u_{010} u_{101} u_{111}-u_{001} u_{011} u_{100} u_{110})
     (u_{000}^2 u_{010}^2-u_{001}^2 u_{011}^2+u_{100}^2 u_{110}^2-u_{101}^2 u_{111}^2) \\
t_{011} & = & 
 (u_{000} u_{010}-u_{001} u_{011}-u_{100} u_{110}+u_{101} u_{111})
 (u_{000} u_{010}-u_{001} u_{011}+u_{100} u_{110}-u_{101} u_{111}) \\* & & 
 (u_{000} u_{010}+u_{001} u_{011}-u_{100} u_{110}-u_{101} u_{111}) 
 (u_{000} u_{010}+u_{001} u_{011}+u_{100} u_{110}+u_{101} u_{111} \\* & & 
 (u_{000} u_{011} u_{101} u_{110} - u_{001} u_{010} u_{100} u_{111})
 (u_{000}^2 u_{011}^2-u_{001}^2 u_{010}^2+u_{100}^2 u_{111}^2-u_{101}^2 u_{110}^2) \\
t_{101} & = & 
-(u_{000} u_{011}+u_{001} u_{010}+u_{100} u_{111}+u_{101} u_{110})
  (u_{000} u_{011}-u_{001} u_{010}+u_{100} u_{111}-u_{101} u_{110}) \\* & & 
\phantom{-}   (u_{000} u_{011}+u_{001} u_{010}-u_{100} u_{111}-u_{101} u_{110})
    (u_{000} u_{011}-u_{001} u_{010}-u_{100} u_{111}+u_{101} u_{110}) \\* & & 
 (u_{000} u_{010} u_{100} u_{110}-u_{001} u_{011} u_{101} u_{111})
     (u_{000}^2 u_{010}^2-u_{001}^2 u_{011}^2-u_{100}^2 u_{110}^2+u_{101}^2 u_{111}^2) \\
t_{111} &  = & 
(u_{000} u_{010}-u_{001} u_{011}+u_{100} u_{110}-u_{101} u_{111})
 (u_{000} u_{010}-u_{001} u_{011}-u_{100} u_{110}+u_{101} u_{111}) \\* & & 
  (u_{000} u_{010}+u_{001} u_{011}+u_{100} u_{110}+u_{101} u_{111}) 
  (u_{000} u_{010}+u_{001} u_{011}-u_{100} u_{110}-u_{101} u_{111}) \\* & & 
   (u_{000} u_{011} u_{100} u_{111}-u_{001} u_{010} u_{101} u_{110})
    (u_{000}^2 u_{011}^2-u_{001}^2 u_{010}^2-u_{100}^2 u_{111}^2+u_{101}^2 u_{110}^2) \\
t_{100} &  =  & 
  [u_{000}^4 u_{010}^4 u_{100}^4 u_{111}^4]_4
+ [u_{000}^6 u_{010}^2 u_{011}^4 u_{100}^2 u_{110}^2]_{16}
+ 2 [u_{000}^4 u_{001}^2 u_{010}^2 u_{011}^4 u_{100}^2 u_{111}^2]_{12} \qquad \\* & &   
- [u_{000}^4 u_{011}^4 u_{100}^4 u_{110}^4]_4 
- [u_{000}^6 u_{010}^4 u_{011}^2 u_{100}^2 u_{111}^2]_{16}  
- 2 [u_{000}^4 u_{001}^2 u_{010}^4 u_{011}^2 u_{100}^2 u_{110}^2]_{12} \qquad \\* & & 
+ 4 [u_{000}^5 u_{001} u_{010}^5 u_{011} u_{100} u_{101} u_{110} u_{111}]_4 
- 4 [u_{000}^5 u_{001} u_{010} u_{011}^5 u_{100} u_{101} u_{110} u_{111}]_4 \qquad
\end{align*}
The last coefficient $t_{100}({\bf u})$ is irreducible and has $72$ terms.
In particular, $f$ is the sum of $1168$ monomials of degree $(16,4)$ in $({\bf u}, {\bf x})$.
\end{lemma}

Bert van Geemen informed us that the polynomial $f$ above
admits a determinantal representation similar to   (\ref{universalk2}).
  That determinant will be derived in a forthcoming paper of~his.

\begin{proof}
We start with the following quartic relation  among the theta constants of {\em genus four}:
\begin{equation}
\begin{split}
\label{eqn:genus4}
&\phantom{-}\,\, \theta[0010|0001](\tau;0) \cdot \theta[0010|1001](\tau;0) \cdot \theta[0010|0101](\tau;0)
\cdot \theta[0010|1101](\tau;0)\\
&-\theta[0011|0011](\tau;0) \cdot \theta[0011|1011](\tau;0)
\cdot \theta[0011|0111](\tau;0) \cdot \theta[0011|1111](\tau;0)\\
&- \theta[0010|0000](\tau;0) \cdot \theta[0010|1000](\tau;0) \cdot \theta[0010|0100](\tau;0)
\cdot \theta[0010|1100](\tau;0)\\
&+ \theta[0011|0000](\tau;0) \cdot \theta[0011|1000](\tau;0)
\cdot \theta[0011|0100](\tau;0) \cdot \theta[0011|1100](\tau;0)\,\,=\,\,0.
\end{split}
\end{equation}
Genus four identities of this form can be derived from {\em genus two} theta relations using
 \cite[Proposition 4.18]{vangeemen}. Such relations hold identically in $\tau\in\mathfrak{H}_4$, not just on the Schottky locus.  
 
The main computation is to express the relation (\ref{eqn:genus4})
  in terms of the genus four moduli variables  ${\bf u}$.
To do this,  we first turn (\ref{eqn:genus4}) into a relation between squares of the $\theta[\epsilon|\epsilon'](\tau;0)$.  This can be done by taking the product of the relation (\ref{eqn:genus4}) with its seven conjugates, where all possible signs in front of each term are chosen.
This gives a polynomial of degree $16$ in the $\theta^2[\epsilon|\epsilon'](\tau;0)$.   We next apply  the genus four version of \eqref{thetasquares2}.  This writes
   the squares of the first order theta constants as quadrics in the
   $16$ moduli variables $u_{0000}, u_{0001}, \ldots, u_{1111}$. The result is a huge homogeneous polynomial of degree $32$ in  
   these $16$ variables. 

To this polynomial we now   apply the  Fourier--Jacobi expansion technique 
as in (\ref{eq:fourier}). This replaces the polynomial above by one of
its initial forms, but now in the $16$ unknowns $u_{ijk}, x_{ijk}$.
The result is a polynomial of degree $(28,4)$ having $12268$  terms.
It lies in the universal Kummer ideal $\mathcal{I}_3$ and is  expressible as a 
$\mathbb{Q}[{\bf u}]$-linear combination of the $15$ degree four invariants.
Remarkably, this degree $(28,4)$ polynomial 
turns out to be reducible. It has an extraneous factor of degree $(12,0)$.
That factor is a polynomial in ${\bf u}$ which  cannot vanish on $\mathcal{S}$.
Dividing out this factor gives the desired bidegree $(16,4)$ polynomial $f \in \mathcal{I}_3$. 
\end{proof}

The action by the modular group ${\rm Sp}_6(\mathbb{Z})$ on $\mathbb{C}[\mathbf{u},\mathbf{x}]$ induced by the action on $\mathbb{P}^7\times{}\mathbb{P}^7$ preserves the ideal $\mathcal{I}_3$.  On the space of Heisenberg invariant polynomials, the principal congruence subgroup $\Gamma_3(2)$ acts trivially, so there is a well-defined action of the quotient group $\text{Sp}_6(\mathbb{F}_2)={\rm Sp}_6(\mathbb{Z})/\Gamma_3(2)$, 
cf.~(\ref{eqn:gamma2}). For a detailed discussion of the relevant
 representation theory, see \cite{projectiverep}. However, the polynomial $f$ in Lemma \ref{lem:july27} is not invariant under this action. Therefore, we  obtain more polynomials in $\mathcal{I}_3$ by applying this action to $f$.

\begin{lemma}
The orbit of the polynomial $f$ in Lemma~\ref{lem:july27} under the action of ${\rm Sp}_6(\F_2)$ contains exactly $945$ elements. They span a vector space over $\mathbb{C}$ of dimension $882$.
\end{lemma}

\begin{proof}
In what follows, we will use a right action of $\mathrm{Sp}_6(\F_2)$. Recall that on the projective coordinates $({\bf u}, {\bf x})$, generators of ${\rm Sp}_6(\F_2)$ of the form $\gamma_1(A)=\begin{pmatrix}
A&0\\
0&(A^{-1})^t
\end{pmatrix}$
act by
$$
\gamma_1(A)\circ a_\sigma\,=\, a_{A\sigma}  \qquad \hbox{for} \,\,\, a \in \{u,x\}.
$$
Generators $\gamma_2(B)=\begin{pmatrix}
1&B\\
0&1
\end{pmatrix}$ with $B^t=B$ act by
$$
\gamma_2(B)\circ a_\sigma\,=\,
e^{\frac{\pi i \sigma^t B\sigma}{2}} a_\sigma \qquad \hbox{for} \,\,\, a \in \{u,x\},
$$
and the Weyl element $\gamma_3=\begin{pmatrix}
0&-1\\
1&0
\end{pmatrix}$
acts by the {\em finite Fourier transform}
$$
\gamma_3\circ a_\sigma \,=
\, \sum_{\rho\in\mathbb{F}_2^3}(-1)^{\rho\cdot\sigma}a_\rho   \qquad \hbox{for} \,\,\, a \in \{u,x\}.
$$

Alternatively, the group ${\rm Sp}_6(\F_2)$ is generated by the following two elements $\mu{}'$ and $\nu{}'$:
\begin{align}
\mu{}' &= \gamma_2(B)\gamma_1(A) \qquad \qquad \qquad \qquad
\text{ for } B=
\begin{pmatrix}1&0&0\\0&0&0\\0&0&0\end{pmatrix},\ A=\begin{pmatrix}1&0&0\\0&1&0\\1&0&1\end{pmatrix}
,\\
\nu{}' &= \gamma_1(\tilde A)\gamma_2(\tilde B)\gamma_3\gamma_2(\tilde B)\gamma_3\gamma_2(\tilde B)
\qquad  \text{ for } \tilde A=\begin{pmatrix}0&0&1\\1&0&0\\0&1&0\end{pmatrix},\
\tilde B=\begin{pmatrix}0&0&0\\0&0&0\\0&0&1\end{pmatrix}.
\end{align}
These two generators correspond to  $\mu{}$ and $\nu{}$ in (\ref{equation:two_generators}). They act on $\mathbb{C}[\mathbf{u},\mathbf{x}]$ by the substitutions
\begin{align*}
\mu{}' \,: \,\,\,
\begin{matrix}
 a_{000} \mapsto  a_{000}, a_{001}\mapsto  a_{001}, 
  a_{100}\mapsto   i \cdot a_{101},
  a_{101}\mapsto   i \cdot a_{100}, \\
 a_{010}\mapsto  a_{010},
 a_{011} \mapsto  a_{011} ,
  a_{110} \mapsto   i \cdot a_{111}, 
  a_{111}\mapsto   i \cdot a_{110}
    \end{matrix}
  \qquad \hbox{for} \,\,\, a \in \{u,x\}, \, i = \sqrt{-1}, \quad \\
\nu{}' : \,
\begin{matrix}
a_{000}\mapsto  a_{000} {+} a_{100}, a_{001}\mapsto  a_{000} {-} a_{100}, 
a_{010} \mapsto  a_{001} {+} a_{101}, a_{011}\mapsto  a_{001} {-} a_{101} \\
a_{100}\mapsto  a_{010} {+} a_{110}, a_{101}\mapsto  a_{010} {-} a_{110}, 
a_{110}\mapsto  a_{011} {+} a_{111}, a_{111}\mapsto  a_{011} {-} a_{111}
\end{matrix}
 \quad \hbox{for} \,\, a \in \{u,x\}.
\end{align*}

We can try to generate the orbit by applying the generators $\mu{}'$ and $\nu{}'$ successively. This is challenging because some of the polynomials in the orbit are very large. Instead, 
for some $N \geq 2$, we choose random vectors $(\mathbf{u}_j,\mathbf{x}_j)$ over a finite field, and we evaluate
\begin{equation}\label{equation:evaluate_orbit}
\bigl( f(\alpha (\mathbf{u}_1, \mathbf{x}_1)) : \dotsb{}:f(\alpha{}(\mathbf{u}_N,\mathbf{x}_N)) \bigr),
\end{equation}
where $\alpha$ is the product of a sequence of $\mu{}'$s and $\nu{}'$s. In this way, we get $945$ distinct 
points in $\PP^{N-1}$. Therefore, the orbit
of $f$ under ${\rm  Sp}_6(\mathbb{F}_2)$ has at least $945$ elements.

To prove that the number $945$ is exact, we study the stabilizer of $f$ in ${\rm Sp}_6(\F_2)$. It suffices to show that the stabilizer has at least $|{\rm Sp}_6(\F_2)|/945=1536$ elements. It can be seen from the explicit form of $f$ in Lemma~\ref{lem:july27} that the stabilizer contains $\gamma_1(A)$ for the $8$ unipotent lower triangular $3\times{}3$-matrices
$A$, and all of the $64$ elements $\gamma_2(B)$. These elements generate a
group of order $8\cdot 64=512$. Note that all of these elements act on $\mathbb{P}^7\times{}\mathbb{P}^7$ by
signed permutations on the coordinates. We found another element in
the stabilizer, namely
\begin{equation}
\nu{}'\gamma_1(A) \quad \text{ for }\,\,\, A=\begin{pmatrix}0&1&0\\0&0&1\\1&0&0\end{pmatrix}.
\end{equation}
This element of order $3$
 does not act by signed permutations on the coordinates and thus is not in
our subgroup of order $512$. Since the size of the stabilizer divides $|{\rm Sp}_6(\F_2)|=1451520$, it must be at least $512 \cdot 3=1536$. This shows the orbit has size $945$. Alternatively, we can do explicit calculations with these matrices in {\tt GAP} and establish the exact size of the subgroup. 

We generated the $945$ polynomials of degree $(16,4)$ in implicit form  as in
(\ref{equation:evaluate_orbit}). Due to the size of the problem, it is not feasible to compute the dimension of the vector space spanned by the orbit of $f$ directly over $\mathbb{Q}(i)$. Therefore,
we evaluated them at random configurations of $N \geq 945$ points over various finite fields.
In each experiment, the resulting $945 \times N$-matrix 
 has rank $882$. Hence the $\mathbb{C}$-vector space spanned by the $945$ orbit elements has dimension at least $882$.

To prove that the number $882$ is exact, we identified exactly $945-882=63$ linearly independent relations among the $945$ orbit elements. Fortunately, the relations are rather simple. There exist $15$ distinct group elements $g_1,\dotsc, g_{15}\in {\rm Sp}_6(\F_2)$ such that
\begin{equation}\label{equation:orbit_relation}
\sum_{j=1}^{15} \pm f\circ g_j=0.
\end{equation}
This relation is found in the following way: first take a set $\mathcal{B}$ of $882$ linearly independent elements of the form $f\circ{}g$. The complement $\mathcal{B}^c$ contains $63$ elements. Take any $g_1\in{}\mathcal{B}^c$. For each $f\circ{}g\in{}\mathcal{B}$, test by computation in a finite field if the $882$ elements in $\mathcal{B}\backslash{}\{f\circ{}g\}\cup{}\{f\circ{}g_1\}$ are still linearly independent. It turns out that only $14$ out of the $882$ elements satisfy this property. These elements are $f\circ{}g_2,\dotsc{},f\circ{}g_{15}$. Then, it is computationally feasible to find and verify the relation (\ref{equation:orbit_relation}) over $\mathbb{Q}(i)$ due to the reduced size of the problem. The complete list is found in our supplementary materials. Applying ${\rm Sp}_6(\F_2)$ to (\ref{equation:orbit_relation}) gives $63$ linearly independent relations since no two relations involve the same conjugate of~$f$.
\end{proof}

\begin{remark} \rm 
We can verify that our polynomials lie in $\mathcal{I}_3$ using direct numerical computations.
Indeed, by running Swierczewski's code for theta functions in {\tt Sage} \cite{sage},
we can generate arbitrarily many points $({\bf u}, {\bf x})$ on the universal Kummer 
variety $ \mathcal{K}_3(2,4)$. The polynomial with $1168$ terms 
was shown to vanish on all of them. Likewise, we can check numerically that this vanishing property is preserved under the two substitutions $\mu{}', \nu{}'$. 
\end{remark}

\begin{remark} \rm
The space spanned by the $\mathrm{Sp}_6(\F_2)$-orbit of $f$ is a representation of $\mathrm{Sp}_6(\F_2)$.
It is equivalent to a subrepresentation of the representation induced from the trivial representation of the stabilizer subgroup of the polynomial $f$ from Lemma~\ref{lem:july27}. According to a {\tt GAP} calculation,
the irreducible subrepresentations of that $945$-dimensional induced
$\mathrm{Sp}_6(\F_2)$-representation have the following dimensions: 
$\,
1,27,27,35,35,84,120,168,168,280$.

Here the two occurrences of 27 and 35 refer to an irreducible representation that appears with multiplicity 2. Note that $945 - 882 = 63 =  1 + 27  +35$, and this is the unique way to build 63 from the dimensions of these irreducible summands. So we can use this to determine the structure of these 882 equations as a representation of $\mathrm{Sp}_6(\F_2)$. \qed
\end{remark}

We do not completely understand what happens
to the universal Kummer threefold $\mathcal{K}_3(2,4)$ 
and the generators found above when we restrict the 
base to one of the $36$ hyperelliptic divisors in $\mathcal{S}$.
The issue is that, for fixed $\tau$ in the hyperelliptic locus, the ideal of the Kummer threefold has a different structure.  It can be seen from equations (\ref{equation:coble_r}), (\ref{equation:coble_s}) and (\ref{equation:coble_t}) that the vanishing of a theta constant $T^\epsilon_{\epsilon'}$ causes the 
universal
Coble quartic to become the product of a monomial in the remaining non-zero theta constants times the square of a quadric with integer coefficients in the ${\bf x}$ variables.  For instance, if $T^0_0=0$ then 
\begin{equation}
\label{eq:squareofquadric}
F_\tau \,\, = \,\,\left(x_{000}^2+x_{001}^2+x_{010}^2+x_{011}^2+x_{100}^2+x_{101}^2+x_{110}^2+x_{111}^2 \right)^2\cdot
\!\! \prod_{\epsilon\in\mathbb{F}_2^3\setminus\{0\}}(T^0_{\epsilon'})^4.
\end{equation}
Hence over the hyperelliptic locus, the prime ideal of a Kummer threefold contains a quadric. This can also be seen from the vanishing of a theta
constant $\theta[\epsilon|\epsilon'](\tau;0)$ in equation (\ref{doubledarg}).
That quadric generates 8 linearly independent cubics and 36 linearly independent quartics.  The additional $34$ quartic relations were identified recently by M\"uller \cite[Theorem 4.2]{Muller}.

\smallskip

We optimistically conjecture that the equations we have found so far 
are sufficient:

\begin{conjecture}
\label{conj:due}
The prime ideal $\mathcal{I}_3$ of the universal Kummer variety
(\ref{eq:univkummer4}) is generated by $891$
bihomogeneous polynomials in $({\bf u}, {\bf x})$:
the Satake polynomial of degree $(16,0)$, the eight
Coble derivatives of degree $(28,3)$, and the $882$
polynomials of degree $(16,4)$.
\end{conjecture}

% NEW 060613
\begin{remark}  \label{rmk:piazza}
 \rm 
In their recent work \cite{dpsm}, Dalla Piazza and Salvati Manni have obtained the defining polynomial of the universal Coble quartic directly from the Fourier--Jacobi expansion of a certain relation among theta constants in genus four.  They also obtain another polynomial in the defining ideal of $\mathcal{K}_3(2,4)$.  At present, we have not determined whether this latter polynomial is contained in the ideal described in Conjecture \ref{conj:due}. \qed
\end{remark}
% NEW 060613

What remains to be done is to better integrate the generic 
case and the hyperelliptic case. This will be crucial for understanding the
relationship between our $70$ extra quartics 
over $\PP^6$ and our $882$ extra quartics over $\mathcal{S}$.
How can we construct these over $\mathcal{G}$?
How can we  lift the map (\ref{eq:64to1revisited}) to the universal Kummer threefolds?
We expect that the study of hyperelliptic
moduli in \cite{FS} and its Macdonald representation of $S_8$
will be relevant here.

\smallskip

In closing we remark that the Fourier--Jacobi method was the key to
success in Lemma~\ref{lem:july27}. Here the systematic passage to 
a non-trivial initial form was driven by a
 toric (or {\em sagbi}) degeneration as in (\ref{eq:fourier}). This is
 fundamental also for tropical geometry, our next topic.

 \section{Next steps in tropical geometry} \label{sec:tropgeom}

We now take a look at Kummer threefolds and their moduli through the lens
of tropical algebraic geometry \cite{Melody, CMV, DFS, HKT, HJJS,  MZ, SW}.
Each of our ideals defines a tropical variety, which is a balanced polyhedral fan.
These fans represent
compactifications of our varieties and moduli spaces, and they allow us to understand
what happens when the  field $\C$ of complex numbers
gets replaced by  a field $K$ with a non-trivial non-archimedean valuation.
This section serves as a manifesto in favor of
explicit polynomial equations. They are essential
for understanding the combinatorics that links
classical and tropical moduli spaces of curves.
This perspective is developed further in
the subsequent article \cite{RSS}.

\smallskip

We begin our discussion with the six-dimensional G\"opel variety $\mathcal{G}$.
Recall from Theorem~\ref{thm:gopel} that $\mathcal{G}$ sits inside $\PP^{14}$
where it is defined by $35$ cubics and $35$ quartics.
The tropicalization of this irreducible variety is a pure six-dimensional
polyhedral fan  in $\mathbb{TP}^{14} = \R^{15}/\R (1,1,\ldots,1)$.
This fan can be computed, at least in principle, with 
{\tt Gfan} \cite{gfan}.
However, it does not have good combinatorial properties.
For tropical geometers, it is much
better to pass to the {\em modification} arising from the re-embedding,
seen in Section \ref{sec:toricgopel}, of the G\"opel variety $\mathcal{G}$ in  $\PP^{134}$.
Its ideal is generated by binomials and linear trinomials.  We
 define the {\em tropical G\"opel variety}
 to be the tropical variety of that ideal. This is a six-dimensional fan
 which lives in   $\mathbb{TP}^{134}$. The
 Weyl group $W(\mathrm{E}_7)$ acts on ${\rm trop}(\mathcal{G})$ by permuting coordinates.
 
 We shall construct the tropical G\"opel
 variety  ${\rm trop}(\mathcal{G})$ combinatorially, not from its defining ideal via {\tt Gfan} \cite{gfan},
 but directly from the parametrization given by  (\ref{eq:rootsys}) and (\ref{eq:monopara}).
This is equivalent to the 
parametrization in Section~\ref{sec:param:gopel}, but we have factored it as follows:
\begin{equation}
\label{eq:matrixfac}
 \PP^6 \,  \buildrel{\ell}\over{\longrightarrow} \,
  \PP^{62} 
  \buildrel{m}\over{\dashrightarrow} \, \PP^{134} .
 \end{equation}
The first map $\ell$ is given by evaluating the $63$ linear
forms (\ref{eq:rootsys}) that represent $\mathrm{E}_7$.
The second map $m$ is the monomial map (\ref{eq:monopara}).
It corresponds to the $63 \times 135$-matrix  $\mathcal{A}$ that encodes
incidences of vectors and Lagrangians in $(\F_2)^6 \backslash \{0\}$.
The tropicalization of the monomial map $m$ is given by the
classically-linear map $\mathbb{TP}^{62} \rightarrow \mathbb{TP}^{134}$
defined by its transpose $\mathcal{A}^t$.
The tropicalization of the linear space ${\rm image}(\ell)$ is the
{\em Bergman fan}  of the matroid $M(\mathrm{E}_7)$ of the reflection arrangement of
type $\mathrm{E}_7$, as defined in \cite{ARW}.
This fan lives in $\mathbb{TP}^{62}$, and we have:

\begin{proposition}
\label{prop:tropG}
The tropical G\"opel variety ${\rm trop}(\mathcal{G})$ coincides with 
the image of the Bergman fan of $M(\mathrm{E}_7)$ under the linear map 
$\mathbb{TP}^{62} \rightarrow \mathbb{TP}^{134}$ given by
the matrix $\mathcal{A}^t$ from Section~\ref{sec:toricgopel}.
\end{proposition}

\begin{proof}
The factorization (\ref{eq:matrixfac}) shows
that $\mathcal{G}$ is the image of a map whose coordinates
are  products of linear forms. The result then
follows immediately from \cite[Theorem 3.1]{DFS}.
\end{proof}

\begin{table}
\newcommand{\rA}{\mathrm{A}}
\newcommand{\rD}{\mathrm{D}}
\newcommand{\rE}{\mathrm{E}}

\begin{tabular}{l|l|l|l|l}
Family \# & Rank & Size & Root subsystem & Equations of a representative flat \\
\hline
{\bf 1} & 1 & 63 & $\rA_1$ & $c_7$\\
\hline
{\bf 2} & 2 & 336 & $\rA_2$ & $c_7,c_1+c_3-c_6$\\
3 & 2 & 945 & $\rA_1 \times \rA_1$ & $c_4,c_7$\\
\hline
{\bf 4} & 3 & 1260 & $\rA_3$ & $c_7,c_4,c_1-c_5$\\
5 & 3 & 5040 & $\rA_1 \times \rA_2$ & $c_7,c_4,c_3+c_5-c_6$\\
6 & 3 & 3780 & $\rA_1^{\times 3}$ & $c_6,c_3-c_5,c_2-c_7$\\
7 & 3 & 315 & $\rA_1^{\times 3}$ & $c_7,c_4,c_3$\\
\hline
{\bf 8} & 4 & 2016 & $\rA_4$ & $c_7,c_4,c_3+c_5-c_6,c_1-c_5$\\
{\bf 9} & 4 & 315 & $\rD_4$ & $c_7, c_5, c_4, c_1$\\
10 & 4 & 7560 & $\rA_1 \times \rA_3$ & $c_7,c_5,c_4,c_3-c_6$\\
11 & 4 & 3360 & $\rA_2 \times \rA_2$ & $c_7, c_6, c_3 - c_4 + c_5, c_1 + c_4 + c_5$\\
12 & 4 & 1260 & $\rA_1 \times \rA_3$ & $c_7, c_5 - c_6, c_4, c_3$\\
13 & 4 & 15120 & $\rA_1^{\times 2} \times \rA_2$ & $c_7, c_6, c_2 + c_3 + 2c_4 + c_5, c_1 + c_4 + c_5$\\
14 & 4 & 3780 & $\rA_1^{\times 4}$ & $c_7, c_6, c_5, c_4$\\
\hline
{\bf 15} & 5 & 336& $\rA_5$ & $c_7, c_6, c_5, c_2 - c_4, c_1 + c_4$\\
{\bf 16} & 5 & 1008& $\rA_5$ & $c_7, c_6, c_4, c_3 + c_5, c_1 + c_5$\\
{\bf 17} & 5 & 378 & $\rD_5$ & $c_7, c_6, c_3, c_2 - c_4, c_1$\\
18 & 5 & 6048& $\rA_1 \times \rA_4$ & $c_7, c_6, c_4, c_2 - c_3 - c_5, c_1 + c_5$\\
19 & 5 & 945& $\rA_1 \times \rD_4$ & $c_7, c_6, c_5, c_4, c_3$\\
20 & 5 & 5040& $\rA_2 \times \rA_3$ & $c_7, c_6, c_4, c_2 - 2c_5, c_1 + c_5$\\
21 & 5 & 7560& $\rA_1^{\times 2} \times \rA_3$ & $c_7, c_6, c_5, c_4, c_1 + c_2$\\
22 & 5 & 10080& $\rA_1 \times \rA_2^{\times 2}$ & $c_7, c_6, c_4, c_2 - c_3 - c_5, c_1 - c_3 - 2c_5$\\
23 & 5 & 5040& $\rA_1^{\times 3} \times \rA_2$ & $c_7, c_6, c_4, c_3, c_1 + c_2 - c_5$\\
\hline
{\bf 24} & 6 & 288 & $\rA_6$ & $c_7, c_6, c_4, c_3 + c_5, c_2 - 2c_5, c_1 + c_5$\\
{\bf 25} & 6 & 63 & $\rD_6$ & $c_7, c_6, c_5, c_4, c_3, c_2$\\
{\bf 26} & 6 & 28 & $\rE_6$ & $c_7, c_6, c_4 - c_5, c_3, c_2 - c_5, c_1$\\
27 & 6 & 1008 & $\rA_1 \times \rA_5$ & $c_7, c_6, c_5, c_4, c_2 + c_3, c_1 - c_3$\\
28 & 6 & 378 & $\rA_1 \times \rD_5$ & $c_7, c_6, c_5, c_4, c_2, c_1 + c_3$\\
29 & 6 & 2016 & $\rA_2 \times \rA_4$ & $c_7, c_6, c_4, c_3 + c_5, c_2 - 2c_5, c_1 - 3c_5$\\
30 & 6 & 5040 & $\rA_1 \times \rA_2 \times \rA_3$ & $c_7, c_6, c_4, c_3, c_2 - 2c_5, c_1 + c_5$
\end{tabular}
\caption{The flats of the $\rE_7$ reflection arrangement.}
\label{fig:E7flats}
\end{table}

The first step towards the tropical G\"opel variety ${\rm trop}(\mathcal{G})$
is to list the flats of the matroid $M(\mathrm{E}_7)$.
In Table~\ref{fig:E7flats}, we present the classification of all proper flats of the matroid $M(\mathrm{E}_7)$. There are $30$ orbits of flats under the
$W(\mathrm{E}_7)$-action.
For each $W(\mathrm{E}_7)$-orbit we list the rank, the size of the orbit,
and a set of linear forms that serves as a representative.
These linear forms define an intersection of the hyperplanes
in the reflection arrangement, of codimension equal to the rank,
and the flat consists of all hyperplanes that contain that linear space.

In Table~\ref{fig:E7flats}, we are using the bijection between flats and
parabolic subgroups of the Weyl group, which can be found in \cite[Theorem 3.1]{BI}. These parabolic subgroups correspond to root subsystems in the finite type Dynkin diagram.
We calculated Table~\ref{fig:E7flats} from scratch. Similar information for the lattice of parabolic subgroups can be found in \cite[Table A.2]{GP}. This can also be taken as an independent verification of the correctness of 
Table~\ref{fig:E7flats}.

Eleven of the $30$ orbits consist of irreducible
root subsystems. Their numbers are marked in bold face.
The total number of proper irreducible flats of the matroid $M(\mathrm{E}_7)$ is therefore
$$ 6091 \,\,\,=\,\,\,  63+336+1260+2016+315+1008+336+378+288+63+28. $$
The Bergman fan of $M(\mathrm{E}_7)$ is a six-dimensional fan in $\mathbb{TP}^{62}$.
It has $f_1 = 6091$ rays, one for each irreducible flat $F$.
The rays are generated by their $0$-$1$-incidence vectors $e_F = \sum_{i \in F} e_i$.
A collection $\mathcal{F}$ of flats is {\em nested} if,
for every antichain $\{F_1,F_2,\ldots,F_r\}$ in $\mathcal{F}$ with $r \geq 2$,
the flat $F_1 \vee F_2 \vee \cdots \vee F_r$ is not irreducible.
(This flat represents the subspace obtained by intersecting the given subspaces).
For any nested set $\mathcal{F}$, we consider the convex cone $C_\mathcal{F}$
spanned by the linearly independent vectors $e_F, \,F \in \mathcal{F}$.
The Bergman fan of $M(\mathrm{E}_7)$ is the collection
of all cones $C_\mathcal{F}$ where $\mathcal{F}$ runs
over all nested sets of irreducible flats  of $M(\mathrm{E}_7)$.

By  results of Ardila, Reiner and Williams in \cite[\S 7]{ARW},
this simplicial fan is the coarsest fan structure on its support.
Further, if $f_i$ is the number of $i$-dimensional cones $C_\mathcal{F}$ then
\begin{equation}
\label{eq:eulerchar}
 f_6-f_5+f_4-f_3+f_2-f_1 + 1 \,\,=\,\,
1 \cdot 5 \cdot 7 \cdot 9 \cdot 11 \cdot 13 \cdot 17 \,\,=\,\,  765765. 
\end{equation}
Equation (\ref{eq:eulerchar}) rests on two non-trivial facts from matroid theory. First,
the simplicial complex underlying the Bergman fan is a wedge of $\mu$ spheres,
where $\mu$ is the M\"obius number of the matroid (see \cite[p.42, Corollary]{AK}).
Second, for the matroid of a finite Coxeter system (as in \cite{ARW}),
the M\"obius number $\mu$ is the product of the {\em exponents} of that group~\cite[(1.1)]{OrSo}.
We note that Bramble \cite{Bra} was the first to determine the 
fundamental invariants of $W(\mathrm{E}_7)$. As for any reflection group,
their degrees  $2,6,8,10,12,14,18$ are the exponents plus one.

\begin{corollary}
The tropical G\"opel variety ${\rm trop}(\mathcal{G})$ in $\mathbb{T}^{134}$ is
the union of the convex polyhedral cones $\mathcal{A}^t C_\mathcal{F}$
where $\mathcal{F}$ runs over all nested sets of irreducible root subsystems of $\mathrm{E}_7$.
\end{corollary}

\begin{proof}
This follows from Proposition \ref{prop:tropG}
and the construction of the Bergman fan in \cite{ARW}.
\end{proof}

Another approach to the tropical G\"opel variety
${\rm trop}(\mathcal{G})$ is to use the formulas
(\ref{eq:fano2}) and (\ref{eq:pascal2})
for the G\"opel functions $f_\bullet, g_\bullet$
as polynomials of degree $7$ in the brackets $[ijk]$.
This defines a  rational map from
the Grassmannian ${\rm Gr}(3,7) 
\subset \PP^{34} $ to the G\"opel variety $\mathcal{G} \subset \PP^{134}$.
The tropicalization of this map is a piecewise
linear map $\mu \colon \mathbb{TP}^{34} \rightarrow \mathbb{TP}^{134}$.
Note that the map $\mu$ is not linear because
the expression for $g_\bullet$ in the brackets $[ijk]$ is a binomial 
and not a monomial. We expect that this rational map 
does not commute with tropicalization. 

The image of 
the tropical Grassmannian ${\rm trop}({\rm Gr}(3,7))$
under the piecewise linear map $\mu$ is a subfan of the
tropical G\"opel variety ${\rm trop}({\cal G})$.
It would be interesting to identify that subfan.
Note that, by \cite[Theorem 2.1]{HJJS}, that tropical Grassmannian has the face numbers
$$ 
f \bigl( {\rm trop}({\rm Gr}(3,7)) \bigr) \quad = \quad
\bigl(721, 16800, 124180, 386155, 522585, 252000 \bigr)  . 
$$
The tropical G\"opel variety  ${\rm trop}(\mathcal{G})$
 is a modification and its face numbers are even larger.

The positive part of the tropical Grassmannian governs the combinatorics of the
{\em cluster algebra} structure on the coordinate ring of ${\rm Gr}(3,7)$.
Interestingly, the  Weyl group relevant for this is of type $\mathrm{E}_6$ and not $\mathrm{E}_7$.
Namely, it is shown in \cite[\S 7]{SW} that
the normal fan of the $\mathrm{E}_6$-associahedron
defines a simplicial fan structure on the positive part of 
$ {\rm trop}({\rm Gr}(3,7)) $, with f-vector $(42, 399, 1547, 2856, 2499, 833)$.
The relationship to cluster algebras  is explained in \cite[\S 8]{SW}.
In light of this, it would be interesting to examine the positive part of
 ${\rm trop}(\mathcal{G})$.
 
The Satake hypersurface $\mathcal{S}$ lives in $\PP^7$, and it is
parametrized by $\mathfrak{H}_3$ via the theta constant map $\vartheta$.
The Newton polytope of its defining polynomial of degree $16$
has  face numbers
$$ f \bigl({\rm Newton}(\mathcal{S})\bigr) \,\,=\,\, (344, 2016, 3584, 2828, 1120, 224, 22). $$
The corresponding tropical hypersurface 
${\rm trop}(\mathcal{S})$ in $\mathbb{TP}^7$ is a six-dimensional fan with
$22$ rays and $2016$ maximal cones. 
Now we shall explain and derive the following result:

\begin{proposition}
\label{prop:melody}
The image of the tropical Siegel space
${\rm trop}(\mathfrak{H}_3)$
under the piecewise-linear map ${\rm trop}(\vartheta)$
is the intersection of ${\rm trop}(\mathcal{S})$
with the normal cone of
${\rm Newton}(\mathcal{S})$ at the vertex
\begin{equation}
\label{eq:monom}
M \quad = \quad  - 2 u_{000}^9 u_{001} u_{010} u_{011} u_{100} u_{101} u_{110} u_{111}. 
\end{equation}
The map ${\rm trop}(\vartheta)$ induces the level structure on ${\rm trop}(\mathcal{A}_3)$ described by Chan in  \cite[\S 7.1]{Melody}.
\end{proposition}

\begin{proof}
We first define the terms in Proposition \ref{prop:melody}.
Following \cite{CMV, MZ}, the {\em tropical Siegel space}
${\rm trop}(\mathfrak{H}_3)$ is  the cone 
${\rm PD}_3$ of positive definite real symmetric $3 \times 3$-matrices.
We use tropical theta functions as described by Mikhalkin and Zharkov in \cite{MZ}.
The coordinates of the {\em tropical theta constant map} ${\rm trop}(\vartheta)$
are indexed by $\sigma \in \{0,1\}^3$, and they are defined~by
\begin{equation}
\label{eq:troptheta}
 {\rm trop}(\vartheta)_\sigma(T) \,\,\, = \,\,\, {\rm min} \bigl\{
(n  + \sigma/2)^t \cdot T \cdot (n + \sigma/2) \,:\,
n \in \mathbb{Z}^3 \bigr\} \qquad \hbox{for} \quad  T \in {\rm PD}_3  .
\end{equation}
The function ${\rm trop}(\vartheta)_\sigma$ 
from ${\rm PD}_3$ to $\R$ is well-defined and takes non-negative  values
since $T$ is positive definite.
It is zero for all $T$ when $\sigma = (0,0,0)$.
Hence we have a well-defined map
$$ 
{\rm trop}(\vartheta) \colon {\rm PD}_3 \,\rightarrow \, {\rm trop}(\mathcal{S}) \,\subset \,\mathbb{TP}^7. 
$$
To show that the image lands in the tropical Satake hypersurface, we set  $\tau = i \rho T$
in  (\ref{thetaconstants}) and (\ref{eq:2ndorder}), where
$\rho \rightarrow \infty$ is a real parameter.
Write $\Theta_2[\sigma](\tau;0)$ as a series in
$\epsilon = {\rm exp}(-\rho)$, with (\ref{eq:troptheta}) as the
exponent of the lowest term. Then each coordinate of
$\vartheta(\tau)$ is a series in $\epsilon$. Plugging these
eight series into the Satake polynomial $\mathcal{S}$, we obtain zero.
Compare the $471$ monomials of $\mathcal{S}$ according to their
order in $\epsilon$ after this substitution.
Two (or more) of the monomials  must have 
the same lowest order in $\epsilon$.
This means that ${\rm trop}(\vartheta(T)) $ lies in~${\rm trop}(\mathcal{S})$.

The vertex of ${\rm Newton}(\mathcal{S})$ given by the monomial
$M$ in (\ref{eq:monom}) is simple, i.e. it has precisely seven adjacent edges.
These edges are the Newton segments of the seven binomials
\begin{equation}
\label{eq:finalfano}
\begin{matrix}
u_{000}^8 u_{001}^2 u_{010}^2 u_{100}^2 u_{111}^2  +  M, & 
u_{000}^8 u_{001}^2 u_{010}^2 u_{101}^2 u_{110}^2  +  M, & 
u_{000}^8 u_{001}^2 u_{011}^2 u_{100}^2 u_{110}^2  +  M, \\
u_{000}^8 u_{001}^2 u_{011}^2 u_{101}^2 u_{111}^2  +  M, & 
u_{000}^8 u_{010}^2 u_{011}^2 u_{100}^2 u_{101}^2  +  M, & 
u_{000}^8 u_{010}^2 u_{011}^2 u_{110}^2 u_{111}^2  +  M, \\
u_{000}^8 u_{100}^2 u_{101}^2 u_{110}^2 u_{111}^2  +  M. & & 
\end{matrix}
\end{equation}
Our computations revealed that these seven are all the
$\epsilon$-leading forms of $\mathcal{S}$
selected by generic matrices $T \in {\rm PD}_3$.
Hence the image of ${\rm trop}(\vartheta)$ consists of
the seven outer normal cones at these edges. This
is precisely what we had claimed in the first statement in  Proposition  \ref{prop:melody}.

The correspondence with the level structure described in \cite[\S 7.1]{Melody}
is seen as follows. The domains of linearity of ${\rm trop}(\vartheta)$
define a subdivision of ${\rm PD}_3$ into infinitely many convex polyhedral cones.
This subdivision is a coarsening of the {\em second Voronoi decomposition}.
We group the cones into seven classes, 
according to which binomial in (\ref{eq:finalfano}) gets selected by $T$.

The seven classes in ${\rm PD}_3$ are naturally labeled by the
seven lines in the Fano plane $\mathbb{P}^2(\mathbb{F}_2)$.
 These lines are given by the  three variables
missing in the respective leading monomials in (\ref{eq:finalfano}).
For instance, the first binomial in (\ref{eq:finalfano})
determines the line $\{(1{:}1{:}0),(1{:}0{:}1),(0{:}1{:}1)\}$
because $u_{110}, u_{101}, u_{011}$ are missing
in the monomial prior to $M$.
This subdivision of ${\rm PD}_3$ into seven classes
is precisely the level structure that was discovered
by Chan in \cite[\S 7.1]{Melody}.
\end{proof}

We expect an even more beautiful structure
when tropicalizing the re-embedding
$\mathcal{S}'$ of the Satake hypersurface $\mathcal{S}$ 
into $\PP^{35}$ given by the quadrics
$(T^\epsilon_{\epsilon'})^2$ in (\ref{addition}).
To see this, we revisit the combinatorial result
 in \cite[Theorem 3.1]{glass}. Glass classified the points
 in  $\mathcal{S}' \subset \PP^{35}$ according to which of
 the $36$ coordinates are zero at that point.
Up to symmetry, there are seven {\em Glass strata} in $\mathcal{S}'$.
Table~\ref{fig:glass} identifies the irreducible components
of these strata and their geometric meaning.
Glass\# is the number of coordinates
$(T^\epsilon_{\epsilon'})^2$ that are~zero.

\begin{table} 
\begin{tabular}{l|l|l|l|l}
Codim  & Glass\# & Geometric description & Components & How Many 
 \\
\hline
1 & 1 & Hyperelliptic locus & quadric $\cap \mathcal{S}$ & $36$\\
2 & 6 & Torelli boundary & $\PP^3 \times \PP^1$ & $336$\\
3 & 9 & Product of three elliptic curves & $\PP^1 \times \PP^1 \times \PP^1$ & $ 1120 $\\
3 & 16 & Satake boundary $\mathcal{A}_2(2,4)$ & $\PP^3 $ & $126$\\
4 & 18 & Torelli boundary of $\mathcal{A}_2(2,4)$  & quadric in $\PP^3 $\ & $1260$\\
5 & 24 & Satake stratum $\mathcal{A}_1(2,4)$ & $\PP^1 $ & $1260$\\
6 & 28 & Satake stratum $\mathcal{A}_0(2,4)$ & point & $1080$\\
\end{tabular}
\caption{Glass strata in the Satake hypersurface}
\label{fig:glass}
\end{table}

Table~\ref{fig:glass} refines our analysis of the Satake hypersurface
in Section~\ref{sec:satakehyper}. For each irreducible
component we know  their defining polynomials
in the unknowns $u_{000}, u_{001}, \ldots, u_{111}$.
These played an important role in this paper.
For instance, the components for Glass\# 16 are
the spaces $H_\epsilon^{\pm} \simeq \PP^3$ in (\ref{eq:fixedspace}),
and the components for Glass\# 6 are the
Segre varieties in (\ref{eq:2x4matrix}).

Now consider the tropical variety ${\rm trop}(\mathcal{S}')$ in $\mathbb{TP}^{35}$.
This is a modification of the tropical hypersurface in $\mathbb{TP}^{7}$ given
by ${\rm Newton}(\mathcal{S})$. It has much better properties
since the coordinate functions now have a geometric meaning:
they are the $36$ divisors of the hyperelliptic locus. These appear
as distinguished rays in  ${\rm trop}(\mathcal{S}')$.
By the principle of {\em geometric tropicalization} \cite[\S 2]{HKT},
they span a distinguished subfan of ${\rm trop}(\mathcal{S}')$.
Here, a collection of coordinate rays spans a cone
if and only if the corresponding intersection of divisors
appears in Table~\ref{fig:glass}.

The poset of Glass strata seems to embed naturally
into the poset of cones in the tropical moduli space  ${\rm trop}(\mathcal{A}_3)$.
Compare \cite[Theorem 3.1]{glass} with  \cite[Figure 8]{Melody}.
This deserves to be studied in more detail.
Can the table above be lifted to a
tropical level structure on $\mathcal{A}_3$?
The  finite groups of Section  \ref{sec:satakehyper}
act on the  fans we described. Taking quotients by 
these groups leads to objects known as {\em stacky fans}. 
The {\em tropical Torelli map} of \cite{Melody,  CMV}
is a morphism of
stacky fans from ${\rm trop}(\mathcal{M}_3)$ onto 
${\rm trop}(\mathcal{A}_3)$. 
It would be desirable to express the combinatorics
of the tropical Torelli map directly in terms of
the polynomials and ideals in this paper.

%A related question: if we lift the map (\ref{eq:64to1revisited}) 
%to $\mathcal{S}' \subset \mathbb{P}^{34}$, how does its
%tropicalization look like?

% NEW022413
In the first version of this article we asked the following  question:
What is the  relationship between the G\"opel variety $\mathcal{G}$
and the moduli space $Y(\Delta)$ of degree $2$ del Pezzo surfaces constructed
by Hacking, Keel, and Tevelev in  \cite{HKT}?
The latter is a tropical compactification
of the  configuration space of seven points in $\PP^2$.
In the meantime, we were able answer this question.
It turns out that ${\rm trop}(\mathcal{G})$ has the same support as
their fan $\mathcal{F}(\Delta)$ in \cite[Corollary 5.3]{HKT}. 
This is explained in \cite{RSS}. 
That paper features calculations with explicit moduli spaces, mostly
in genus $2$, that are considerably smaller than the ones we studied here. Our
readers may enjoy looking at  \cite{RSS} as a point of entry also to the present~work.
% NEW022413

\bigskip

\noindent
\footnotesize {\bf Authors' address}:
Department of Mathematics, University of California, Berkeley, CA 94720, USA. \hfill
{\tt \{qingchun,svs,guss,bernd\}@math.berkeley.edu}

\begin{thebibliography}{DHBHS}

\footnotesize
%\small
\setlength{\itemsep}{-1mm}

\bibitem[AK]{AK}
F.~Ardila and C.~Klivans:
{\em The Bergman complex of a matroid and phylogenetic trees},
J.~Combin.~Theory Ser. B {\bf 96} (2006) 38--49, \arxiv{math/0311370v2}.

\bibitem[ARW]{ARW}
F.~Ardila, V.~Reiner and L.~Williams:
{\em Bergman complexes, Coxeter arrangements, and graph associahedra},
S\'em.~Lothar.~Combin {\bf 54A} (2005/07), Art. B54A, \arxiv{math/0508240v2}.

\bibitem[BI]{BI}
H.~Barcelo and E.~Ihrig:
{\em Lattices of parabolic subgroups in connection with hyperplane arrangements},
J. Algebraic Combin. {\bf 9} (1999) 5--24.

\bibitem[Bea03]{beauville} A.~Beauville: {\em The Coble hypersurfaces}, C. R. Math. Acad. Sci. Paris {\bf 337} (2003), no.~3, 189--194, \arxiv{math/0306097v1}.

\bibitem[Bea06]{Beau}
A.~Beauville:  {\em Vector bundles on curves and theta functions},
in ``Moduli Spaces and Arithmetic Geometry'', 145--156, Adv. Stud. Pure Math.,
 45, Math. Soc. Japan, Tokyo, 2006, \arxiv{math/0502179v2}.
 
\bibitem[BL]{BL}
C.~Birkenhake and  H.~Lange:
{\em Complex Abelian Varieties}, Second Edition,
Grundlehren {\bf 302}, Springer Verlag, Berlin, 2004.

\bibitem[Bou]{bourbaki} N.~Bourbaki: {\it \'El\'ements de math\'ematique. Fasc. XXXIV. Groupes et alg\`ebres de Lie. Chapitre IV: Groupes de Coxeter et systèmes de Tits. Chapitre V: Groupes engendr\'es par des r\'eflexions. Chapitre VI: syst\`emes de racines}, Actualit\'es Scientifiques et Industrielles, No. 1337 Hermann, Paris 1968.

\bibitem[Bra]{Bra} C.~Bramble: {\em A collineation group isomorphic with the group
of double tangents of the plane quartic}, Amer. J. Math.
{\bf 40} (1918) 351--365.

\bibitem[Cay]{Cay}
A.~Cayley: {\em Algorithm for the characteristics of the
triple $\vartheta$-functions}, J. Reine Angew. Mathematik {\bf 87} (1879) 165--169.
 
 \bibitem[Cha]{Melody}
M.~Chan: {\em Combinatorics of the tropical Torelli map},
Algebra Number Theory {\bf 6} (2012) 1133--1169,
 \arxiv{1012.4539v2}.

\bibitem[CMV]{CMV}
M.~Chan, M.~Melo and F.~Viviani:
{\em Tropical Teichm\"uller and Siegel spaces},
 in ``Tropical Geometry'', Proceedings Castro Urdiales 2011, (editors E. Brugall\'e, M.A. Cueto, A. Dickenstein, E.M. Feichtner and I. Itenberg), American Mathematical Society, Contemporary Mathematics,
\arxiv{1207.2443v1}.

\bibitem[Cob]{Cob} A.~Coble:
{\em Algebraic Geometry and Theta Functions},
American Math.~Society, 1929.

\bibitem[CGL]{CGL} E.~Colombo, B.~van Geemen and E.~Looijenga:
{\em Del Pezzo moduli via root systems},
 Algebra, arithmetic, and geometry: in honor of Yu. I. Manin. 
 Vol. I, 291--337, Progr. Math. {\bf 269}, Birkh\"auser, Boston, 2009,
\arxiv{math.AG/0702442v1}.

\bibitem[CLS]{CLS} D.~Cox, J.~Little and H.~Schenck:
{\em Toric Varieties},
Graduate Studies in Math.~{\bf 124},
American Math.~Society, 2011.

\bibitem[DG]{projectiverep} F.~Dalla Piazza and B.~van Geemen:
{\em Siegel modular forms and finite symplectic groups},
 Adv.~Theor.~Math.~Phys. {\bf 13} (2009) 1771--1814,
 \arxiv{0804.3769v2}.
 
 \bibitem[DPSM]{dpsm} F.~ Dalla Piazza and R.~ Salvati Manni:
{\em On the Coble quartic and Fourier-Jacobi expansion of theta relations},
\arxiv{1304.7659}.
 
\bibitem[DHBHS]{DHBHS}
B.~Deconinck, M.~Heil, A.~Bobenko, 
M.~van Hoeij, and M.~Schmies: 
{\em Computing Riemann theta functions},
Math. Comp. {\bf 73} (2004) 1417--1442, \arxiv{nlin/0206009v2}.

\bibitem[DFS]{DFS}
A.~Dickenstein,  E.~Feichtner and B.~Sturmfels:
{\em Tropical discriminants},
J. Amer. Math. Soc. {\bf 20} (2007) 1111--1133, \arxiv{math/0510126v3}.

\bibitem[Dol]{dolgachev:invt} I.~Dolgachev: {\it Lectures on Invariant Theory},
London Mathematical Society, Lecture Notes Series {\bf 296}, Cambridge, 2003.

\bibitem[DO]{DO} I.~Dolgachev and D.~Ortland:
{\em Point Sets in Projective Spaces and Theta Functions},
  Ast\'erisque {\bf 165} (1988).

\bibitem[Eis95]{eisenbud} D.~Eisenbud: {\it Commutative Algebra with a View Toward Algebraic Geometry}, Graduate Texts in Mathematics {\bf 150}, Springer-Verlag, 1995.

\bibitem[Eis05]{geom_syz} D.~Eisenbud: {\it The Geometry of Syzygies}, Graduate Texts in Mathematics {\bf 229}, Springer-Verlag, 2005.

\bibitem[FS]{FS} E.~Freitag and R.~Salvati Manni:
{\em The modular variety of hyperelliptic curves of genus three},
 Trans.~Amer.~Math.~Soc. {\bf 363} (2011) 281--312, \arxiv{0710.5920v1}.
  
\bibitem[Gap]{gap} The GAP Group, GAP -- Groups, Algorithms, and Programming, Version 4.4.12; 2008. (\url{http://www.gap-system.org})

\bibitem[GP]{GP}
M.~Geck and G.~Pfeiffer: {\em Characters of Finite Coxeter groups and Iwahori-Hecke Algebras},
 London Mathematical Society Monographs, New Series, {\bf 21},
  Oxford University Press, New York, 2000.
 
\bibitem[Gee]{vangeemen} B.~van Geemen: {\em Siegel modular forms vanishing on the moduli space of curves}, Invent.~Math.~{\bf 78} (1984) 329--349.

\bibitem[GG]{GG} B.~van Geemen and G.~van der Geer:
{\em Kummer varieties and the moduli spaces of abelian varieties},
 Amer.~J.~Math.~{\bf 108} (1986) 615--641. 

\bibitem[Gla]{glass} J.~P.~Glass: {\em Theta constants of genus three},
 Compositio Math.~{\bf 40} (1980) 123--137.

\bibitem[M2]{m2} D.~Grayson and M.~Stillman:
{\em Macaulay 2, a software system for research in algebraic geometry}, 
Available at \url{http://www.math.uiuc.edu/Macaulay2/}.

\bibitem[GS1]{GS}
S.~Grushevsky and R.~Salvati Manni:
{\em The Scorza correspondence in genus $3$}, Manuscripta Math.~{\bf 141} (2013), no.~1-2, 111--124, \arxiv{1009.0375v2}.

% NEW022413
\bibitem[GS2]{gsnew}
S.~Grushevsky and R.~Salvati Manni:
{\em On the Coble quartic},
\arxiv{1212.1895v1}.

\bibitem[GSW]{gsw} L.~Gruson, S.~Sam, and J.~Weyman:
{\em Moduli of Abelian varieties, Vinberg $\theta$-groups, and free resolutions}, Commutative Algebra (edited by Irena Peeva), 419--469, Springer, 2013, \arxiv{1203.2575v2}.

\bibitem[HKT]{HKT}
P.~Hacking, S.~Keel and J.~Tevelev:
{\em Stable pair, tropical, and log canonical compactifications of moduli spaces of del Pezzo surfaces},
Invent.~Math. {\bf 178} (2009) 173--227, \arxiv{math/0702505v2}.

\bibitem[HJJS]{HJJS} S.~Hermann, A.~Jensen, M.~Joswig and B.~Sturmfels:
{\em How to draw tropical planes},
 Electron. J. Combin. {\bf 16} (2009), no. 2,  Anders Bj\"orner volume, Research Paper 6, \arxiv{0808.2383v4}.

\bibitem[HMSV]{HMSV} B.~Howard, J.~Millson, A.~Snowden and R.~Vakil:
{\em The ideal of relations for the ring of invariants of $n$ points on the line},
J.~Eur.~Math.~Soc.~(JEMS)~{\bf 14} (2012) 1--60, \arxiv{0909.3230v1}.

\bibitem[Hud]{Hu} R.~Hudson: {\em Kummer's Quartic Surface},
Cambridge University Press, 1905.

\bibitem[Gfan]{gfan}
A.~Jensen:
{\em {G}fan, a software system for {G}r{\"o}bner fans and tropical varieties},
 Available at \url{http://home.imf.au.dk/jensen/software/gfan/gfan.html}.
     
\bibitem[Kha]{khaled} A.~Khaled: {\em Projective normality and equations of Kummer varieties},
 J.~Reine Angew.~Math. {\bf 465} (1995) 197--217.

\bibitem[Kim]{kimura} T.~Kimura: {\em Remark on some combinatorial construction 
of relative invariants}, Tsukuba J. Math. {\bf 5} (1981) 101--115.

\bibitem[Kon]{Kon} S.~Kondo: 
{\em Moduli of plane quartics, G\"opel invariants and Borcherds products},
Int. Math. Res. Not. IMRN (2011), no. 12, 2825--2860, \arxiv{0906.2598v1}.

\bibitem[Mac]{macdonald} I.~G.~Macdonald:
{\em Some irreducible representations of Weyl groups},
 Bull.~London Math.~Soc.~{\bf 4} (1972) 148--150.

\bibitem[Man]{manivel} L.~Manivel:
{\em Configurations of lines and models of Lie algebras},
 J. Algebra {\bf 304} (2006) 457--486, \arxiv{math/0507118v1}.

\bibitem[MZ]{MZ} G.~Mikhalkin and I.~Zharkov:
{\em Tropical curves, their Jacobians and theta functions},
 Curves and abelian varieties, 203--230, 
Contemporary Math., {\bf 465}, Amer.~Math.~Soc., Providence, RI, 2008, 	\arxiv{math/0612267v2}.

\bibitem[Mul]{Muller} J.~S.~M\"uller: {\em Computing Canonical Heights on Jacobians},
 Dissertation, Universit\"at Bayreuth, 2010,
\url{http://www.math.uni-hamburg.de/home/js.mueller/mueller-thesis.pdf}

\bibitem[Mum]{Mum}  D.~Mumford:
{\em On the equations defining abelian varieties. I,}
Invent.~Math.~{\bf 1} (1966) 287--354.

\bibitem[OrSo]{OrSo}
P.~Orlik and L.~Solomon:
{\em Unitary reflection groups and cohomology}, 
Invent.~Math.~{\bf 59} (1980) 77--94.

\bibitem[OtSe]{OS}
G.~Ottaviani and E.~Sernesi:
{\em  On the hypersurface of L\"uroth quartics},
 Michigan Math. J. {\bf 59} (2010) 365--394, \arxiv{0903.5149v2}.

\bibitem[Pau]{Pau} C.~Pauly:
{\em Self-duality of Coble's quartic hypersurface and applications},
Michigan Math.~J.~{\bf 50} (2002) 551--574, \arxiv{math/0109218v1}.

\bibitem[RSS]{RSS} Q.~Ren, S.~Sam and B.~Sturmfels:
{\em Tropicalization of classical moduli spaces}, \arxiv{1303.1132v1}.

\bibitem[Run]{runge} B. Runge: {\em On Siegel modular forms. I.},
J. Reine Angew. Math. {\bf 436} (1993), 57--85. 

\bibitem[Sal]{SM} R.~Salvati Manni:
{\em Modular varieties with level $2$ theta structure},
Amer. J.~Math.~{\bf 116} (1994) 1489--1511.

\bibitem[SW]{SW} D.~Speyer and L.~Williams:
{\em The tropical totally positive Grassmannian}, 
J. Algebraic Combin. {\bf 22} (2005) 189--210, \arxiv{math/0312297v1}.

\bibitem[Sta]{stanley} R.~Stanley:
{\em Hilbert functions of graded algebras},
 Advances in Math.~{\bf 28} (1978) 57--83.

\bibitem[Sage]{sage} W.~A. Stein et al.: {\em  Sage Mathematics Software} (Version 5.1), The Sage Development Team, 2012, \url{http://www.sagemath.org}.

\bibitem[SD]{SD} C.~Swierczewski and B.~Deconinck:
{\em Computing Riemann theta functions in Sage with applications}, Math. Comput. Simulation, to appear.

\bibitem[Tsu]{tsuyumine} S.~Tsuyumine: {\em Factorial property of a ring of automorphic forms}, Trans.~Amer.~Math.~Soc.~{\bf 296} (1986) 111--123. 

\bibitem[Vin]{vinberg} \`E.~B. Vinberg: {\em The Weyl group of a graded Lie algebra}, Math. USSR-Izv.~{\bf 10} (1976) 463--495.  

\bibitem[Wir]{wirt}  W.~Wirtinger: {\em Untersuchungen \"uber Thetafunktionen},
Teubner Verlag, Leipzig, 1895.


\end{thebibliography}
\end{document}